\documentclass[review]{elsarticle}
\usepackage{amsmath}
\usepackage{amsthm}
\usepackage{amssymb}
\usepackage{amscd}
\usepackage{lineno,hyperref}
\modulolinenumbers[5]

\begin{document}
\newtheorem{definition}{Definition}[section]
\newtheorem{example}[definition]{Example}
\newtheorem{algorithm}[definition]{Algorithm}
\newtheorem{theorem}[definition]{Theorem}

\newtheorem{axiom}[definition]{Axiom}
\newtheorem{prop}[definition]{Property}
\newtheorem{proposition}[definition]{Proposition}
\newtheorem{lemma}[definition]{Lemma}
\newtheorem{corollary}[definition]{Corollary}
\newtheorem{remark}[definition]{Remark}
\newtheorem{condition}[definition]{Condition}
\newtheorem{conclusion}[definition]{Conclusion}
\newtheorem{assumption}[definition]{Assumption}
\newtheorem{problem}[definition]{Problem}

\makeatletter
\@addtoreset{equation}{section}
\makeatother
\renewcommand{\theequation}{\arabic{section}.\arabic{equation}}

\begin{frontmatter}

\title{Two fixed point theorems in complete random normed modules and their applications to backward stochastic equations}

\author[mymainaddress]{Tiexin Guo\fnref{myfootnote} \corref{mycorrespondingauthor}}
\cortext[mycorrespondingauthor]{Corresponding author}
\ead{tiexinguo@csu.edu.cn}

\fntext[myfootnote]{This work was supported by National Natural Science Foundation of China (Grant No. 11571369).}

\author[mymainaddress]{Erxin Zhang\fnref{myfootnote}}
\ead{zhangerxin6666@163.com}

\author[mymainaddress]{Yachao Wang\fnref{myfootnote}}
\ead{wychao@csu.edu.cn}

\author[mysecondaryaddress]{Zichen Guo}
\ead{4403116@gmail.com}

\address[mymainaddress]{School of Mathematics and Statistics, Central South University, ChangSha 410083, China}
\address[mysecondaryaddress]{Department of Mathematics and Statistics, Dalhousie University, {\rm6316} Coburg Road, \\ PO BOX {\rm15000}, Halifax, Nova Scotia, Canada B3H 4R2.}

\begin{abstract}
 Let $(\Omega,\mathcal{F},\mathbb{F},P)$ be a filtered probability space with a filtration $\mathbb{F} = (\mathcal{F}_t)_{t \in [0,T]}$ satisfying the usual conditions and $T$ a finite time, $L^0(\mathcal{F}_0)$ the algebra of equivalence classes of $\mathcal{F}_0$--measurable real--valued random variables on $\Omega$, $L^p(\mathcal{F}_T, R^d)$ the Banach space of equivalence classes of $p$--integrable (for $1 \leq p< +\infty$) or essentially bounded (for $p= +\infty$) $\mathcal{F}_T$--measurable $R^d$--valued functions on $\Omega$ and $L^p_{\mathcal{F}_0}(\mathcal{F}_T, R^d)$ the $L^0(\mathcal{F}_0)$--module generated by $L^p(\mathcal{F}_T, R^d)$, namely $L^p_{\mathcal{F}_0}(\mathcal{F}_T, R^d) \\
  = \{ \xi \cdot x : \xi \in L^0(\mathcal{F}_0)$ and $x \in L^p(\mathcal{F}_T, R^d) \}$. The usual backward stochastic equations $(BSEs)$ are studied for their terminal conditions $\xi$ that are required to belong to $L^p(\mathcal{F}_T, R^d)$. Motivated by the study of dynamic risk measures, this paper studies $BSEs$ with their terminal conditions $\xi$ in $L^p_{\mathcal{F}_0}(\mathcal{F}_T, R^d)$. To look for the spaces where solutions of the more general class of $BSEs$ lie, we construct various kinds of complete random normed modules ($RN$ modules) consisting of $RCLL$ (right continuous with left limit) adapted processes by means of the generalized conditional mathematical expectation. Since $L^p_{\mathcal{F}_0}(\mathcal{F}_T, R^d)$ is a complete $RN$ module, this paper first proves two fixed point theorems in complete random normed modules, which are respectively the random generalizations of the classical Banach's contraction mapping principle and Browder--Kirk's fixed point theorem. As applications, the first is used to give the existence and uniqueness of solutions to various kinds of backward stochastic equations under $L^0$--Lipschitz assumptions and the second is used to establish the existence of solutions to backward stochastic equations of nonexpansive type. Finally, since $L^p_{\mathcal{F}_0}(\mathcal{F}_T, R^d)$ can be represented as the countable concatenation hull of $L^p(\mathcal{F}_T, R^d)$, this paper concludes with a brief discussion of the relation between the solutions of $BSEs$ with a terminal condition $\xi$ in $L^p(\mathcal{F}_T, R^d)$ and $L^p_{\mathcal{F}_0}(\mathcal{F}_T, R^d)$.
\end{abstract}


\end{frontmatter}

\section{Introduction}
The study of backward stochastic equations $(BSEs)$ is a fundamental topic in stochastic analysis, and in particular the interest in $BSEs$ also comes from their connections with different mathematical fields, such as mathematical finance, partial differential equations and stochastic control, see, e.g., \cite{EH97,BK00,Peng04,BLP09,CE02,CN17,DI10a,DI10b,Kob00,LLQ11,PP90,P99,PY09,SG18,TL94,MT03}.
\par
Throughout this paper, whenever we discuss $BSEs$ we always let $(\Omega,\mathcal{F},\mathbb{F},P)$ be a filtered probability space with a filtration $\mathbb{F} = (\mathcal{F}_t)_{t \in [0,T]}$ satisfying the usual conditions and $T$ a finite time, $L^0(\mathcal{F}_0)$ the algebra of equivalence classes of $\mathcal{F}_0$--measurable real--valued random variables on $\Omega$, $L^p(\mathcal{F}_T, R^d)$ the Banach space of equivalence classes of $p$--integrable (for $1 \leq p< +\infty$) or essentially bounded (for $p= +\infty$) $\mathcal{F}_T$--measurable $R^d$--valued functions on $\Omega$ and $L^p_{\mathcal{F}_0}(\mathcal{F}_T, \\
 R^d)$ the $L^0(\mathcal{F}_0)$--module generated by $L^p(\mathcal{F}_T, R^d)$, namely $L^p_{\mathcal{F}_0}(\mathcal{F}_T, R^d) = \{ \xi \cdot x : \xi \in L^0(\mathcal{F}_0)$ and $x \in L^p(\mathcal{F}_T, R^d) \}$. The usual backward stochastic equations $(BSEs)$ are studied for their terminal conditions $\xi$ that are required to belong to $L^p(\mathcal{F}_T, R^d)$. Motivated by the study of dynamic risk measures, this paper studies $BSEs$ with their terminal conditions $\xi$ in $L^p_{\mathcal{F}_0}(\mathcal{F}_T, R^d)$.
\par
It is well known that under some proper conditions the following classical $BSDE$:\\
\begin{equation}\nonumber
  \left\{
            \begin{array}{ll}
              -dY_t = g(t,Z_t)dt - Z_tdW_t \\

              Y_T = \xi
            \end{array}
          \right.
\end{equation}
has a unique solution $(Y,Z)$ for any given terminal condition $\xi \in L^2(\mathcal{F}_T)$, denote $Y_t$ by $\mathcal{E}^g[\xi ~|~ \mathcal{F}_t]$. Under the further additional conditions, the conditional $g$--expectation $\mathcal{E}^g[\cdot ~|~ \mathcal{F}_t]$ can induce a family of conditional convex risk measures $\rho^g_t : L^2(\mathcal{F}_T) \to L^2(\mathcal{F}_t)$ given by $\rho^g_t(x) = \mathcal{E}^g[-x~|~\mathcal{F}_t]$ for any $x \in L^2(\mathcal{F}_T)$, in particular $\{ \rho^g_t \}_{t \in [0 ,T]}$ forms a time--consistent dynamic risk measure, see \cite{PP90,GZZ14,Peng04} and the references therein for the related details, where Guo, et.al pointed out that each $\rho^g_t$ can extend to a conditional convex risk measure $\bar{\rho}^g_t : L^2_{\mathcal{F}_t}(\mathcal{F}_T) \to L^0(\mathcal{F}_t)$. It is obvious that the extension is essentially equivalent to expanding a terminal condition in $L^2(\mathcal{F}_T)$ to one in $L^2_{\mathcal{F}_t}(\mathcal{F}_T)$, we naturally hope to know the essence buried behind the expanding of terminal conditions. We conjuncture that this extension $\bar{\rho}^g_t$ can be also directly obtained from solving the above--stated $BSDE$ with the terminal condition $\xi$ in $L^2_{\mathcal{F}_t}(\mathcal{F}_T)$. For this, we consider the $BSEs$ of the form (\ref{1.3}) below.
\par
To look for the tool for the study of the more general class of $BSEs$, let us briefly introduce the recent enlightening work from \cite{CN17}. By $|\cdot|$, we denote the Euclidean norm on $R^d$, and for a $d$--dimensional $\mathcal{F}$--measurable random vector $X$ , we define $\|X\|_p := (E|X|^p)^{1/p}$ if $0<p<+\infty$ and $\|X\|_{\infty} := ess \sup_{\omega \in \Omega} |X(\omega)|$.
\par
For $p \in (1,+\infty]$, we set:
\par
$L^p(\mathcal{F}_t)^d$ : the Banach space of equivalence classes of $d$--dimensional $\mathcal{F}_t$--measurable random vectors $X$ satisfying $\|X\|_p < +\infty$, equipped with the norm $\|\cdot\|_p$.
\par
$E_t(X) := E[X~|~\mathcal{F}_t]$ denotes the conditional mathematical expectation of $X$ with respect to $\mathcal{F}_t$ for $X \in L^1(\mathcal{F})^d$, where $L^1(\mathcal{F})^d$ is the Banach space of equivalence classes of $d$--dimensional integrable $\mathcal{F}$--measurable random vectors on $(\Omega,\mathcal{F},P)$.
\par
$S^p :$ the Banach space of equivalence classes of $R^d$--valued $RCLL$ (right continuous with left limit) adapted processes $(Y_t)_{t \in [0,T]}$ satisfying $\|Y\|_p := \|\sup_{0 \leq t \leq T} |Y_t| \|_p < +\infty$, equipped with the norm $\|\cdot\|_p$, $S^p_0 := \{ Y \in S^p : Y_0 =0 \}$.
\par
$M^p := \{ Y \in S^p : Y$ is a martingale $\}$, $M^p_0 := \{ Y \in M^p : Y_0 =0 \}$.
\par
As usual, two random processes $X$ and $Y$ are equivalent iff they are stochastically equivalent, namely $X_t = Y_t$ a.s., $\forall t \in [0,T]$.
\par
Recently, in \cite{CN17} Cheridito and Nam studied backward stochastic equations $(BSEs)$ of the following form
\begin{equation}\label{1.1}
  Y_t + F_t(Y,M) + M_t = \xi + F_T(Y,M) + M_T, \forall t \in [0,T].
\end{equation}
Here, $F : S^p \times M^p_0 \to S^p_0$ is the generator and $\xi \in L^p(\mathcal{F}_T)^d$ is the terminal condition. When $F_t(Y,M) = \int^t_0 f(s,Y,M)ds, \forall t \in [0,T]$, for some driver $f$, (\ref{1.1}) includes various kinds of back stochastic differential equations $(BSDEs)$, see \cite{CN17} and its references.
\par
$F$ is said to satisfy condition $(S)$ if for all $y \in L^p(\mathcal{F}_0)^d$ and $M \in M^p_0$, the stochastic equation:
\begin{equation}\label{1.2}
  Y_t = y-F_t(Y,M) - M_t, \forall t \in [0,T].
\end{equation}
has a unique solution $Y \in S^p$.
\par
Let $F$ satisfy condition $(S)$. For a given $V \in L^p(\mathcal{F}_T)^d$, then $y^V = E_0(V) \in L^p(\mathcal{F}_0)^d$ and $M^V_t := E_0(V) - E_t(V)$ is in $M^p_0$, we denote by $Y^V$ the solution of (\ref{1.2}). $(F,\xi)$ defines a mapping $G : L^p(\mathcal{F}_T)^d \to L^p(\mathcal{F}_T)^d$ by $G(V) = \xi + F_T(Y^V,M^V)$ for any $V \in L^p(\mathcal{F}_T)^d$. Denote by $FPS(G)$ the set of fixed points of $G$ and by $SOL(F,\xi)$ the set of solutions of (\ref{1.1}), then the main and excellent idea of \cite{CN17} is to prove that $\varphi : L^p(\mathcal{F}_T)^d \to S^p \times M^p_0$ defined by $\varphi (V) := (Y^V,M^V)$, is a bijection from $FPS(G)$ onto $SOL(F,\xi)$ when $\varphi$ is limited to $FPS(G)$, at this time whose inverse is $\pi : SOL(F,\xi) \to FPS(G)$ given by $\pi (Y,M) = Y_0 - M_T$ for any $(Y,M) \in SOL(F,\xi)$. In other words, Cheridito and Nam \cite{CN17} shows that a problem of solving $BSE$(\ref{1.1}) is equivalent to one of studying the fixed points of $G$. Thus the known famous fixed point theorems in functional analysis can be used to study the $BSE$(\ref{1.1}). For example, in \cite{CN17} Banach's contraction mapping principle is used for the uniqueness and existence and Krasnoselskii's fixed point theorem \cite{Kra64} for the existence of solutions of the $BSE$(\ref{1.1}). Since Krasnoselskii's fixed point theorem is a combination of Banach's and Schauder's fixed point theorem, its applications requires compactness as the premise.
\par
In this paper we will also emphasize applications of the famous Browder--Kirk's fixed point theorem \cite{B65,GR84,G65,K65,K83}: let $(B,\|\cdot\|)$ be a Banach space, $D$ a weakly compact convex subset with normal structure and $T : D \to D$ a nonexpansive mapping, then $T$ has a fixed point. Let us recall : a closed convex subset $H$ of $(B,\|\cdot\|)$ is said to have normal structure if for each closed and bounded convex subset $C$ of $H$ with $D(C) := \sup \{ \|x-y\| : x,y \in C \} >0$ there is always  a point $x$ in $C$ such that $\sup \{ \|x-y\| : y \in C \} < D(C)$. This emphasis comes from an observation, see Remark3.2 of \cite{CN17} or Example3.1 of \cite{DI10a} : let $F_t(Y,M) = atY_0$ for a constant $a$, then when $aT =1$ and $E_0(\xi) = 0$ the $BSE$ : $Y_t - a(T-t)Y_0 = \xi + M_T -M_t$ has infinitely many solutions $(Y,M) : Y_t = (1- t/T)Y_0 + E_t(\xi)$ and $M_t = - E_t(\xi)$ for the different choices of $\xi$ and $Y_0$. This example was presented in \cite{CN17} as a counterexample for the existence and uniqueness for $BSEs$ with path--dependent coefficients without the assumption that the Lipschitz constant in question is sufficiently small and in \cite{DI10a} as a counterexample showing that time--delayed $BSDEs$ with Lipschitz coefficients are not always well--posed. However, Browder--Kirk's fixed point theorem can give another natural explanation for the counterexample : $G : L^p(\mathcal{F}_T)^d \to L^p(\mathcal{F}_T)^d$ defined by $G(V) = \xi + F_T(Y^V,M^V) = \xi + aTY_0^V = \xi + E_0(V)$ for any $V \in L^p(\mathcal{F}_T)^d$, is clearly a nonexpansive mapping, which maps every closed ball $H := B_r(\xi)$ with center $\xi$ and radius $r$ into itself, and hence when $1<p<+\infty,  G$ and $B_r(\xi)$ satisfy Browder--Kirk's fixed point theorem by noticing that $L^p(\mathcal{F}_T)^d$ is a uniformly convex Banach space so that $B_r(\xi)$ is a weakly compact convex subset with normal structure. Clearly, $\xi + Y_0$ is always a fixed point of $G$ for any $Y_0 \in L^p(\mathcal{F}_0)^d$, and $\{ \xi + Y_0~|~Y_0 \in L^p(\mathcal{F}_0)^d \}$ correspond in one to one way to the infinitely many solutions mentioned formerly. In fact, the nice property that $L^p(\mathcal{F}_T)^d$ are always uniformly convex when $1<p<+\infty$, which ensures that every bounded closed convex subset of $L^p(\mathcal{F}_T)^d$ is a weakly compact set with normal structure, brings much convenience to applications of Browder--Kirk's fixed point theorem to the study of $BSDEs$, not as Schauder's or Krasnoselskii's fixed point theorem always requires bounded closed convex sets in question to be compact. This paper provides more applications of Browder--Kirk's fixed point theorem to $BSDEs$ in Section\ref{section6} of this paper.
\par
To look for the spaces where the solutions of $BSEs$ (\ref{1.3}) below lie, let us first recall the notion of a generalized conditional mathematical expectation : for a given probability space $(\Omega,\mathcal{F},P)$ and a sub $\sigma$--algebra $\mathcal{E}$ of $\mathcal{F}$, a real--valued $\mathcal{F}$--measurable random variable $X$ is said to be $\sigma$--integrable or conditionally integrable with respect to $\mathcal{E}$ if there is an nondecreasing sequence $\{ A_n, n \in N \}$ in $\mathcal{E}$ with $\cup_{n \geq1}A_n = \Omega$ such that $E(|X|I_{A_n}) := \int_{A_n} |X|dp < +\infty$ for any $n \in N$, at which time there is an almost surely $($a.s.$)$ unique real--valued $\mathcal{E}$--measurable random variable $Y$ such that $E(XI_{A}) = E(YI_A)$ for each $A \in \mathcal{E}$ with $E(|X|I_A) < +\infty$. $Y$ is called the generalized conditional expectation of $X$ with respect to $\mathcal{E}$, denoted by $E[X~|~\mathcal{E}]$, where $I_A$ stands for the characteristic function on $A$ and $N$ for the set of positive integers. For a random vector $X$, $E[X~|~\mathcal{E}]$ can be understood in a componentwise way.
\par
From now on, $E[\cdot~|~\mathcal{E}]$ always denotes the kind of generalized expectation unless otherwise stated. It is well known, see, e.g., \cite{HWY92}, that $X$ is conditionally integrable with respect to $\mathcal{E}$ if and only if there is an a.s. positive $\mathcal{E}$--measurable random variable $\eta$ such that $X\eta$ is integrable. An $\mathbb{F}$--adapted random process $(X_t)_{t \in [0,T]}$ is called a generalized martingale if for any $0 \leq s <t \leq T,  X_t$ is conditionally integrable with respect to $\mathcal{F}_s$ and $E[X_t~|~\mathcal{F}_s] = X_s (a.s.)$.
\par
Let $S$ be the set of equivalence classes of $R^d$--valued $RCLL$ $\mathbb{F}$--adapted processes, then for $1 \leq p<+\infty$, let $S^p_{\mathcal{F}_0} := \{ Y = (Y_t)_{t \in [0,T]} \in S~|~(\sup_{t \in [0,T]} |Y_t|)^p$ is conditionally integrable with respect to $\mathcal{F}_0 \}$, \\ $S^p_{\mathcal{F}_0,0} := \{ Y \in S^p_{\mathcal{F}_0}~|~Y_0 = 0 \}$, \\
$M:= \{ Y \in S ~|~ Y$ is a generalized martingale \}, \\
$M^p_{\mathcal{F}_0} := M \cap S^p_{\mathcal{F}_0}$,  \\
$M^p_{\mathcal{F}_0,0} := M \cap S^p_{\mathcal{F}_0,0}$.
\par
As usual, for any given sub $\sigma$--algebra $\mathcal{E}$ of $\mathcal{F}$, let $L^0(\mathcal{E},R^d)$ be the linear space of equivalence classes of $R^d$--valued $\mathcal{E}$--measurable random vectors and $L^0(\mathcal{E}) := L^0(\mathcal{E},R^1)$. Then $L^0(\mathcal{E})$ is an algebra over the real number field $R^1$ and $L^0(\mathcal{E},R^d)$ is a free $L^0(\mathcal{E})$--module of rank $d$. Further, for a filtration $\mathbb{F} := (\mathcal{F}_t)_{t \in [0,T]}$ as above and $1 \leq p<+\infty$, it is easy to see that $L^p_{\mathcal{F}_0}(\mathcal{F}_T,R^d) = \{ V \in L^0(\mathcal{F}_T, R^d)~|~ |V|^p$ is conditionally integrable with respect to $\mathcal{F}_0 \}$, which is also called the conditional $L^p$--space under the conditional $p$--norm $|||V|||_p := (E[|V|^p~|~\mathcal{F}_0])^{1/p}$ in the literature of mathematical finance, see, e.g., \cite{HR87,FKV09}. In fact, $L^p_{\mathcal{F}_0}(\mathcal{F}_T,R^d)$ has served as a model space for dynamic asset pricing and risk measures \cite{HR87,GZZ14,FKV09}. It is easy to see that $(L^p_{\mathcal{F}_0}(\mathcal{F}_T,R^d),|||\cdot|||_p)$ are a special class of random normed modules (briefly, $RN$ modules), which were independently introduced by Guo in \cite{Guo89,Guo92,Guo93,YZG91} and by Haydon, Levy and Raynaud in \cite{HLR91} and deeply and systematically developed by Guo, et.al, in \cite{Guo95,Guo96a,Guo96b,Guo97, Guo99,Guo07,Guo08,Guo10,Guo13,GL05,GY96,GZ10}.
\par
Similarly, $S^p_{\mathcal{F}_0}$ is the $L^0(\mathcal{F}_0)$--module generated by the Banach space $S^p$, namely, $S^p_{\mathcal{F}_0} = L^0(\mathcal{F}_0) \cdot S^p = \{ \xi \cdot Y~|~ \xi \in L^0(\mathcal{F}_0)$ and $Y \in S^p \}$, where $(\xi \cdot Y)(\omega,t) = \xi(\omega)\cdot Y_t(\omega)$ for any $(\omega,t) \in \Omega \times [0,T]$. Further, $S^p_{\mathcal{F}_0}$ becomes an $RN$ module endowed with the $L^0$--norm $|||Y|||_p := (E[(\sup_{t \in [0,T]} |Y_t|^p)~|~\mathcal{F}_0])^{1/p}$ for any $Y := (Y_t)_{t \in [0,T]} \in S^p_{\mathcal{F}_0}$. Similarly, $M^p_{\mathcal{F}_0} = L^0(\mathcal{F}_0) \cdot M^p$ is also an $RN$ module.
\par
The central purpose of this paper is to consider the $BSEs$ of the form:
\begin{equation}\label{1.3}
  Y_t + F_t(Y,M) + M_t = \xi + F_T(Y,M) + M_T
\end{equation}
for the generator $F : S^p_{\mathcal{F}_0} \times M^p_{\mathcal{F}_0,0} \to S^p_{\mathcal{F}_0,0}$ and the terminal condition $\xi \in L^p_{\mathcal{F}_0}(\mathcal{F}_T,R^d)$.
\par
Similarly, $F$ is said to satisfy condition $(S)$ if for any $y \in L^0(\mathcal{F}_0,R^d)$ and any $M \in M^p_{\mathcal{F}_0,0}$ the stochastic equation $Y_t = y - F_t(Y,M) - M_t$ has a unique solution $Y \in S^p_{\mathcal{F}_0}$.
\par
Let $V$ belong to $L^p_{\mathcal{F}_0}(\mathcal{F}_T,R^d)$, we still use $E_t(V)$ for $E[V~|~\mathcal{F}_t]$ for each $t \in [0,T]$, then $y^V := E_0(V) \in L^0(\mathcal{F}_0,R^d)$, and let $M^V_t := E_0(V) - E_t(V)$ for any $t \in [0,T]$, then $M^V \in M^p_{\mathcal{F}_0,0}$ (see Section \ref{section4} of this paper), it is also easy to check that each element in $M^p_{\mathcal{F}_0}$ is a generalized martingale. if $F$ satisfies condition $(S)$, let $Y^V \in S^p_{\mathcal{F}_0}$ be the unique solution of the stochastic equation $Y_t = y^V - F_t(Y,M^V) - M^V_t$ for any $t \in [0,T]$, then $(F,\xi)$ defines the mapping $G : L^p_{\mathcal{F}_0}(\mathcal{F}_T,R^d) \to L^p_{\mathcal{F}_0}(\mathcal{F}_T,R^d)$ by $G(V) := \xi + F_T(Y^V,M^V)$. In the same way as \cite{CN17}, it is very easy to prove that the set of solutions of $BSE$ (\ref{1.3}) corresponds in one to one way to that of fixed points of $G$. It remains to establish the fixed point theorems in $RN$ modules, precisely speaking, it requires the generalization of the famous fixed point theorems such as Banach's contraction mapping principle, Browder--Kirk's fixed point theorem and Schauder's fixed point theorem from Banach spaces to complete $RN$ modules since $L^p_{\mathcal{F}_0}(\mathcal{F}_T.R^d)$ is a complete $RN$ module (under the usual $(\varepsilon,\lambda)$--topology) rather than a Banach space.
\par
It is very easy to generalize Banach's contraction mapping principle. The essential difficulty lies in the generalization of the latter two fixed point theorems since the two theorems both involve compactness. It is well known that a complete $RN$ module is a Fr\'{e}ch\'{e}t space (namely a complete metrizable linear topological space), on which there does not necessarily exist any nontrivial continuous linear functional in general. Historically, generalizing Schauder's fixed point theorem from Banach spaces to such Fr\'{e}ch\'{e}t spaces is itself a famous open problem, see \cite[Problem 54]{Mau81}. On the other hand, it was already proved in \cite{Guo08} that a class of important closed convex subsets--almost surely (a.s.) bounded closed $L^0$--convex subsets, which are frequently encountered in the theory and applications of $RN$ modules, are rarely compact, although the famous Brouwer fixed point theorem was recently generalized to a special class of $RN$ modules in a nice way, namely $(L^0)^d$--the space of $d$--dimensional random vectors, see \cite{DKKS13}, there is a long way to go if one wants to generalize a similar result to general complete $RN$ modules. It is very fortunate that in this paper we can generalize Browder--Kirk's fixed point theorem to general complete $RN$ modules.
\par
Generalizing Browder--Kirk's fixed point theorem forces us to solve the following problem: How to introduce the two notions of `` weak compactness " and `` normal structure " for a.s. bounded closed $L^0$--convex subsets of $RN$ modules in a proper way ? `` Normal structure " (namely, random normal structure in the sense of this paper) can be easily introduced in a way similar to the classical case, however, it is rather delicate for us to look for a proper substitute for `` weak compactness ". It does not make sense to speak of weak compactness for $RN$ modules since they are not locally convex in general under the $(\varepsilon,\lambda)$--topology. Motivated by \v{Z}itkovi\'{c}'s work \cite{Z10}, we recently presented the notion of $L^0$--convex compactness for an $L^0$--convex subset of a topological $L^0$--module in \cite{GZWW17} and further gave a characterization for a closed $L^0$--convex subset of a complete $RN$ module to be $L^0$--convexly compact, which can be regarded as a natural generalization of the classical James theorem \cite{J64} that characterizes weak compactness of a closed convex subset of a Banach space, so that we can find a proper substitute for weak compactness, namely $L^0$--convex compactness. Therefore, we can eventually prove Browder--Kirk's fixed point theorem in a complete $RN$ module $S$ as follows : let $V$ be a closed $L^0$--convex subset of $S$ such that $V$ is $L^0$--convexly compact and has random normal structure, then a nonexpansive mapping $T$ from $V$ to $V$ has a fixed point.
\par
The remainder of this paper is organized as follows: Section \ref{section2} of this paper provides some preliminaries for $RN$ modules and gives an interesting generalization of Banach's contraction mapping principle; Section \ref{section3} is devoted to introducing and studying the notions of $L^0$--convex compactness and random normal structure, at the same time we prove Browder--Kirk's fixed point theorem in complete $RN$ modules; similar to \cite{CN17}, Section \ref{section4} gives the relation between $BSEs$ of the form (\ref{1.3}) and fixed point problems in complete $RN$ modules; Section \ref{section5} is devoted to applications of the above--stated generalized Banach's contraction mapping principle to the existence and uniqueness of solutions of $BSEs$ of the form (\ref{1.3}), the part can be regarded as the conditional versions of the corresponding main results of \cite{CN17}; Section \ref{section6} is devoted to applications of the classical Browder--Kirk's fixed point theorem to the existence of solutions of $BSEs$ of the form (\ref{1.1}) and applications of the above--stated generalized Browder--Kirk's fixed point theorem to the existence of solutions of $BSEs$ of the form (\ref{1.3}), the part is completely new. Finally, since $L^p_{\mathcal{F}_0}(\mathcal{F}_T,R^d)$ can be represented as the countable concatenation hull of $L^p(\mathcal{F}_T,R^d)$, in Section \ref{section7} this paper concludes with a brief discussion of the relation between the solutions of $BSEs$ of the forms (\ref{1.1})  and (\ref{1.3}).


\section{Some preliminaries and an interesting generalization of Banach's contraction mapping principle }\label{section2}

Throughout this paper, unless otherwise stated, $K$ denotes the field $R$ of real numbers or $C$ of complex numbers and $(\Omega,\mathcal{F},P)$ a probability space. $L^0(\mathcal{F},K)$ denotes the algebra of equivalence classes of $K$--valued $\mathcal{F}$--measurable random variables on $(\Omega,\mathcal{F},P)$, $L^0(\mathcal{F}) := L^0(\mathcal{F},R)$ and $\bar{L}^0(\mathcal{F})$ is the set of equivalence classes of extended real--valued $\mathcal{F}$--measurable random variables on $(\Omega,\mathcal{F},P)$.
\par
Proposition \ref{proposition2.1} below can be regarded as the randomized version of the supremum and infimum principle for $\bar{R}:=[-\infty, +\infty]$ and $R$.
\begin{proposition}\cite{DS58}\label{proposition2.1}
Define a partial order $\leq$ on $\bar{L}^0(\mathcal{F})$ as follows: $\xi \leq \eta$ if $\xi^{0}(\omega) \leq \eta^{0}(\omega)$ for almost all $\omega$ in $\Omega$ $($briefly, $\xi^0 \leq \eta^0$ a.s.$)$, where $\xi^{0}$ and $\eta^{0}$ are respectively arbitrarily chosen representatives of $\xi$ and $\eta$ in $\bar{L}^0(\mathcal{F})$. Then $(\bar{L}^{0}(\mathcal{F}),\leq)$ is a complete lattice, $\bigvee H$ and $\bigwedge H$ respectively stand for the supremum and infimum of a subset $H$ of $\bar{L}^0(\mathcal{F})$. Furthermore, the following statements hold:
\begin{enumerate}[(1)]
\item There exist two sequences $\{ a_n, n \in N \}$ and $\{ b_n, n \in N\}$ in $H$ such that $\bigvee H = \bigvee_{n \geq 1} a_n $ and $\bigwedge H = \bigwedge_{n \geq 1} b_n $.
\item If $H$ is directed upwards ,namely for any two elements $h_1$ and $h_2$ in $H$ there exists some $h_3 $ in $H$ such that $ h_3 \geq h_1 \bigvee h_2$, then $\{a_n, n \in N\}$ stated above can be chosen as nondecreasing; similarly, if $H$ is directed downwards, then $\{b_n, n\in N\}$ stated above can ba chosen as nonincreasing.
\item $L^0(\mathcal{F})$, as a sublattice of $\bar{L}^0(\mathcal{F})$, is Dedekind complete (namely, any subset with an upper bound has a supremum).
\end{enumerate}

\end{proposition}
\par
As usual, for two elements $\xi$ and $\eta$ in $\bar{L}^0(\mathcal{F}), \xi > \eta$ means $\xi \geq \eta$ and $\xi \neq \eta$. Very often, for an $\mathcal{F}$--measurable set $A$, we say $\xi > \eta$ on $A$ if $\xi^0(\omega) >\eta^0(\omega)$ for almost all $\omega$ in $A$, where  $\xi^0$ and $\eta^0$  are respectively arbitrarily chosen representatives of $\xi$ and $\eta$. Similarly, one can understand $\xi \geq \eta$ on $A$.
\par
In this paper, we always employ the following notation and terminologies:
\par
$L^0_+(\mathcal{F}) := \{ \xi \in L^0(\mathcal{F}) ~|~ \xi \geq 0 \}$;
\par
$L^0_{++}(\mathcal{F}) := \{ \xi \in L^0(\mathcal{F}) ~|~ \xi >0$ on $\Omega \}$;
\par
$\bar{L}^0_{++}(\mathcal{F}) := \{ \xi \in \bar{L}^0(\mathcal{F}) ~|~ \xi >0$ on $\Omega \}$;\\
$I_A$ stands for the characteristic function of $A \in \mathcal{F}$, namely $I_A(\omega) =1$ if $\omega \in A$ and 0 otherwise and $\tilde{I}_A$ stands for the equivalence class of $I_A$; $[A]$ stands for the equivalence class of $A \in \mathcal{F}$, namely $[A] := \{ B \in \mathcal{F}~|~ P(A \vartriangle B) =0 \}$, where $A \vartriangle B := (A \cap B^c) \cup (B \cap A^c)$ and $A^c$ stands for the complement of $A$ relative to $\Omega$; For any two elements $\xi$ and $\eta$ in $\bar{L}^0(\mathcal{F})$, let $\xi^0$ and $\eta^0$  be respectively arbitrarily chosen representatives of $\xi$ and $\eta$, since $A := \{ \omega \in \Omega ~|~ \xi^0(\omega) > \eta^0(\omega) \}$ only differs by a null set for different choices of $\xi^0$ and $\eta^0$, we simply write $(\xi >\eta)$ for $A$ for convenience, whereas $[\xi >\eta]$ stands for the equivalence class of $A$, and $I_{[\xi > \eta]}$ stands for $\tilde{I}_A$. Similarly, one can understand $(\xi \geq \eta)$ or $[\xi \geq \eta]$, and so on.
\par
\begin{definition}\cite{Guo92,Guo93}\label{definition2.2}
An ordered pair $(E,\|\cdot\|)$ is called a random normed module over $K$ with base $(\Omega, \mathcal{F},P)$ $($briefly, an $RN$ module$)$ if $E$ is a left module over the algebra $L^0(\mathcal{F},K)$ (briefly, an $L^0(\mathcal{F},K)$--module) and $\|\cdot\|$ is a mapping from $E$ to $L^0_+(\mathcal{F})$ such that the following are satisfied:
\begin{enumerate}[(RNM-1)]
\item $\|x\| = 0$ implies $x = \theta$ $($the null in $E$$)$;
\item $\|\xi \cdot x \| = |\xi| \cdot \|x\|$ for any $(\xi,x) \in L^0(\mathcal{F},K) \times E$, where $\cdot : L^0(\mathcal{F},K) \times E \to E$ stands for the module multiplication operation, $\xi \cdot x$ is often written as $\xi x$;
\item $\|x + y\| \leq \|x\| + \|y\|$ for any $x$ and $y \in E$;

\end{enumerate}
 $\|\cdot\|$ is called the $L^0$--norm on $E$. If $\|\cdot\|$ only satisfies $(RNM-2)$ and $(RNM-3)$, then it is called an $L^0$--seminorm on $E$.
\end{definition}

\begin{example}\label{example2.3}
Let $(B,\|\cdot\|)$ be a normed space over $K$, a mapping $V : (\Omega,\mathcal{F},P) \\
\to B$ is called a random element if $V^{-1}(G) := \{ \omega \in \Omega ~|~ V(\omega) \in G \} \in \mathcal{F}$ for each open subset $G$ of $B$, see \cite{B72,B76}. A random element $V$ is said to be simple if $V$ only takes finitely many values. A random element $V$ is said to be strongly measurable if $V$ is a pointwise limit of a sequence of simple random elements. It is easy to check that a random element is strongly measurable iff its range is separable. Denote by $L^0(\mathcal{F},B)$ the linear space of equivalence classes of $B$--valued strongly measurable random elements on $(\Omega,\mathcal{F},P)$, then $L^0(\mathcal{F},B)$ is an $L^0(\mathcal{F},K)$--module with the module multiplication operation $\cdot : L^0(\mathcal{F},K) \times L^0(\mathcal{F},B) \to L^0(\mathcal{F},B)$ defined by  $\xi x :=$ the equivalence class of $\xi^0 \cdot x^0$, where $\xi^0$ and $x^0$ are respectively arbitrarily chosen representatives of $\xi $ and $x$ for any $(\xi,x) \in L^0(\mathcal{F},K) \times L^0(\mathcal{F},B)$ and $(\xi^0 \cdot x^0)(\omega) = \xi^0(\omega) \cdot x^0(\omega)$ for each $\omega$ in $\Omega$. Furthermore, $L^0(\mathcal{F},B)$ is an $RN$ module over $K$ with base $(\Omega,\mathcal{F},P)$, with the $L^0$--norm induced by the norm $\|\cdot\|$, which is still denoted by $\|\cdot\|$ and is given by $\|x\| =$ the equivalence class of $\|x^0\|$, where $x^0$ is an arbitrarily chosen representatives of $x$ for any $x \in L^0(\mathcal{F},B)$ and $\|x^0\|(\omega) = \|x^0(\omega)\|$ for any $\omega$ in $\Omega$. Specially, $L^0(\mathcal{F},K)$ is an $RN$ module.
\end{example}

\begin{example}\label{example2.4}
When $1 \leq p < +\infty, (S^p_{\mathcal{F}_0},|||\cdot|||_p)$ and $(L^p_{\mathcal{F}_0}(\mathcal{F}_T,R^d), |||\cdot|||_p)$ as in Introduction of this paper are $RN$ modules over $R$ with base $(\Omega,\mathcal{F}_0,P)$. For $p = +\infty$, let $S^\infty_{\mathcal{F}_0} = \{ Y \in S ~|~ |||Y|||_\infty := \bigwedge \{ \eta \in \bar{L}^0_+(\mathcal{F}_0) ~|~ \sup_{t \in [0,T]} |Y_t| \leq \eta \} \in L^0_+(\mathcal{F}_0) \}$, then it is easy to check that $S^\infty_{\mathcal{F}_0}$ is the $L^0(\mathcal{F}_0)$--module generated by $S^\infty$ (see \cite{CN17} for $S^\infty$) and $(S^\infty_{\mathcal{F}_0},|||\cdot|||_\infty)$ becomes an $RN$ module over $R$ with base $(\Omega,\mathcal{F}_0,P)$. Similarly, let $L^\infty_{\mathcal{F}_0}(\mathcal{F}_T,R^d) = \{ x \in L^0(\mathcal{F}_T,R^d)~|~ |||x|||_\infty := \bigwedge \{ \eta \in \bar{L}^0_+(\mathcal{F}) ~|~ |x| \leq \eta \} \in L^0_+(\mathcal{F}) \}$, then $L^\infty_{\mathcal{F}_0}(\mathcal{F}_T,R^d)$ is the $L^0(\mathcal{F}_0)$--module generated by $L^\infty(\mathcal{F}_T,R^d)$ and $(L^\infty_{\mathcal{F}_0}(\mathcal{F}_T,R^d), |||\cdot|||_\infty )$ becomes an $RN$ module over $R$ with base $(\Omega,\mathcal{F}_0,P)$.
\end{example}

\begin{proposition}\cite{Guo92,Guo93}\label{proposition2.5}
Let $(E,\|\cdot\|)$ be an $RN$ module over $K$ with base $(\Omega,\mathcal{F},P)$. For any given positive numbers $\varepsilon$ and $\lambda$  such that $0 < \lambda <1$, let $N_{\theta}(\varepsilon,\lambda) = \{ x \in E ~|~ P\{\omega \in \Omega ~|~ \|x\|(\omega) < \varepsilon \} > 1-\lambda \}$, called the $(\varepsilon,\lambda)$--neighborhood of $\theta$. Then $\{ N_{\theta}(\varepsilon,\lambda)~|~\varepsilon >0 , 0 <\lambda <1 \}$ forms a local base for some metrizable linear topology on $E$, called the $(\varepsilon,\lambda)$--topology. Specially, $L^0(\mathcal{F},K)$ is a Hausdorff topological algebra over $K$ under its $(\varepsilon,\lambda)$--topology, it is obvious that the $(\varepsilon,\lambda)$--topology on $L^0(\mathcal{F},K)$ is exactly the topology of convergence in probability. Further, $E$ is a Hausdorff topological module over the topological $L^0(\mathcal{F},K)$ when $E$ and $L^0(\mathcal{F},K)$ are endowed with their respective $(\varepsilon,\lambda)$--topologies.
\end{proposition}
\par
From now on, for brevity and convenience the $(\varepsilon,\lambda)$--topology for any $RN$ module is always denoted by $\mathcal{T}_{\varepsilon,\lambda}$.
\par
The terminology `` $(\varepsilon,\lambda)$--topology " comes from theory of probabilistic metric spaces \cite{SS83}, where the spaces more general than $RN$ modules are often endowed with this kind of topology by B.Schweizer and A.Sklar. Essentially speaking, the $(\varepsilon,\lambda)$--topology for an $RN$ module is not locally convex in general, for example, the simplest $RN$ module $L^0(\mathcal{F},K)$ does not admit any nontrivial continuous linear functional if $\mathcal{F}$ does not contain any atom. To overcome the difficulty, Guo introduced the notion of a.s. bounded random linear functionals, established the Hahn--Banach theorems for them and thus led to the theory of random conjugate spaces for $RN$ modules. Let $(E,\|\cdot\|)$ be an $RN$ module over $K$ with base $(\Omega,\mathcal{F},P)$, a linear operator $f : E \to L^0(\mathcal{F},K)$ is said to be a.s. bounded if there exists some $\xi \in L^0_+(\mathcal{F})$ such that $|f(x)| \leq \xi \cdot \|x\|$ for any $x \in E$, $\|f\| := \bigwedge \{ \xi \in L^0_+(\mathcal{F}) ~|~ |f(x)| \leq \xi \cdot \|x\|$ for any $x \in E \}$ is called the $L^0$--norm of $f$. It was proved in \cite{Guo92,Guo93,Guo96b} that a linear operator $f : E \to L^0(\mathcal{F},K)$ is a.s. bounded if and only if $f$ is a continuous module homomorphism from $(E,\mathcal{T}_{\varepsilon,\lambda})$ to $(L^0(\mathcal{F},K),\mathcal{T}_{\varepsilon,\lambda})$, in which case $\|f\| = \bigvee \{ |f(x)| ~|~ x \in E$ and $\|x\| \leq 1 \}$. Denote by $E^*_{\varepsilon,\lambda}$ the $L^0(\mathcal{F},K)$--module of continuous module homomorphisms from $(E,\mathcal{T}_{\varepsilon,\lambda})$ to $(L^0(\mathcal{F},K),\mathcal{T}_{\varepsilon,\lambda})$, then $(E^*_{\varepsilon,\lambda},\|\cdot\|)$ is again an $RN$ module over $K$ with base $(\Omega,\mathcal{F},P)$, called the random conjugate space of $E$ under the $(\varepsilon,\lambda)$--topology. Finally, by the Hahn--Banach theorem, $\|x\| = \bigvee\{ |f(x)| ~|~ f \in E^*$ and $\|f\| \leq 1 \}$ for any $x \in E$.
\par
Roughly speaking, the $(\varepsilon,\lambda)$--topology is an abstract generalization of the ordinary topology of convergence in probability, it is very natural and useful in probability theory and stochastic finance. However, it is too weak to meet some needs of random convex analysis, which is established to provide a convex--analysis--like tool for conditional risk measures and related optimization problems, see \cite{FKV09,GZWW17,GZWYYZ17,GZZ14,GZZ15a,GZZ15b}. In 2009, the financial applications motivated Filipovi\'{c}, et.al to introduce another kind of topology for an $RN$ module in \cite{FKV09}, called the locally $L^0$--convex topology. As a part of their work of \cite{FKV09},we state it as follows:
\begin{proposition}\cite{FKV09}\label{proposition2.6}
Let $(E,\|\cdot\|)$ be an $RN$ module over $K$ with base $(\Omega,\mathcal{F},P)$. For any given $\varepsilon \in L^0_{++}(\mathcal{F})$, let $B_\theta(\varepsilon) = \{ x \in E ~|~\|x\| <\varepsilon$ on $ \Omega \}$, a subset $G$ of $E$ is said to be open if for each $g \in G$ there exists some $\varepsilon \in L^0_{++}(\mathcal{F})$ such that $g + B_\theta(\varepsilon) \subset G$, then all such open subsets forms a Hausdorff topology, called the locally $L^0$--convex topology for $E$, denoted by $\mathcal{T}_c$. Specially, $L^0(\mathcal{F}, K)$ is a topological ring under its locally $L^0$--convex topology and $E$ is a topological module over the topological ring $L^0(\mathcal{F}, K)$ when $E$ and $L^0(\mathcal{F}, K)$ are respectively endowed with their locally $L^0$--convex topologies, in which case $\{ B_\theta(\varepsilon)~|~\varepsilon \in L^0_{++}(\mathcal{F})\}$ forms a local base for the topological module $(E,\mathcal{T}_c)$.
\end{proposition}
\par
From now on, the locally $L^0$--convex topology for any $RN$ module is always denoted by $\mathcal{T}_c$. When $(\Omega,\mathcal{F},P)$ is a trivial probability space, namely $\mathcal{F} = \{ \Omega, \emptyset\}$, an $RN$ module with base $(\Omega,\mathcal{F},P)$ degenerates to an ordinary normed space, at which time the $(\varepsilon,\lambda)$--topology and the locally $L^0$--convex topology coincide. In general, $\mathcal{T}_c$ is much stronger than $\mathcal{T}_{\varepsilon,\lambda}$ and $\mathcal{T}_c$ is rarely a linear topology, but there are many natural connections between the two theories derived from the two kinds of topologies, which was first studied by Guo in \cite{Guo10}, for example, let $(E,\|\cdot\|)$ be an $RN$ module over $K$ with base $(\Omega,\mathcal{F},P)$ and $E^*_c$ the $L^0(\mathcal{F},K)$--module of continuous module homomorphisms from $(E,\mathcal{T}_{c})$ to $(L^0(\mathcal{F},K),\mathcal{T}_{c})$, (called the random conjugate space of $E$ under $\mathcal{T}_c$), then $E^*_{\varepsilon,\lambda} = E^*_c$, so we often simply write $E^*$ for both $E^*_{\varepsilon,\lambda}$ and $E^*_c$. Other deep connections are realized via the following crucial notion-- the countable concatenation property.
\begin{definition}\cite{Guo10}\label{definition2.7}
A subset $G$ of an $L^0(\mathcal{F}, K)$--module $E$ is said to have the countable concatenation property (simply, $G$ is stable )if for any sequence $\{ x_n: n \in N \}$ in $G$ and any countable partition $\{ A_n: n \in N \}$ of $\Omega$ to $\mathcal{F}$ ( namely, each $A_n \in \mathcal{F}, A_n \cap A_m = \emptyset $ for $n \neq m $ and $\cup_{n \in N} A_n = \Omega $) there exists some $x \in G$ such that ${\tilde I}_{A_n}x = {\tilde I}_{A_n}x_n$ for each $n \in N$.
\end{definition}
\begin{remark}\label{remark2.8}
(1). For a subset $G$ of an $RN$ module $(E,\|\cdot\|)$ or, more generally, a random locally convex module $(E,\mathcal{P})$ (see \cite{Guo10}), if $G$ is stable, then $x$ in Definition \ref{definition2.7} always uniquely exists, in which case $x$ is usually denoted by $\sum_{n \geq 1} \tilde{I}_{A_n}x_n$. (2). Guo proved in \cite{Guo10} that for each $\mathcal{T}_{\varepsilon,\lambda}$--complete $RN$ module $(E,\|\cdot\|)$, $E$ is always stable, it is straightforward to verify that $(S^p_{\mathcal{F}_0},|||\cdot|||_p)$,$(M^p_{\mathcal{F}_0},|||\cdot|||_p)$ and $(L^p_{\mathcal{F}_0}(\mathcal{F}_T,R^d),|||\cdot|||_p)$ are all $\mathcal{T}_{\varepsilon,\lambda}$--complete, so $S^p_{\mathcal{F}_0}, M^p_{\mathcal{F}_0}$ and $L^p_{\mathcal{F}_0}(\mathcal{F}_T,R^d)$ are all stable (in fact, the verification that $L^p_{\mathcal{F}_0}(\mathcal{F}_T,R^d)$ is stable has been done in \cite{GZZ14}). (3). It is obvious that $L^0(\mathcal{F},B)$ is also stable for any normed space. (4). It was proved in \cite{Guo10} that an $RN$ module $(E,\|\cdot\|)$ is $\mathcal{T}_{\varepsilon,\lambda}$--complete if and only if both $E$ is stable and  $(E,\|\cdot\|)$ is $\mathcal{T}_c$--complete.
\end{remark}
\par
The following proposition is a special case of \cite[Theorem 3.12]{Guo10}, which can be used to derive some interesting results simplifying the proofs of this paper.
\begin{proposition}\label{proposition2.9}\cite{Guo10}
Let $G$ be a stable nonempty subset of an $RN$ module $(E, \|\cdot\|)$ . Then $\bar{G}_{\varepsilon,\lambda} = \bar{G}_c$, where $\bar{G}_{\varepsilon,\lambda}$ and $\bar{G}_c$ respectively stand for the $\mathcal{T}_{\varepsilon,\lambda}$--closure and the $\mathcal{T}_c$--closure of $G$.
\end{proposition}
\begin{lemma}\label{lemma2.10}
Let $G$ be a stable nonempty subset of $L^0(\mathcal{F})$ such that $G$ is bounded above (namely there exists some $\eta \in L^0(\mathcal{F})$ such that $\xi \leq \eta$ for any $\xi \in G$), then for any given $\varepsilon \in L^0_{++}(\mathcal{F})$ there exists some $\xi_{\varepsilon} \in G$ such that $\xi_{\varepsilon} > \bigvee G - \varepsilon$ on $\Omega$. Similarly, if $G$ is bounded below, then for any given $\varepsilon \in L^0_{++}(\mathcal{F})$ there exists some $\xi_{\varepsilon} \in G$ such that $\xi_{\varepsilon} < \bigwedge G + \varepsilon$ on $\Omega$.
\end{lemma}
\begin{proof}
We only give the proof for the case when $G$ is bounded above, the other case is similar.
\par
First, $G$ is directed upwards : for any two elements $\xi_1$ and $\xi_2$, let $\xi_3 = \tilde{I}_{(\xi_1 <\xi_2)} \xi_2 + \tilde{I}_{(\xi_1 \geq \xi_2)} \xi_1$, then it is obvious that $\xi_3 \in G$ and $\xi_3 = \xi_1 \bigvee \xi_2$. By (2) of Proposition \ref{proposition2.1}, there is a nondecreasing sequence $\{ \xi_n, n \in N \}$ of $G$ such that $\{ \xi_n, n \in N \}$ converges a.s. to $\bigvee G$ in a nondecreasing way. Since $G$ is bounded above, $\bigvee G \in L^0(\mathcal{F})$, it is very easy to see $\bigvee G \in \bar{G}_{\varepsilon,\lambda}$, and hence $\bigvee G \in \bar{G}_c$ by Proposition \ref{proposition2.9}, which in turn implies that there exists some $\xi_\varepsilon \in G$ for any given $\varepsilon \in L^0_{++}(\mathcal{F})$ such that $\xi_{\varepsilon} > \bigvee G - \varepsilon$ on $\Omega$.
\end{proof}
\par
Let $(E_1,\|\cdot\|_1)$ and $(E_2,\|\cdot\|_2)$ be two $RN$ modules over $K$ with base $(\Omega,\mathcal{F},P)$ and $G$ a nonempty subset of $E_1$ such that both $G$ and $E_2$ are stable. A mapping $T : G \to E_2$ is said to be stable if $T(\sum_{n\geq 1} \tilde{I}_{A_n} g_n) = \sum_{n\geq 1} \tilde{I}_{A_n}T(g_n)$ for any sequence $\{ g_n: n \in N \}$ in $G$ and any countable partition $\{ A_n: n \in N \}$ of $\Omega$ to $\mathcal{F}$. It is very to see that if $T$ is stable, then $T(G)$ is also stable.
\begin{lemma}\label{lemma2.11}
Let $(E_1,\|\cdot\|_1)$ and $(E_2,\|\cdot\|_2)$ be two $\mathcal{T}_{\varepsilon,\lambda}$--complete $RN$ modules over $K$ with base $(\Omega,\mathcal{F},P)$, $G$ a stable subset of $E_1$ and $T : G \to E_2$ is $L^0$--Lipschitzian $($namely, there exists $\xi \in L^0_+(\mathcal{F})$ such that $\|T(x_1) - T(x_2)\|_2 \leq \xi \|x_1 -x_2 \|_1$ for any $x_1$ and $x_2$ in $E_1)$. Then $T$ is stable.
\end{lemma}
\begin{proof}
We first prove that $T(\tilde{I}_{A} x_1 +\tilde{I}_{A^c} x_2) = \tilde{I}_{A} T(x_1) + \tilde{I}_{A^c} T(x_2)$ for any $A \in \mathcal{F}$ and for any $x_1$ and $x_2$ in $G$. In fact, $\|T(\tilde{I}_{A}x_1 + \tilde{I}_{A^c}x_2) - \tilde{I}_{A}T(x_1) - \tilde{I}_{A^c}T(x_2)\|_2 = \|(\tilde{I}_{A} +\tilde{I}_{A^c})T(\tilde{I}_{A}x_1 + \tilde{I}_{A^c}x_2) - \tilde{I}_{A}T(x_1) - \tilde{I}_{A^c}T(x_2)\|_2  \leq \tilde{I}_{A}\cdot\tilde{I}_{A^c}\cdot\xi\cdot\|x_1-x_2\|_1 + \tilde{I}_{A^c}\cdot\tilde{I}_{A}\cdot\xi\cdot\|x_1-x_2\|_1 =0$, so $T(\tilde{I}_{A}x_1 + \tilde{I}_{A^c}x_2) = \tilde{I}_{A}T(x_1) + \tilde{I}_{A^c}T(x_2)$.
\par
Second, for any sequence $\{ x_n: n \in N \}$ in $G$ and any countable partition $\{ A_n: n \in N \}$ of $\Omega$ to $\mathcal{F}$, let $x := \sum_{n \geq 1} \tilde{I}_{A_n} x_n $, since $E_1$ is $\mathcal{T}_{\varepsilon,\lambda}$--complete and $\{ \sum_{n = 1}^{k} \tilde{I}_{A_n} x_n, k \in N \}$ is a $\mathcal{T}_{\varepsilon,\lambda}$--Cauchy sequence, $x$ is exactly the limit of the Cauchy sequence. Since $\sum_{n\geq1} \tilde{I}_{A_n} x_n = \sum^{k}_{n=1} \tilde{I}_{A_n} x_n + \tilde{I}_{(\bigcup^{k}_{n=1}A_n)^c} (\sum_{n\geq1} \tilde{I}_{A_n} x_n)$,  $T(\sum_{n\geq 1} \tilde{I}_{A_n} x_n) = \sum^{k}_{n=1} \tilde{I}_{A_n} T(x_n) + \tilde{I}_{(\bigcup^{k}_{n=1}A_n)^c} T(
\sum_{n \geq 1}  \tilde{I}_{A_n} x_n)$. Further, since $E_2$ is $\mathcal{T}_{\varepsilon,\lambda}$--complete, $\sum_{n\geq1} \tilde{I}_{A_n} T(x_n) = \lim_{k\to \infty} 
 \sum^{k}_{n=1} \tilde{I}_{A_n} T(x_n)$. It is obvious that $\tilde{I}_{(\bigcup^{k}_{n=1}A_n)^c} T(\sum_{n \geq 1}  \tilde{I}_{A_n} x_n)$ converges in $\mathcal{T}_{\varepsilon,\lambda}$ to $\theta$ when $k \to \infty$ by noticing $\lim_{k \to \infty}\sum^{\infty}_{n=k+1} P(A_n) =0$, then $T(\sum_{n\geq1} \tilde{I}_{A_n} x_n) 
 = \sum_{n\geq1}  \tilde{I}_{A_n}  T(x_n)$.
\end{proof}
\par
Let $(E,\|\cdot\|)$ be a $\mathcal{T}_{\varepsilon,\lambda}$--complete $RN$ module over $K$ with base $(\Omega,\mathcal{F},P)$ (hence $E$ is stable), $T$ a mapping from $E$ to $E$ and $L : (\Omega,\mathcal{F},P) \to N$ a positive integer--valued $\mathcal{F}$--measurable random variable. Denote $L(\omega) = \sum^{\infty}_{k=1} I_{(L=k)}(\omega) \cdot k$ for any $\omega \in \Omega$, $T^{(L)} : E \to E$ is defined by $T^{(L)}(x) =\sum^{\infty}_{k=1} \tilde{I}_{(L=k)} T^{(k)} (x)$ for any $x \in E$, where $T^{(k)}$ stands for the $k$--th iteration of $T$.
\par
When we study the existence and uniqueness of some $BSDEs$ as special cases of $BSEs$ of the form (\ref{1.3}), we need the following interesting generalization of Banach's contraction mapping principle (i.e., Theorem \ref{theorem2.13} below) since the contraction constant is random.
\par
Let us first recall from \cite{SS83} : an ordered pair $(E,d)$ is called a random metric space with base $(\Omega,\mathcal{F},P)$ if $E$ is a nonempty set and $d : E \times E \to L^0_+(\mathcal{F})$ satisfies the following three axioms: (i) $d(x,y) =0 \Leftrightarrow x=y$; (ii) $d(x,y) = d(y,x)$ for any $(x,y) \in E \times E$, and (iii) $d(x,z) \leq d(x,y) + d(y,z)$ for any $x,y$ and $z$ in $E$. In fact, the definition is essentially an equivalent variant of the original one of a random metric space introduced in \cite{SS83}. For any given positive number $\varepsilon$ and $\lambda$ such that $0 < \lambda < 1$, let $\mathcal{U}(\varepsilon,\lambda) = \{ (x,y) \in E \times E ~|~ P \{ \omega \in \Omega ~|~ d(x,y)(\omega) < \varepsilon \} > 1 - \lambda \}$, then the family $\{ \mathcal{U}(\varepsilon,\lambda) ~|~ \varepsilon >0 , 0 < \lambda < 1 \}$ forms a base for some metrizable uniform structure on $E$. Further, if the uniform structure is complete, then $(E,d)$ is said to be $\mathcal{T}_{\varepsilon,\lambda}$--complete. Similar to the proof of the classical Banach's contraction mapping principle, Proposition \ref{proposition2.12} below was proved in \cite{Guo89}, an application of which was given in \cite{Guo99}.
\begin{proposition} \cite{Guo89} \label{proposition2.12}
Let $(E,d)$ is a $\mathcal{T}_{\varepsilon,\lambda}$--complete random metric space with base $(\Omega,\mathcal{F},P)$ and $T : E \to E$ a mapping. If there are a positive integer $l \in N$ and $\xi \in L^0_+(\mathcal{F})$ satisfying $\xi <1$ on $\Omega$ such that $d(T^{(l)}(x),T^{(l)}(y)) \leq \xi d(x,y)$ for all $x$ and $y$ in $E$, then $T$ has a unique fixed point.
\end{proposition}
\begin{theorem}\label{theorem2.13}
Let $(E,\|\cdot\|)$ be a $\mathcal{T}_{\varepsilon,\lambda}$--complete $RN$ module over $K$ with base $(\Omega,\mathcal{F},P)$, $G$ a  $\mathcal{T}_{\varepsilon,\lambda}$--closed stable subset of $E$ and $T : G \to G$ a stable mapping. If there exist some positive integer--valued $\mathcal{F}$--measurable random variable $L$ and some $\xi \in L^0_+(\mathcal{F})$ such that $\xi <1$ on $\Omega$ and $\|T^{(L)}(x) -T^{(L)}(y) \| \leq \xi \cdot \|x-y\|$ for any $x$ and $y$ in $G$, then $T$ has a unique fixed point in $G$.
\end{theorem}
\begin{proof}
One can easily see from Proposition \ref{proposition2.12} that $T^{(L)}$ has a unique fixed point $x^* \in G$, Since $T$ is stable, $T$ and $T^L$ are commutative, then $x^*$ is also a unique fixed point of $T$.
\end{proof}

\section{$L^0$--convex compactness, random normal structure and Browder--Kirk's fixed point theorem in $\mathcal{T}_{\varepsilon,\lambda}$--complete $RN$ modules}\label{section3}

Let $E$ be an $L^0(\mathcal{F},K)$--module. A nonempty subset $G$ of $E$ is said to be $L^0$--convex if $\xi x + (1-\xi)y \in G$ for any $x$ and $y$ in $G$ and for any $\xi \in L^0_+(\mathcal{F})$ such that $0 \leq \xi \leq 1$.
\begin{definition}\cite{GZWW17} \label{definition3.1}
Let $(E,\mathcal{T})$ be a Hausdorff topological module over the topological algebra $(L^0(\mathcal{F},K), \mathcal{T}_{\varepsilon,\lambda})$ (namely, $(E,\mathcal{T})$ is a Hausdorff linear topological space over $K$ and $E$ is also an $L^0(\mathcal{F},K)$--module such that the module multiplication operation $\cdot : (L^0(\mathcal{F},K),\mathcal{T}_{\varepsilon,\lambda}) \times (E,\mathcal{T}) \to (E,\mathcal{T})$ is continuous). A nonempty $L^0$--convex subset $G$  of $E$ is said to be $L^0$--convexly compact $($or, $G$ is said to have $L^0$--convex compactness$)$ if any family of $\mathcal{T}$--closed $L^0$--convex subsets of $G$ has the nonempty intersection whenever the family has the finite intersection property.
\end{definition}
\par
Let $(E,\|\cdot\|)$ be an $RN$ module over $K$ with base $(\Omega,\mathcal{F},P)$, then $(E,\mathcal{T}_{\varepsilon,\lambda})$ is a Hausdorff topological module over the topological algebra $(L^0(\mathcal{F},K),\mathcal{T}_{\varepsilon,\lambda})$ according to Proposition \ref{proposition2.6}. Theorem \ref{theorem3.2} below, which was recently proved in \cite{GZWW17}, gives a characterization for an $L^0$--closed convex subset of a complete $RN$ module to be $L^0$--convexly compact.
\par
Let $(E,\|\cdot\|)$ be a $\mathcal{T}_{\varepsilon,\lambda}$--complete $RN$ module over $K$ with base $(\Omega,\mathcal{F},P)$ and $1\leq p \leq +\infty$. Further, let $L^p(E) = \{ x \in E : \|x\|_p < +\infty \}$, where $\|x\|_p$ denotes the ordinary $L^p$--norm of $\|x\|$, namely $\|x\|_p = (\int_{\Omega} \|x\|^p dP)^{1/p}$ for $1\leq p < +\infty$ and $\|x\|_{\infty} = \inf\{ M \in [0,+\infty) ~|~ \|x\| \leq M \}$, then $(L^p(E),\|\cdot\|_p)$ is a Banach space over $K$. Let $E^{**}$ be its second random conjugate space, the canonical mapping $J: E \to E^{**}$ is defined by $J(x)(f) = f(x)$ for any $x \in E$ and $f \in E^*$, then $J$ is $L^0$--norm--preserving by the Hahn--Banach theorem for random linear functionals, if, in addition, $J$ is also surjective, then $(E,\|\cdot\|)$ is said to be random reflexive. In 1997, Guo proved in \cite{Guo97} that $(E,\|\cdot\|)$ is random reflexive if and only if $L^p(E)$ is reflexive for any given $p$ such that $1<p<+\infty$, which was further used by Guo and Li in 2005 in \cite{GL05} to prove that $(E,\|\cdot\|)$ is random reflexive if and only if each $f \in E^*$ can attain its $L^0$--norm on the random closed unit ball of $E$. Recently, it was proved in \cite{GZWW17} that $(E,\|\cdot\|)$ is random reflexive if and only if  each $\mathcal{T}_{\varepsilon,\lambda}$--closed $L^0$--convex and a.s. bounded subset of $E$ is $L^0$--convexly compact, where a subset $G$ of $E$ is said to be a.s. bounded if $\bigvee \{ \|g\| : g \in G \} \in L^0_+(\mathcal{F})$. In particular, we have obtained the following more general result:
\par
\begin{theorem}\label{theorem3.2}\cite{GZWW17}
Let $(E,\|\cdot\|)$ be a $\mathcal{T}_{\varepsilon,\lambda}$--complete $RN$ module over $K$ with base $(\Omega,\mathcal{F},P)$ and $G$ a $\mathcal{T}_{\varepsilon,\lambda}$--closed $L^0$--convex subset of $E$. Then $G$ is $L^0$--convexly compact if and only if for each $f \in E^*$ there exists $g_0 \in G$  such that $Re(f(g_0)) = \bigvee\{ Re(f(g)) : g \in G \}$.
\end{theorem}
\par
\begin{definition}\label{definition3.3}
Let $(E,\|\cdot\|)$ be a $\mathcal{T}_{\varepsilon,\lambda}$--complete $RN$ module over $K$ with base $(\Omega,\mathcal{F},P)$. A $\mathcal{T}_{\varepsilon,\lambda}$--closed $L^0$--convex subset $G$ of $E$ is said to have random normal structure if for each a.s. bounded $\mathcal{T}_{\varepsilon,\lambda}$--closed $L^0$--convex subset $H$ of $G$ such that $D(H) := \bigvee\{ \|x-y\| : x,y \in H \}>0$, there exists a nondiametral point $h$ of $H$, namely $\bigvee\{ \|h-x\| : x \in H \} < D(H)$ on $(D(H)>0)$. $D(H)$ is called the random diameter of $H$.
\end{definition}
\par
To study random normal structure, let $(E,\|\cdot\|)$ be a $\mathcal{T}_{\varepsilon,\lambda}$--complete $RN$ module over $K$ with base $(\Omega,\mathcal{F},P)$ and $\xi = \bigvee \{ \|x\| : x \in E \}$, since $E$ is an $L^0(\mathcal{F},K)$-- module, for an arbitrarily chosen representative $\xi^0$ of $\xi$, it always holds that $P\{ \omega \in \Omega ~|~ \xi^0(\omega) = + \infty$ or $0 \} = 1$, let $supp(E) = \{ \omega \in E~|~ \xi^0(\omega) = + \infty \}$, then $supp(E)$ is called the support of $E$ ($supp(E)$ is unique a.s.). If $P(supp(E)) =1$, then $E$ is said to have full support.

\begin{lemma}\label{lemma3.4}
Let $(E,\|\cdot\|)$ be a $\mathcal{T}_{\varepsilon,\lambda}$--complete $RN$ module over $K$ with base $(\Omega,\mathcal{F},P)$ such that $P(sup p(E))>0($otherwise, $E$ only consists of the null $\theta)$ and $A\in \mathcal{F}$ such that $A\subset supp(E)$ and $P(A)>0$. Then we have the following assertions:
\begin{enumerate}[(1)]
\item There exists $x\in E$ such that $\|x\|=\tilde{I}_A$ and $x=\tilde{I}_Ax$.
\item Let $E_A=\tilde{I}_AE:=\{\tilde{I}_Ax: x\in E\}$ and $\|\cdot\|_A: E_A\rightarrow L^0_+(A\cap \mathcal{F})$ be defined by $\|y\|_A:=$ the limitation of $\|y\|$ to $A$, where $A\cap \mathcal{F}=\{A\cap F: F\in \mathcal{F}\}$. Then $(A,A\cap \mathcal{F},P_A)$ is a probability space, where $P_A(A\cap F)=P(A\cap F)/P(A)$ for any $F\in \mathcal{F}$. Further, for any given $\xi \in L^0(A\cap F,K)$, let $\xi^0$ be an arbitrarily chosen representative of $\xi$, $\xi^0_A: \Omega\rightarrow K$ is defined by $\xi^0_A(\omega)=\xi^0(\omega)$ for any $\omega \in A$, and $\xi^0_A(\omega)=0$ for any $\omega \in A^c$, and $\xi_A$ stands for the equivalence of $\xi^0_A$. Define $\pi: L^0(A\cap \mathcal{F},K)\rightarrow \tilde{I}_AL^0(\mathcal{F},K):= \{\tilde{I}_A\eta: \eta\in L^0(\mathcal{F},K)\}$ by $\pi(\xi)=\xi_A$ for any $\xi\in L^0(A\cap \mathcal{F},K)$, and define the module multiplication operation $*: L^0(A\cap\mathcal{F},K)\times E_A\rightarrow E_A$ by $\xi * x=\pi(\xi)\cdot x$, then it is easy to verify the following facts: $(i)$ $\pi$ is an algebraic isomorphism from $L^0(A\cap\mathcal{F},K)$ onto $\tilde{I}_AL^0(\mathcal{F},K)$; $(ii)$ since $E_A$ is a $\tilde{I}_AL^0(\mathcal{F},K)$--module under the original module multiplication, $E_A$ is an $L^0(A\cap\mathcal{F},K)$--module under the induced module multiplication $*$; $(iii)$ $(E_A,\|\cdot\|_A)$ is a $\mathcal{T}_{\varepsilon,\lambda}$--complete $RN$ module over $K$ with base $(A,A\cap\mathcal{F},P_A)$ such that $(E_A,\|\cdot\|_A)$ has full support, where the module structure of $E_A$ is defined by the module multiplication $*$.
\end{enumerate}
\end{lemma}

\begin{proof}
\begin{enumerate}[(1)]
\item It follows from \cite[Lemma3.1]{Guo08} that there exists $z\in E$ such that $\|z\|=\tilde{I}_{supp(E)}$, then $x:=\tilde{I}_Az$ satisfies (1).
\item It is obvious that $\|x\|_A=1$, and hence $(E_A,\|\cdot\|_A)$ has full support.
\end{enumerate}
\end{proof}

\par
By $(2)$ of Lemma \ref{lemma3.4}, we can, without loss of generality, assume that $(E,\|\cdot\|)$ has full support $($otherwise, we consider $(E_{supp(E)},\|\cdot\|_{supp(E)})$ in the place of $(E,\|\cdot\|) )$ in the following study of random uniform convexity. Geometry of complete $RN$ modules began with Guo and Zeng's paper \cite{GZ10}, we employ the following notation: \\
$\varepsilon_\mathcal{F}[0,2]=\{\varepsilon\in L^0_{++}(\mathcal{F})~|$ there exists a positive number $\lambda$ such that $\lambda\leq\varepsilon\leq 2\}$.\\
$\delta_{\mathcal{F}}[0,1]=\{\delta\in L^0_{++}(\mathcal{F})~|$ then exists a positive number $\eta$ such that $\eta \leq \delta\leq 1\}$.

\par
For a $\mathcal{T}_{\varepsilon,\lambda}$--complete $RN$ module $(E,\|\cdot\|)$ over $K$ with base $(\Omega,\mathcal{F},P)$, $A_x=(\|x\|>0)$ for any $x\in E$ and $B_{x,y}=A_x\cap A_y\cap A_{x-y}$ for any $x$ and $y$ in $E$.

\begin{definition}\label{definition3.5}\cite{GZ10}
Let $(E,\|\cdot\|)$ be a $\mathcal{T}_{\varepsilon,\lambda}$--complete $RN$ module over $K$ with base $(\Omega,\mathcal{F},P)$ such that $(E,\|\cdot\|)$ has full support. $E$ is said to be random uniformly convex if for each $\varepsilon\in \varepsilon_\mathcal{F}[0,2]$ there exists $\delta\in \delta_{\mathcal{F}}[0,1]$ such that $\|x-y\|\geq \varepsilon$ on $D$ always implies $\|x+y\|\leq 2(1-\delta)$ on $D$ for any $x$ and $y\in U(1)$ and any $D\in \mathcal{F}$ such that $D\subset B_{x,y}$ and $P(D)>0$, where $U(1)=\{z\in E~|~\|z\|\leq 1\}$, called the random closed unit ball of $E$.
\end{definition}

\par
Definition \ref{definition3.5} is different from the original Definition4.1 of \cite{GZ10}, but they are equivalent since they are both equivalent to $(5)$ of \cite[Theorem4.1]{GZ10}. The definition of random uniform convexity is fruitful since it can make us prove that $(E,\|\cdot\|)$ is random uniformly convex if and only if $L^p(E)$ is uniformly convex for some $p$ such that $1<p<+\infty$ in \cite{GZ10,GZ12}. Thus $L^p_{\mathcal{F}_0}(\mathcal{F}_T,R^d)$ is random uniformly convex for any $p$ such that $1<p<+\infty$ since when $E=L^p_{\mathcal{F}_0}(\mathcal{F}_T,R^d)$, ~$L^p(E)$ is exactly $L^p(\mathcal{F}_T,R^d)$, which is always a uniformly convex Banach space. Besides, it is also proved that $L^0(\mathcal{F},B)$ is random uniformly convex if and only if $B$ is a uniformly convex Banach space in \cite{GZ10}. Finally, Zeng's paper \cite{Z13} gave a nice discussion on random convexity modulus.

\par
\begin{theorem}\label{theorem3.6}
Let $(E,\|\cdot\|)$ be a $\mathcal{T}_{\varepsilon,\lambda}$--complete $RN$ module over $K$ with base $(\Omega,\mathcal{F},P)$ such that $(E,\|\cdot\|)$ has full support and is random uniformly convex. Then every $\mathcal{T}_{\varepsilon,\lambda}$--closed $L^0$--convex subset $G$ of $E$ has random normal structure.
\end{theorem}

\begin{proof}
For any given $\mathcal{T}_{\varepsilon,\lambda}$--closed and a.s. bounded $L^0$--convex subset $H$ of $G$ such that $D(H):=\bigvee\{\|x-y\|: x,y\in H\}>0$, we can, without loss of generality, assume that $D(H)>0$ on $\Omega$, namely $D(H)\in L^0_{++}(\mathcal{F})($otherwise, let $A=(D(H)>0)$, we can consider $(E_A,\|\cdot\|_A)($see Lemma \ref{lemma3.4}$)$, then for $E_A$ we have: $D(H)\in L^0_{++}(A\cap\mathcal{F}))$.
\par
Since $E\times E$ is a $\mathcal{T}_{\varepsilon,\lambda}$--complete $RN$ module with the $L^0$--norm $\|(x,y)\|=\|x\|+\|y\|$ for any $(x,y)\in E\times E$, $H\times H$ is $\mathcal{T}_{\varepsilon,\lambda}$--closed and $L^0$--convex, it is very easy to see $H\times H$ is stable. Further, the mapping $f: H\times H \rightarrow L^0(\mathcal{F})$ defined by $f(x,y)=\|x-y\|$ for any $x$ and $y$ in $H$ is clearly $L^0$--Lipschitzian, then $f(H\times H)=\{\|x-y\|: (x,y)\in H\times H\}$ is stable by Lemma\ref{lemma2.11}. In Lemma\ref{lemma2.10}, taking $\varepsilon=\frac{D(H)}{2}$ yields that there exists $(x,y)\in H\times H$ such that $\|x-y\|>D(H)-\frac{D(H)}{2}=\frac{1}{2}D(H)$ on $\Omega$.

\par
Let $z=\frac{1}{2} (x+y)$, then $z-h=\frac{1}{2}[(x-h)+(y-h)]$ for each $h\in H$, further let $u_1=\frac{x-h}{D(H)}$ and $u_2=\frac{y-h}{D(H)}$, then $\|u_1\|\leq 1$, $\|u_2\|\leq 1$ and $\|u_2-u_1\|=\frac{\|x-y\|}{D(H)}\geq \frac{1}{2}$. In Definition\ref{definition3.5}, taking $\varepsilon=\frac{1}{2}$ yields that there exists $\delta\in \delta_{\mathcal{F}}[0,1]$ such that $\|v_1-v_2\|\geq \varepsilon$ on $D$ implies $\frac{\|v_1+v_2\|}{2}\leq 1-\delta$ on $D$ for any $v_1$ and $v_2$ in $U(1):= \{z\in E~|~\|z\|\leq 1\}$ and any $D\in \mathcal{F}$ such that $D\subset B_{v_1,v_2}$ and $P(D)>0$. Taking $v_1=u_1$ and $v_2=u_2$, then at this time $A_{u_1-u_2}:= (\|u_1-u_2\|>0)=\Omega$ and hence $B_{u_1,u_2}=A_{u_1}\cap A_{u_2}\cap A_{u_1-u_2}=A_{u_1}\cap A_{u_2}$. Further, let $D=A_{u_1}\cap A_{u_2}$, then it is obvious that $\|u_1-u_2\|\geq \frac{1}{2}$ on $D$, then $\frac{\|u_1+u_2\|}{2}\leq 1-\delta$ on $D$. On the other hand, it is also obvious that $\frac{\|u_1+u_2\|}{2}\leq \frac{1}{2}=1-\frac{1}{2}$ on $(A_{u_1}\cap A_{u_2})^c$.

\par
Finally, let $\hat{\delta}=\delta\wedge \frac{1}{2}$, then $\hat{\delta}\in \delta_{\mathcal{F}}[0,1]$ is such that $\frac{\|u_1+u_2\|}{2}\leq 1-\hat{\delta}$, so $\frac{\|z-h\|}{D(H)}=\|\frac{u_1+u_2}{2}\|\leq 1-\hat{\delta}$, i.e., $\|z-h\|\leq(1-\hat{\delta})D(H)$ for each $h\in H$, which shows that $\bigvee\{\|z-h\|: h\in H\}\leq (1-\hat{\delta})D(H)<D(H)$ on $\Omega$, namely $z$ is a nondiametral point of $H$.
\end{proof}

\begin{corollary}\label{corollary3.7}
Let $(B,\|\cdot\|)$ be a uniformly convex Banach space, $V$ a closed convex subset of $B$ and $L^0(\mathcal{F},V)$ the set of equivalence classes of $V$--valued strongly $\mathcal{F}$--measurable functions on $(\Omega, \mathcal{F}, P)$. Then $L^0(\mathcal{F},V)$, as a $\mathcal{T}_{\varepsilon,\lambda}$--closed $L^0$--convex subset of $L^0(\mathcal{F},B)$, has random normal structure.
\end{corollary}

\begin{proof}
It follows from Theorem\ref{theorem3.6} and the fact that $L^0(\mathcal{F},B)$ is random uniformly convex.
\end{proof}
\par
The idea of proof of Theorem \ref{theorem3.8} is very similar to that of the classical Browder--Kirk's fixed point theorem \cite[p.486]{K83}, but it is nontrivial since it can apply to complete $RN$ modules, which are a kind of more general spaces than Banach spaces and in particular include a wide class of not locally convex spaces.
\begin{theorem}\label{theorem3.8}
Let $(E,\|\cdot\|)$ be a $\mathcal{T}_{\varepsilon,\lambda}$--complete $RN$ module over $K$ with base $(\Omega, \mathcal{F}, P)$ and $G$ a $\mathcal{T}_{\varepsilon,\lambda}$--closed $L^0$--convex subset of $E$ such that $G$ is an $L^0$--convexly compact subset with random normal structure. Then a nonexpansive mapping $T$ form $G$ to $G$ $($namely $\|T(x)-T(y)\|\leq \|x-y\|$ for any $x$ and $y\in G)$ has a fixed point.
\end{theorem}

\begin{proof}
 Let $M=\{H: H$ is a nonempty $\mathcal{T}_{\varepsilon,\lambda}$--closed $L^0$--convex subset of $G$ such that $T(H)\subset H\}$, then $M$ is a partially ordered set under the usual inclusion relation. By the $L^0$--convex compactness of $G$, the Zorn's Lemma guarantees the existence of a minimal element $H_0$ in $M$.
\par
Lemma2.19 of \cite{GZWW17} shows that $G$ is a.s. bounded, so is $H_0$. We assert that $H_0$ is a singleton: otherwise, $D(H_0):= \bigvee\{\|h_1-h_2\|: h_1,h_2\in H_0\}>0$, let $A=(D(H_0)>0)$, then $P(A)>0$. Since $G$ has random normal structure, there exists $z\in H_0$ such that $r:= \bigvee\{\|z-h\|: h\in H_0\}<D(H_0)$ on $A$, let $H_1=\{y\in H_0: H_0\subset B(y,r)\}$, where $B(y,r)=\{x\in E: \|x-y\|\leq r\}$, then $H_1\neq \emptyset$ since $z\in H_1$. Let $y$ be any given element in $H_1$, then $T(H_0)\subset B(T(y),r)$ by the nonexpansiveness of $T$, and hence $\overline{Conv_{L^0}(T(H_0))}\subset B(T(y),r)$, where $\overline{Conv_{L^0}(T(H_0))}$ stands for the $\mathcal{T}_{\varepsilon,\lambda}$--closed $L^0$--convex hull of $T(H_0)$. Since $H_0$ is a $\mathcal{T}_{\varepsilon,\lambda}$--closed $L^0$--convex subset of $G$ such that $T(H_0)\subset H_0$, $T(\overline{Conv_{L^0}(T(H_0))})\subset \overline{Conv_{L^0}(T(H_0))}$, then $\overline{Conv_{L^0}(T(H_0))}=H_0$ by the minimality of $H_0$, so that $T(y)\in H_1$, that is to say, $T(H_1)\subset H_1$, then $H_1=H_0$, and thus $\|u-v\|\leq r$ for any $u$ and $v$ in $H_0$, which contradicts the fact that $r<D(H_0)$ on $A$. Thus $H_0$ consists of a single element, which is a fixed point of $T$.
\end{proof}

\begin{corollary}\label{corollary3.9}
Let $(E,\|\cdot\|)$ be a $\mathcal{T}_{\varepsilon,\lambda}$--complete random uniformly convex $RN$ module, $G$ an a.s. bounded $\mathcal{T}_{\varepsilon,\lambda}$--closed $L^0$--convex subset of $E$ and $T: G\rightarrow G$ a nonexpansive mapping. Then $T$ has a fixed point.
\end{corollary}

\begin{proof}
Theorem4.6 of \cite{GZ10} shows that $E$ is random reflexive and further Corollary 2.23 of \cite{GZWW17} shows that $G$ is $L^0$--convexly compact, which, combined with Theorem\ref{theorem3.6}, implies that $T$ has a fixed point by Theorem\ref{theorem3.8}.
\end{proof}

\begin{remark}\label{remark3.10}
Corollay\ref{corollary3.9} is a natural generalization of a fixed point theorem independently due to F.E.Browder\cite{B65} and D.G\"{o}hde \cite{G65}. Besides, in Corollay\ref{corollary3.9}, if $G$ is only assumed to be a $\mathcal{T}_{\varepsilon,\lambda}$--closed $L^0$--convex subset of $E$ such that $T(G)$ is a.s. bounded, then $T$ still has a fixed point: in fact, one only needs to consider $\overline{Conv_{L^0}(T(G))}$, then it is a.s. bounded, $\mathcal{T}_{\varepsilon,\lambda}$--closed, $L^0$--convex and mapped by $T$ into itself.
\end{remark}

\section{$BSEs $ of the form$(\ref{1.3})$ and fixed point problems in the $\mathcal{T}_{\varepsilon,\lambda}$--complete $RN$ module $L^p_{\mathcal{F}_0}(\mathcal{F}_T,R^d)$ for $p\in (1, +\infty]$} \label{section4}

\par
Let $V$ belong to $L^p_{\mathcal{F}_0}(\mathcal{F}_T,R^d)$, since $L^p_{\mathcal{F}_0}(\mathcal{F}_T,R^d)=L^0(\mathcal{F}_0)\cdot L^p(\mathcal{F}_T,R^d)$, $V$ can be written as $V=\xi\cdot V_1$ for some $\xi\in L^0(\mathcal{F}_0)$ and $V_1\in L^p(\mathcal{F}_T,R^d)$. As pointed out by Cheridito and Nam in \cite{CN17}, one can see from Jensen's inequality that $y^{V_1}:= E_0(V_1)=E[V_1~|~\mathcal{F}_0]\in L^p(\mathcal{F}_0,R^d)$ and from Doob's $L^p$--maximal inequality that $M^{V_1}_t:= E_0(V_1)-E_t(V_1)$ is in $M^p_0$, then $y^V:= E_0(V)=\xi\cdot E_0(V_1)\in L^0(\mathcal{F}_0,R^d)$ and $M^V_t=\xi(E_0(V_1)-E_t(V_1))\in M^p_{\mathcal{F}_0,0}$.

\par
Let us recall that $F: S^p_{\mathcal{F}_0}\times M^p_{\mathcal{F}_0,0}\rightarrow S^p_{\mathcal{F}_0,0}$ satisfies condition $(S)$, namely for any $y\in L^0(\mathcal{F}_0,R^d)$ and $M\in M^p_{\mathcal{F}_0,0}$ the stochastic equation $Y_t=y-F_t(Y,M)-M_t$ for any $t\in [0,T]$ has a unique solution $Y\in S^p_{\mathcal{F}_0}$. If $F$ satisfies condition $(S)$, we denote by $Y^V$ the unique solution of the equation $Y_t=y^V-F_t(Y,M^V)-M^V_t$.

\par
A $BSE(\ref{1.3})$ depends on the generator $F$ and terminal condition $\xi$, if $F$ satisfies condition $(S)$ and $\xi$ belongs to $L^p_{\mathcal{F}_0}(\mathcal{F}_T,R^d)$ then the pair$(F,\xi)$ define a mapping $G: L^p_{\mathcal{F}_0}(\mathcal{F}_T,R^d)\rightarrow L^p_{\mathcal{F}_0}(\mathcal{F}_T,R^d)$ by $G(V)= \xi+F_T(Y^V,M^V)$ for any $V\in L^p_{\mathcal{F}_0}(\mathcal{F}_T,R^d)$.

\par
To relate solutions of the $BSE(\ref{1.3})$ to fixed points of $G$, we define the two mappings $\pi: S^p_{\mathcal{F}_0}\times M^p_{\mathcal{F}_0,0}\rightarrow L^p_{\mathcal{F}_0}(\mathcal{F}_T,R^d)$ and $\phi: L^p_{\mathcal{F}_0}(\mathcal{F}_T,R^d)\rightarrow S^p_{\mathcal{F}_0}\times M^p_{\mathcal{F}_0,0}$ as follows: \\
$\pi(Y,M)=Y_0-M_T$ for any $(Y,M)\in S^p_{\mathcal{F}_0}\times M^p_{\mathcal{F}_0,0}$; \\
$\phi(V)=(Y^V,M^V)$ for any $V\in L^p_{\mathcal{F}_0}(\mathcal{F}_T,R^d)$.

\par
If it is observed that $M^V$ is always a generalized martingale for any $V\in L^p_{\mathcal{F}_0}(\mathcal{F}_T,R^d)$, then Theorem\ref{theorem4.1} and Corollary\ref{corollary4.2} below can be proved in the same way that Theorem2.3 and Corollary2.4 of \cite{CN17} was proved in \cite{CN17}, so their proofs are omitted.

\begin{theorem}\label{theorem4.1}
Assume that $F$ satisfies condition $(S)$. Then the following hold:
\begin{enumerate}[(a)]
\item $V=(\pi\circ\phi)(V)$ for all $V\in L^p_{\mathcal{F}_0}(\mathcal{F}_T,R^d)$. In particular, $\phi$ is injective.
\item If $V\in L^p_{\mathcal{F}_0}(\mathcal{F}_T,R^d)$ is a fixed point of $G$, then $\phi(V)$ is a solution of the $BSE(\ref{1.3})$.
\item If $(Y,M)\in S^p_{\mathcal{F}_0}\times M^p_{\mathcal{F}_0,0}$ solves the $BSE(\ref{1.3})$, then $\pi(Y,M)$ is a fixed point of $G$ and $(Y,M)=(\phi\circ\pi)(Y,M)$.
\item $V$ is a unique fixed point of $G$ in $L^p_{\mathcal{F}_0}(\mathcal{F}_T,R^d)$ if and only if $\phi(V)$ is a unique solution of the $BSE(\ref{1.3})$ in $S^p_{\mathcal{F}_0}\times M^p_{\mathcal{F}_0,0}$.
\end{enumerate}
\end{theorem}
\par
In the special case when $F$ does not depend on $Y$, condition $(S)$ holds trivially, and it is enough to find a fixed point of the mapping $G_0(V):= G(V)-E_0(G(V))$ in the $\mathcal{T}_{\varepsilon,\lambda}$--closed submodule $L^p_{\mathcal{F}_0,0}(\mathcal{F}_T,R^d):=\{V\in L^p_{\mathcal{F}_0}(\mathcal{F}_T,R^d)~| \\
~E_0(V)=0\}$.

\begin{corollary}\label{corollary4.2}
If $F$ does not depend on $Y$, the following hold:
\begin{enumerate}[(a)]
\item If $V\in L^p_{\mathcal{F}_0,0}(\mathcal{F}_T,R^d)$ is a fixed point of $G_0$, then the processes $Y_t:= E_0(\xi)+E_0(F_T(M))-F_t(M)-M_t$ and $M_t:=-E_t(V)$ form a solution of the $BSE(\ref{1.3})$ in $S^p_{\mathcal{F}_0}\times M^p_{\mathcal{F}_0,0}$.
\item If $(Y,M)\in S^p_{\mathcal{F}_0}\times M^p_{\mathcal{F}_0,0}$ solves the $BSE(\ref{1.3})$, then $-M_T$ is a fixed point of $G_0$.
\item $V$ is a unique fixed point of $G_0$ in $L^p_{\mathcal{F}_0,0}(\mathcal{F}_T,R^d)$ if and only if the pair $(Y,M)$ given by $Y_t:= E_0(\xi)+E_0(F_T(M))-F_t(M)-M_t$ and $M_t:=-E_t(V)$  is a unique solution of the $BSE(\ref{1.3})$ in $S^p_{\mathcal{F}_0}\times M^p_{\mathcal{F}_0,0}$.
\end{enumerate}
\end{corollary}

\par
Lemma\ref{lemma4.3} below provides a sufficient condition for $F$ to satisfy condition $(S)$. For $(Y,M)\in S^p_{\mathcal{F}_0}\times M^p_{\mathcal{F}_0,0}$ and $k\in N$, define $F^{(k)}_t(Y,M)=F_t(Y^{(k,M)},M)$, where $Y^{(k,M)}$ is recursively given by $Y^{(1,M)}=Y$ and $Y_t^{(k,M)}=Y_0-F_t(Y^{(k-1,M)},M)-M_t$, for $k\geq 2$. Then, for a positive integer--valued $\mathcal{F}_0$--measurable random variable $L: (\Omega,\mathcal{F}_0,P)\rightarrow N$, we can correspondingly have $F^{(L)}_t(Y,M)=\sum^{\infty}_{k=1}\tilde{I}_{(L=k)}F^{(k)}_t(Y,M)$.

\par
Since $S^p_{\mathcal{F}_0}, M^p_{\mathcal{F}_0,0}$ and $S^p_{\mathcal{F}_0,0}$ are all stable, $S^p_{\mathcal{F}_0}\times M^p_{\mathcal{F}_0,0}$ is stable. According to the definition before Lemma\ref{lemma2.11}, let us recall that $F: S^p_{\mathcal{F}_0}\times M^p_{\mathcal{F}_0,0}\rightarrow S^p_{\mathcal{F}_0,0}$ is stable if and only if
\begin{equation}\label{4.1}
F(\sum^{\infty}_{n=1}\tilde{I}_{A_n}(Y^{(n)},M^{(n)}))=\sum^{\infty}_{n=1}\tilde{I}_{A_n}F(Y^{(n)},M^{(n)})
\end{equation}
 for any sequence $\{(Y^{(n)}, M^{(n)}),n\in N\}$ in $S^p_{\mathcal{F}_0}\times M^p_{\mathcal{F}_0,0}$ and any countable partition $\{A_n, n\in N\}$ of $\Omega$ to $\mathcal{F}_0$. The condition that $F$ is stable is not too restrictive, for example, when $F$ is $L^0$--Lipschitzian $F$ is stable by Lemma\ref{lemma2.11}, when the generator $F$ is of the form: $F_t(Y,M)=\int^t_0 f(s,Y_s,M_s)ds$ or $F_t(M)=\int^t_0\int_{[0,s]}g(s-r,Z^M_{s-r}, U^M_{s-r})v(dr)ds$, it is easy to see that $F$ ia always stable.
\par
By the way, we point out that the convention that $\{ \tilde{I}_{A_n}, n \in N\}$ occurs in the expression (\ref{4.1}) comes from the theory of $RN$ modules since we would like to emphasize equivalence classes, whereas the scholars working in the field of probability theory or mathematical finance prefer the following expression (\ref{4.2}) to (\ref{4.1}):
\begin{equation}\label{4.2}
  F(\sum_{n=1}^\infty I_{A_n} (Y^{(n)}, M^{(n)})) = \sum_{n=1}^\infty I_{A_n}F(Y^{(n)}, M^{(n)}),
\end{equation}
since they prefer to speak of random processes rather than their equivalence classes, where the equality in (\ref{4.2}) means that the two sides of (\ref{4.2}) are stochastically equivalent random processes, so the two conventions express the same thing.

\begin{lemma}\label{lemma4.3}
Let $p$ belong to $(1, +\infty]$. If for any given $y\in L^0(\mathcal{F}_0,R^d)$ and any given $M\in M^p_{\mathcal{F}_0,0}$, there exist a positive integer--valued $\mathcal{F}_0$--measurable random variable $L$ and $C\in L^0_+(\mathcal{F})$ with $C<1$ on $\Omega($i.e., $C(\omega)<1$ a.s.$)$ such that $|||F^{(L)}(Y,M)-F^{(L)}(Y',M)|||_p\leq C|||Y-Y'|||_p$ for all $Y,Y'\in S^p_{\mathcal{F}_0}$ with $Y_0=Y_0'=y$, then when $F$ is stable the $SDE: Y_t=y-F_t(Y,M)-M_t, \forall t\in [0,T]$, has a unique solution. Specially, when $L$ is a constant, the conclusion still holds without the stability of $F$.
\end{lemma}

\begin{proof}
Define $H=\{Y\in S^p_{\mathcal{F}_0}~|~Y_0=y\}$ and $J: H\rightarrow H$ by $(J(Y))_t=y-F_t(Y,M)-M_t$ for any $Y\in H$ and any $t\in [0,T]$, then $H$ is a $\mathcal{T}_{\varepsilon,\lambda}$--closed stable subset of $S^p_{\mathcal{F}_0}($and hence also $\mathcal{T}_{\varepsilon,\lambda}$--complete$)$ and $(J^{(k)}(Y))_t=y-F^{(k)}_t(Y,M)-M_t$ for any $Y\in H, t\in [0,T]$ and $k\in N$, and thus $(J^{(L)}(Y))_t=y-F^{(L)}_t(Y,M)-M_t$, where $J^{(k)}$ is the $k$--th iteration of $J$.
\par
When $F$ is stable, $J$ is obviously stable, too, then $J$ has a unique fixed point in $H$ by Theorem\ref{theorem2.13}, namely, the $SDE: Y_t=y-F_t(Y,M)-M_t, \forall t\in [0,T]$, has a unique solution. The final part comes from Proposition\ref{proposition2.12}.
\end{proof}

\begin{proposition}\label{proposition4.4}
If $F$ is stable and satisfies condition $(S)$, then $G$ is stable.
\end{proposition}

\begin{proof}
Since $G(V)=\xi+F_T(Y^V,M^V)$ for any $V\in L^p_{\mathcal{F}_0}(\mathcal{F}_T,R^d)$, it is obvious that the mapping sending $V\in L^p_{\mathcal{F}_0}(\mathcal{F}_T,R^d)$ to $M^V\in M^p_{\mathcal{F}_0,0}$ is stable, then for the proof of stability of $G$ it suffices to prove that the mapping sending $V$ to $Y^V$ is stable.
\par
For any given countable partition $\{A_n,n\in N\}$ of $\Omega$ to $\mathcal{F}_0$ and any given sequence $\{V_n,n\in N\}$ in $L^p_{\mathcal{F}_0}(\mathcal{F}_T,R^d)$, let $V=\sum^{\infty}_{n=1}\tilde{I}_{A_n}V_n$. Then $Y^{V_n}_t=E_0(V_n)-F_t(Y^{V_n}, M^{V_n})-M_t^{V_n}$ for each $n\in N$ and each $t\in [0,T]$, and thus $(\sum^{\infty}_{n=1}\tilde{I}_{A_n}Y^{V_n})_t~=E_0(\sum^{\infty}_{n=1}\tilde{I}_{A_n}V_n)~-F_t(\sum^{\infty}_{n=1}\tilde{I}_{A_n}Y^{V_n}, \sum^{\infty}_{n=1}\tilde{I}_{A_n}M^{V_n})~- \\ M_t^{\sum^{\infty}_{n=1}\tilde{I}_{A_n}V_n}=E_0(V)-F_t(\sum^{\infty}_{n=1}\tilde{I}_{A_n}Y^{V_n},M^V)-M^V_t$, it follows from condition $(S)$ that $Y^V=\sum^{\infty}_{n=1}\tilde{I}_{A_n}Y^{V_n}$.
\end{proof}

\section{Existence and uniqueness of solutions under $L^0$--Lipschitz assumptions} \label{section5}

Lipschitz conditions are usually imposed on the generators or drivers to guarantee existence and uniqueness of solutions, where Lipschitz coefficients are often nonnegative constants. However, El Karoui and Huang \cite{EH97} and Bender and Kohlmann \cite{BK00} began to study existence and uniqueness of solutions of $BSDEs$ under stochastic Lipschitz conditions, where stochastic Lipschitz coefficients are a nonnegative adapted process, whose results was recently used by Sun and Guo in \cite{SG18} to solve optimal mean--variance investment and reinsurance problem for an insurer with stochastic volatility. Motivated by this, this paper formulates our results on existence and uniqueness of solutions under a kind of random Lipschitz assumptions, namely $L^0$--Lipschitz assumptions. But our study is essentially different from \cite{EH97,BK00}, since our terminal condition belongs to $L^p_{\mathcal{F}_0}(\mathcal{F}_T,R^d)$ whereas the terminal condition in \cite{EH97,BK00} belongs to a Banach space. ``$L^0$--Lipschitz" means that Lipschitz coefficients are $\mathcal{F}_0$--measurable nonnegative random variables. We employ the $L^0$--Lipschitz assumptions mainly because our estimates are based on the $L^0$--norms $($more precisely, the conditional $L^p$--norms with respect to $\mathcal{F}_0)$ $|||\cdot|||_p$ on $S^p_{\mathcal{F}_0}$ or $L^p_{\mathcal{F}_0}(\mathcal{F}_T,R^d)$ rather than the usual $L^p$--norms on $S^p$ or $L^p(\mathcal{F}_T,R^d)$ as in \cite{CN17}, in fact, one can easily construct a great number of examples satisfying $L^0$--Lipschitz assumptions. Besides, $L^0$--Lipschitz assumptions also allow us to consider some time--delayed $BSDEs$ involving a random measure, for example, see Proposition\ref{proposition5.10} below.
\par
This section provides the conditional analogues to the main results of \cite{CN17} on existence and uniqueness of solutions under Lipschitz assumptions. When $\mathcal{F}_0$ is trivial $($namely, $\mathcal{F}_0$ only consists of the events $A$ with $P(A)=0$ or 1$)$, our results and the corresponding results of \cite{CN17} coincide, whereas they do not include each other in general.

\par
For the sake of convenience, for $p\in(1.+\infty]$ this paper always uses $|||\cdot|||_p$ for the conditional $p$--norms on $S^p_{\mathcal{F}_0},M^p_{\mathcal{F}_0}$ and $L^p_{\mathcal{F}_0}(\mathcal{F}_T,R^d)$, respectively, and $||\cdot||_p$ for the usual $p$--norms on $S^p,M^p$ and $L^p(\mathcal{F}_T,R^d)$, respectively. We first give the following three lemmas which expound upon some theoretical issues.

\par
Lemma\ref{lemma5.1} below is the conditional version of the classical Doob's $L^p$--maximal inequality.
\begin{lemma}\label{lemma5.1}
Let $C_p=\frac{p}{p-1}$ for $1<p<+\infty$ and $C_{\infty}=1$. Then for any $M\in M^p_{\mathcal{F}_0}$ it always holds that $|||M|||_p\leq C_p|||M_T|||_p$, in particular $|||M|||_p=|||M_T|||_p$ for $p=+\infty$.
\end{lemma}

\begin{proof}
We first consider the case when $1<p<+\infty$. Since $M=\xi\cdot \overline{M}$ for some $\xi\in L^0(\mathcal{F}_0)$ and some $\overline{M}\in M^p, |||M|||_p=|\xi|~ |||\overline{M}|||_p$ and $|||M_T|||_p=|\xi|~ |||(\overline{M})_T|||_p$, it suffices to prove that $|||\overline{M}|||_p\leq C_p |||\overline{M}_T|||_p$. For a discrete martingale, the conditional version of Doob's inequality is well known, see \cite[Lemma]{Chang94}. For $\overline{M}\in M^p$, we can similarly proceed as follows as in \cite{Chang94}: since for each $A\in \mathcal{F}_0, I_A\cdot \overline{M}$ still belongs to $M^p$, by the classical Doob's inequality $( \int_{\Omega}I_A \sup_{0\leq t\leq T} |\overline{M}_t|^pdP)^{1/p}\leq \frac{p}{p-1}( \int_{\Omega}I_A |\overline{M}_t|^pdP)^{1/p}$, then $E[(\sup_{0\leq t\leq T}|\overline{M}_t| \\
)^p ~|~\mathcal{F}_0]^{1/p}\leq \frac{p}{p-1}E[|\overline{M}_t|^p~|~\mathcal{F}_0]^{1/p}$, namely $|||\overline{M}|||_p\leq C_p|||\overline{M}_T|||_p$.
\par
For $p=+\infty$, if $\xi\in L^0_+(\mathcal{F})$ is such that $|M_T|\leq\xi$, since $M$ is a generalized martingale, $M_t=E[M_T~|~\mathcal{F}_t]$ a.s. for each fixed $t\in [0,T]$, so $|M_t(\omega)|\leq \xi(\omega)$ for almost all $\omega\in \Omega$ and any $t\in [0,T]$, then $\sup_{0\leq t\leq T}|M_t|\leq\xi$ follows from the fact that $M$ is right continuous with left limit, namely $|||M|||_{\infty}\leq \xi$. By the definition of $|||\cdot|||_{\infty}$, one can have that $|||M|||_{\infty}\leq |||M_T|||_{\infty}$, as for the converse inequality, it is obvious that $|||M_T|||_{\infty}\leq |||M|||_{\infty}$.
\end{proof}

\par
Let $(\Omega,\mathcal{F},P)$ be a probability space and $\mathcal{F}_0$ a sub $\sigma$--algebra of $\mathcal{F}$. Further, for an extended $\mathcal{F}$--measurable nonnegative random variable $\xi: \Omega \rightarrow [0,+\infty]$, whose generalized conditional mathematical expectation $E[\xi~|~\mathcal{F}_0]$ with respect to $\mathcal{F}_0$ is usually defined as the almost sure limit of the nondecreasing sequence $\{E[\xi\wedge n~|~\mathcal{F}_0]: n\in N\}$ of $\mathcal{F}_0$--measurable random variables. Clearly, $E[\xi~|~\mathcal{F}_0]$ is an extended $\mathcal{F}_0$--measurable nonnegative random variable, it is also easy to check that $\xi$ is conditionally integrable $($or $\sigma$--integrable$)$ with respect to $\mathcal{F}_0$ iff $E[\xi~|~\mathcal{F}_0]$ is almost surely finite. By the Levy's convergence theorem $\int_A\xi dP=\int _AE[\xi~|~\mathcal{F}_0]dP$ for any $A\in \mathcal{F}_0$.

\par
Lemma\ref{lemma5.2} below is a conditional version of the classical Fubini's theorem. Again, let us recall: let $(\Omega,\mathcal{A},\mu)$ be a $\sigma$--finite measure space, an extended real--valued function $f: \Omega \rightarrow [-\infty,+\infty]$ is said to be $\mu$--measurable if there exists an extended real--valued $\mathcal{A}$--measurable function $g: \Omega \rightarrow [-\infty,+\infty]$ and an $\mathcal{A}$--measurable $\Gamma$ with $\mu(\Gamma)=0$ such that $f(\omega)=g(\omega)$ for each $\omega\in \Omega\setminus \Gamma$.
\begin{lemma}\label{lemma5.2}
Let $(\Omega,\mathcal{F},P)$ be a probability space, $\mathcal{F}_0$ a sub $\sigma$--algebra of $\mathcal{F}$, $(G,\mathcal{E},\mu)$ a $\sigma$--finite measure space and $f: \Omega\times G\rightarrow [-\infty,+\infty]$ an $\mathcal{F}\otimes\mathcal{E}$--measurable function. Denote by $g$ the $\mathcal{F}$--measurable extended nonnegative function defined by $g(\omega)=\int_G|f(\omega,x)| \mu(dx)$ for any $\omega\in\Omega$. $(1)$ If $g$ is conditionally integrable with respect to $\mathcal{F}_0$, then the following three statements holds:
\begin{enumerate}[(i)]
\item For $\mu$--almost all $x$ in $G$, $f(\cdot, x)$ is conditionally integrable with respect to $\mathcal{F}_0$;
\item The function $h: \Omega\times G\rightarrow R$ defined by $h(\omega,x)=E[f(\cdot,x)~|~\mathcal{F}_0](\omega)$ for any $(\omega,x)\in \Omega\times G$ is a $P\otimes \mu$--measurable function on $(\Omega\times G,\mathcal{F}_0\otimes\mathcal{E},P\otimes\mu)$;
\item $E[\int_Gf(\cdot,x)\mu(dx)~|~\mathcal{F}_0]=\int_GE[f(\cdot,x)~|~\mathcal{F}_0]\mu(dx), P$--a.s.
\end{enumerate}
$(2)$ If $f$ is nonnegative, then the extended real--valued function $h: \Omega\times G\rightarrow [0,+\infty]$ given by $h(\omega,x)=E[f(\cdot,x)~|~\mathcal{F}_0](\omega)$ for any $(\omega,x)\in \Omega\times G$ is $P\otimes\mu$--measurable on $(\Omega\times G,\mathcal{F}_0\otimes\mathcal{E},P\otimes\mu)$, and further, one has $E[\int_Gf(\cdot,x)\mu(dx)~|~\mathcal{F}_0] \\
=\int_GE[f(\cdot,x)~|~\mathcal{F}_0]\mu(dx), P$--a.s.
\end{lemma}
\begin{proof}
$(1)$ Since $g$ is conditionally integrable with respect to $\mathcal{F}_0$, there exists some $\mathcal{F}_0$--measurable function $\xi: \Omega\rightarrow R$ such that $\xi(\omega)>0$ for each $\omega\in\Omega$ and $\int_{\Omega}\xi\cdot gdP<+\infty$, namely $\xi\cdot f$ is $P\otimes \mu$--integrable. Let $\bar{f}:=\xi\cdot f$, then $f=\frac{1}{\xi}\cdot \bar{f}$, and thus we only need to prove the lemma for $\bar{f}$, namely we can, without loss of generality, assume that $f$ is $P\otimes \mu$--integrable and nonnegative.
\begin{enumerate}[(i)]
\item By Fubini's theorem, for $\mu$--almost all $x$ in $G$, $f(\cdot,x)$ is $P$--integrable, it is, of course, conditionally integrable with respect to $\mathcal{F}_0$.
\item We can, without loss of generality, assume that $\mu$ is finite. Let $\mathcal{M}=\{H\in \mathcal{F}\otimes \mathcal{E}:$ the function $h: \Omega\times G\rightarrow R$, given by $h(\omega,x)=E[\int_GI_H(\cdot,x)\mu (dx)~|
    \\~\mathcal{F}_0](\omega)$ for any $(\omega,x)\in \Omega\times G$, is $P\otimes \mu$--measurable on $(\Omega\times G,\mathcal{F}_0\otimes\mathcal{E},P\otimes\mu)\}$, then is is clear that $\mathcal{M}\supset\mathcal{F}\times\mathcal{E}:=\{A\times B: A\in\mathcal{F}$ and $B\in \mathcal{E}\}$ and $\mathcal{M}$ is a $\lambda$--class, by the monotone class theorem one has $\mathcal{M}=\mathcal{F}\otimes\mathcal{E}$. Further, let $\{f_n: n\in N\}$ be a sequence of nonnegative simple $\mathcal{F}\otimes\mathcal{E}$--measurable functions on $(\Omega\times G,\mathcal{F}\otimes\mathcal{E},P\otimes\mu)\}$ such that $\int_{\Omega\times G}|f_n-f|d(P\otimes\mu)\rightarrow 0 (n\rightarrow\infty)$, then
\begin{equation}\nonumber
\aligned
&\int_{\Omega}|E[\int_Gf(\cdot,x)\mu (dx)~|~\mathcal{F}_0](\omega)-E[\int_Gf_n(\cdot,x)\mu (dx)~|~\mathcal{F}_0](\omega)]|P(d\omega)\\
&\leq \int_{\Omega}E[\int_G|f(\cdot,x)-f_n(\cdot,x)|\mu(dx)~|~\mathcal{F}_0](\omega)P(d\omega)\\
&=\int_{\Omega\times G}|f_n(\omega,x)-f(\omega,x)|P\otimes \mu (d\omega,dx)\rightarrow 0.\\
\endaligned
\end{equation}
\par
Thus, the function $h: \Omega \times G\rightarrow R$, given by $h(\omega,x)=E[\int_Gf(\cdot,x)\mu (dx)|\mathcal{F}_0] \\
 (\omega)$ for any $(\omega,x)\in \Omega\times G$, must be $P\otimes \mu$--measurable.
\item For any $A\in\mathcal{F}_0,
 \int_AE[\int_Gf(\cdot,x)\mu (dx)~|~\mathcal{F}_0](\omega)P(d\omega)\\
 =\int_A(\int_Gf(\omega,x)\mu(dx))P(d\omega)\\
 =\int_G(\int_Af(\omega,x)P(d\omega))\mu(dx)\\
 =\int_G(\int_AE[f(\cdot,x)~|~\mathcal{F}_0](\omega)P(d\omega))\mu(dx)\\
    =\int_A(\int_GE[f(\cdot,x)~|~\mathcal{F}_0](\omega)\mu(dx))P(d\omega)$,
\end{enumerate}
namely $(iii)$ holds.\\
$(2)$ Since $f$ is nonnegative, for any $n\in N, f\wedge n$ is $P\otimes\mu$--integrable. Let $h(\omega,x)=E[f(\cdot,x)~|~\mathcal{F}_0](\omega)$ and $h_n(\omega,x)=E[f(\cdot,x)\wedge n~|~\mathcal{F}_0](\omega)$ for any $(\omega,x)\in \Omega\times G$, then each $h_n$ is $P\otimes\mu$--measurable on $(\Omega\times G,\mathcal{F}_0\otimes\mathcal{E},P\otimes\mu)$ by $(ii)$ as above, so is $h:=$ the a.s.--limit of $\{h_n:n\in N\}$.
\par
Further, for any $A\in \mathcal{F}_0$, one has:
\begin{equation}\nonumber
\aligned
&\int_A(\int_GE[f(\cdot,x)~|~\mathcal{F}_0](\omega)\mu(dx))P(d\omega)\\
&=\int_A(\int_G(lim_{n\rightarrow\infty}E[f(\cdot,x)\wedge n~|~\mathcal{F}_0](\omega))\mu(dx))P(d\omega)\\
&=\int_A(lim_{n\rightarrow\infty}\int_G(E[f(\cdot,x)\wedge n~|~\mathcal{F}_0](\omega))\mu(dx))P(d\omega)\\
&=lim_{n\rightarrow\infty}\int_{A\times G}(E[f(\cdot,x)\wedge n~|~\mathcal{F}_0](\omega))(P\otimes\mu)(d\omega,dx)\\
&=lim_{n\rightarrow\infty}\int_G(\int_A(E[f(\cdot,x)\wedge n~|~\mathcal{F}_0](\omega))P(d\omega))\mu(dx)\\
&=lim_{n\rightarrow\infty}\int_G(\int_A(f(\omega,x)\wedge n)P(d\omega))\mu(dx)\\
&=\int_{A\times G}f(\omega,x)P\otimes\mu(d\omega,dx)\\
&=\int_A(\int_Gf(\omega,x)\mu(dx))P(d\omega)\\
&=\int_A(E[\int_Gf(\cdot,x)\mu(dx)~|~\mathcal{F}_0](\omega))P(d\omega).
\endaligned
\end{equation}
\end{proof}
\begin{lemma}\label{lemma5.3}
Let $(\Omega,\mathcal{F},P)$ be a probability space and $\mathcal{F}_0$ a sub $\sigma$--algebra of $\mathcal{F}$. Then, for any $V\in L^2_{\mathcal{F}_0}({\mathcal{F}, R^d}), V-E[V~|~\mathcal{F}_0]$ and $E[V~|~\mathcal{F}_0]$ is conditionally orthogonal $($or, random orthogonal$)$, and hence $|||V-E[V~|~\mathcal{F}_0]|||_2\leq |||V|||_2$.
\end{lemma}

\begin{proof}
Since $L^2_{\mathcal{F}_0}(\mathcal{F},R^d)$ is a complete random inner product module $($or a conditional Hilbert space in terms of \cite{HR87} $)$ under the $L^0$--inner product $(\cdot,\cdot): L^2_{\mathcal{F}_0}(\mathcal{F},R^d)\times L^2_{\mathcal{F}_0}(\mathcal{F},R^d)\rightarrow L^0(\mathcal{F}_0)$ defined by $(x,y)=E[\sum^d_{i=1}x_i\cdot y_i~|~\mathcal{F}_0]$ for any $x=(x_1,x_2,\cdot\cdot\cdot, x_d)'$ and $y=(y_1,y_2,\cdot\cdot\cdot, y_d)'\in L^2_{\mathcal{F}_0}(\mathcal{F},R^d)$, where $'$ stands for the transposition. Let $V=(V_1,V_2,\cdot\cdot\cdot, V_d)'\in L^2_{\mathcal{F}_0}(\mathcal{F},R^d)$, then $E[V~|~\mathcal{F}_0]=(E[V_1~|~\mathcal{F}_0],E[V_2~|~\mathcal{F}_0],\cdot\cdot\cdot,E[V_d~|~\mathcal{F}_0])'$, so that $|E[V~|~\mathcal{F}_0]|^2=\sum^d_{i=1}E[V_i~|~\mathcal{F}_0]^2\leq \sum^d_{i=1}E[V_i^2~|~\mathcal{F}_0]=|||V|||^2_2$, which shows that $E[V~|~\mathcal{F}_0]\in L^2_{\mathcal{F}_0}(\mathcal{F},R^d)$ by the definition of $L^2_{\mathcal{F}_0}(\mathcal{F},R^d)$.
\par
$(V-E[V~|~\mathcal{F}_0],E[V~|~\mathcal{F}_0])=E[V\cdot E[V~|~\mathcal{F}_0]~|~\mathcal{F}_0]-E[E[V~|~\mathcal{F}_0]\cdot E[V~|~\mathcal{F}_0]~|~\mathcal{F}_0]=E[V~|~\mathcal{F}_0]\cdot E[V~|~\mathcal{F}_0]-E[V~|~\mathcal{F}_0]\cdot E[V~|~\mathcal{F}_0]=0$, where $\cdot$ stands for the inner product in $R^d$. Thus $V-E[V~|~\mathcal{F}_0]$ and $E[V~|~\mathcal{F}_0]$ is random orthogonal, further $|||V|||^2_2=|||V-E[V~|~\mathcal{F}_0]|||^2_2+|||E[V~|~\mathcal{F}_0]|||^2_2\geq|||V-E[V~|~\mathcal{F}_0]|||^2_2$.
\end{proof}

\par
Let us denote $c_2=\frac{1}{5},c_{\infty}=\frac{1}{4}$ and $c_p=\frac{p-1}{4p-1}$ for $p\in (1,\infty)\setminus \{2\}$. We start with a general existence and uniqueness result---Theorem\ref{theorem5.4} below, which can be regarded as a conditional version of \cite[Theorem3.1]{CN17}. Although the proof of Theorem\ref{theorem5.4} is very similar to that of \cite[Theorem3.1]{CN17}, we would like to give the proof of Theorem\ref{theorem5.4} in detail so that readers can easily understand our idea of conditional version.

\begin{theorem}\label{theorem5.4}
Let $\xi\in L^p_{\mathcal{F}_0}(\mathcal{F}_T,R^d)$ for some $p\in (1,+\infty]$. If $F$ is stable and there exist an integer--valued $\mathcal{F}_0$--measurable random variable $L:\Omega\rightarrow N$ and $C\in L^0_+(\mathcal{F}_0)$ with $C<c_p$ on $\Omega$ such that $|||F^{(L)}(Y,M)-F^{(L)}(Y',M')|||_p\leq C(|||Y-Y'|||_p+|||M-M'|||_p)$ for all $Y,Y'\in S^p_{\mathcal{F}_0}$ and $M,M'\in M^p_{\mathcal{F}_0,0}$, then $BSE(\ref{1.3})$ has a unique solution $(Y,M)$ in $S^p_{\mathcal{F}_0}\times M^p_{\mathcal{F}_0,0}$. When $L=k$ for some positive integer $k\in N$, the theorem still holds without the stability of $F$.
\end{theorem}

\begin{proof}
Since $C<1$ on $\Omega$, it follows from Lemma\ref{lemma4.3} that $F$ satisfies condition $(S)$. So by Lemma\ref{lemma4.3}, it is enough to prove that $G$ has a unique fixed point in $L^p_{\mathcal{F}_0}(\mathcal{F}_T,R^d)$. This follows from Theorem\ref{theorem2.13} if we can prove that $G$ is a random contraction on $L^p_{\mathcal{F}_0}(\mathcal{F}_T,R^d)$.
\par
Since for any $V\in L^p_{\mathcal{F}_0}(\mathcal{F}_T,R^d), Y^V$ is the unique fixed point of the mapping $Y\rightarrow E_0(V)-F(Y,M^V)-M^V$, it follows from the definition of $F^{(k)}$ that $F^{(k)}(Y^V,M^V)=F(Y^V,M^V)$ for any $k\in N$, so that $F^{(L)}(Y^V,M^V)=F(Y^V,M^V)$.
\par
Hence, one has, for all $V,V'\in L^p_{\mathcal{F}_0}(\mathcal{F}_T,R^d)$, $Y^V_t-Y^{V'}_t=E_0(V)-E_0(V')-(F^{(L)}_t(Y^V,M^V)-F^{(L)}_t(Y^{V'},M^{V'}))-(M^V_t-M^{V'}_t)=E_0(V-{V'})-M_t^{V-{V'}}-(F^{(L)}_t(Y^V,M^V)-F^{(L)}_t(Y^{V'},M^{V'}))$.
\par
Then, $\sup_{0\leq t\leq T}~|Y^V_t-Y_t^{V'}|\leq \sup_{0\leq t\leq T}~|E_0(V-V')-M_t^{V-V'}|~+~\sup_{0\leq t\leq T} \\
|F^{(L)}_t(Y^V,M^V)- F^{(L)}_t (Y^{V'},M^{V'})|$, it follows that $|||Y^V-Y^{V'}|||_p \\ \leq|||E_0(V-V')-M^{V-V'}|||_p+|||F^{(L)}(Y^V,M^V)-F^{(L)}(Y^{V'},M^{V'})|||_p \\ \leq|||E_0(V-V')-M^{V-V'}|||_p+C(|||Y^V-Y^{V'}|||_p+|||M^V-M^{V'}|||_p)$. \\
In particular, $|||Y^V-Y^{V'}|||_p\leq \frac{1}{1-C}(|||E_0(V-V')-M^{V-V'}|||_p+C|||M^{V-V'}|||_p)$, and thus we have:
\begin{equation}\nonumber
\aligned
&|||G(V)-G(V')|||_p\\
&=|||F^{(L)}_T(Y^V,M^V)-F^{(L)}_T(Y^{V'},M^{V'})|||_p\\
&\leq C(|||Y^V-Y^{V'}|||_p+|||M^V-M^{V'}|||_p)\\
&\leq \frac{C}{1-C}(|||E_0(V-V')-M^{V-V'}|||_p+C|||M^{V-V'}|||_p)+C|||M^{V-V'}|||_p\\
&=\frac{C}{1-C}(|||E_0(V-V')-M^{V-V'}|||_p+|||M^{V-V'}|||_p)
\endaligned
\end{equation}
\par
By Lemma\ref{lemma5.1}, if we let $C_p=\frac{p}{p-1}$ for $p\in (1,+\infty)$ and $C_{\infty}=1$, then $|||M^{V-V'}|||_p\leq C_p|||V-V'-E_0(V-V')|||_p$ and $|||E_0(V-V')-M^{V-V'}|||_p\leq C_p|||V-V'|||_p$.
\par
Hence, $|||M^{V-V'}|||_p$

\begin{equation*}
  \leq\left\{
\begin{aligned}
2|||V-V'-E_0(V-V')|||_2  & \leq  2|||V-V'|||_2 ~by~ Lemma\ref{lemma5.3}, for~ p=2, \\
C_p|||V-V'-E_0(V-V')|||_p& \leq  2C_p|||V-V'|||_p ~by~ Lemma\ref{lemma5.1}, for ~p\neq 2.
\end{aligned}
\right.
\end{equation*}

and $|||G(V)-G(V')|||_p$
\begin{equation*}
 \leq\left\{
\begin{aligned}
\frac{4C}{1-C}|||V-V'|||_2   &,  for~~ p=2, \\
3C_p\frac{C}{1-C}|||V-V'|||_p&,  for ~~p\neq 2.
\end{aligned}
\right.
\end{equation*}
This shows that $G$ is a random contraction.
\par
The final part of the theorem can be seen  by applying Proposition\ref{proposition2.12} instead of Theorem\ref{theorem2.13}.
\end{proof}

\par
If the generator is of integral form $F_t(Y,M)=\int^t_0f(s,Y,M)ds$ for a driver $f: [0,T]\times \Omega\times S^p_{\mathcal{F}_0}\times M^p_{\mathcal{F}_0,0}\rightarrow R^d$, then $BSE(\ref{1.3})$ becomes a $BSDE$ of the general form:
\begin{equation}\label{5.1}
  Y_t=\xi+\int^T_tf(s,Y,M)ds+M_T-M_t
\end{equation}

\par
If for an equivalence class $X$ of a $RCLL$ measurable process, one denotes $|||X|||_{p,[0,t]}:= E[(\sup_{0\leq s \leq t}|X_s|)^p~|~ \mathcal{F}_0]^{1/p}$ for $p\in (1,+\infty)$ and $|||X|||_{\infty,[0,t]}:=\bigwedge\{\eta\in \bar{L}^0_+(\mathcal{F}_0)~|~\sup_{0\leq s\leq t}|X_s|\leq\eta\}$.

\par
Proposition\ref{proposition5.5} below can be regarded as the conditional version of \cite[Proposition3.3]{CN17}.

\begin{proposition}\label{proposition5.5}
Let $\xi\in L^p_{\mathcal{F}_0}(\mathcal{F}_T,R^d)$ for some $p\in (1,+\infty]$. Then the $BSDE(5.1)$ has a unique solution $(Y,M)\in S^p_{\mathcal{F}_0}\times M^p_{\mathcal{F}_0,0}$ for every driver $f: [0,T]\times \Omega \times S^p_{\mathcal{F}_0}\times M^p_{\mathcal{F}_0,0}\rightarrow R^d$ satisfying the following conditions:
\begin{enumerate}[(i)]
\item For all $(Y,M)\in S^p_{\mathcal{F}_0}\times M^p_{\mathcal{F}_0,0}, f(\cdot,Y,M)$ is progressively measurable with $|||\int^T_0|f(t,0,0)|dt|||_p<+\infty$, $P$--a.s.
\item There exist $C_1\in L^0_{++}(\mathcal{F}_0)$ and $C_2\in L^0_+(\mathcal{F}_0)$ with $C_2<\frac{c_pC_1}{e^{C_1T}-1}$ on $\Omega$ such that $|||f(t,Y,M)-f(t,Y',M')|||_p\leq C_1|||Y-Y_0+M-(Y'-Y'_0+M')|||_{p,[0,t]}+C_2(|Y_0-Y_0'|+|||M-M'|||_p)$ for all $(Y,M),(Y',M')\in S^p_{\mathcal{F}_0}\times M^p_{\mathcal{F}_0,0}$. Where $|Y_0-Y'_0|:=(\sum^n_{i=1}(Y_{0,i}-Y'_{0,i})^2)^{1/2}$, $Y_{0,i}$ stands for the $i$--th component of the random vector $Y_0$, since $Y_0$ and $Y'_0$ are $\mathcal{F}_0$--measurable, $|Y_0-Y'_0|=|||Y_0-Y'_0|||_p$ for any $p\in (1,+\infty]$.
\end{enumerate}
\end{proposition}

\begin{proof}
Let $q=\frac{p}{p-1}\in [1,+\infty)$ and denote $L^q_{\mathcal{F}_0}(\mathcal{F}_T)=L^q_{\mathcal{F}_0}(\mathcal{F}_T,R^1)$, then $L^q_{\mathcal{F}_0}(\mathcal{F}_T)^*\cong L^p_{\mathcal{F}_0}(\mathcal{F}_T)~($namely the random conjugate space of $L^q_{\mathcal{F}_0}(\mathcal{F}_T)$ is $L^p_{\mathcal{F}_0}(\mathcal{F}_T))$, see \cite{FKV09,Guo10} for details. For any $x\in L^p_{\mathcal{F}_0}(\mathcal{F}_T)$, it is well known that $|||x|||_p=\bigvee\{|E[xy~|~\mathcal{F}_0]|: y\in L^q_{\mathcal{F}_0}(\mathcal{F}_T)$ and $|||y|||_q\leq 1\}$

\par
We want to prove that $|||\int^T_0|f(t,Y,M)|dt|||_p<+\infty$($P$--a.s.) for all $(Y,M)\in S^p_{\mathcal{F}_0}\times M^p_{\mathcal{F}_0,0}$. Since $|||\int^T_0|f(t,Y,M)|dt|||_p\leq|||\int^T_0|f(t,Y,M)-f(t,0,0)|dt|||_p+|||\int^T_0|f(t,0,0)|dt|||_p$, we only need to prove that $|||\int^T_0|f(t,Y,M)~-~f(t,0,0)|dt \\
|||_p <+\infty (P$--a.s.$)$.
\par
For each $n\in N,(\int^T_0|f(t,Y,M)-f(t,0,0)|dt)\wedge n^{\frac{1}{p}}\in L^p_{\mathcal{F}_0}(\mathcal{F}_T)$, so \\ $|||(\int^T_0|f(t,Y,M)-f(t,0,0)|dt)\wedge n^{\frac{1}{p}}|||_p\\
=\bigvee\{E[((\int^T_0|f(t,Y,M)-f(t,0,0)|dt)\wedge n^{\frac{1}{p}})\cdot |y|~|~\mathcal{F}_0]: y\in L^q_{\mathcal{F}_0}(\mathcal{F}_T)$ and $|||y|||_q\leq 1\}\\
\leq \bigvee\{E[(\int^T_0|f(t,Y,M)-f(t,0,0)|dt)\cdot |y|~|~\mathcal{F}_0]: y\in L^q_{\mathcal{F}_0}(\mathcal{F}_T)$ and $|||y|||_q\leq 1\}\\
=\bigvee\{E[\int^T_0(|f(t,Y,M)-f(t,0,0)|\cdot |y|)dt~|~\mathcal{F}_0]: y\in L^q_{\mathcal{F}_0}(\mathcal{F}_T)$ and $|||y|||_q\leq 1\}\\
=\bigvee\{\int^T_0E[(|f(t,Y,M)-f(t,0,0)|\cdot |y|)~|~\mathcal{F}_0]dt: y\in L^q_{\mathcal{F}_0}(\mathcal{F}_T)$ and $|||y|||_q\leq 1\}($by $(2)$ of Lemma\ref{lemma5.2}$)$.\\
$\leq \bigvee\{\int^T_0||||f(t,Y,M)-f(t,0,0)|||_p\cdot |||y|||_qdt: y\in L^p_{\mathcal{F}_0}(\mathcal{F}_T)$ and $|||y|||_q\leq 1\}$\\
$\leq \int^T_0||||f(t,Y,M)-f(t,0,0)|||_pdt$\\
$\leq TC_1|||Y-Y_0+M|||_p+TC_2(|Y_0|+|||M|||_p)$.\\
So $|||\int^T_0|f(t,Y,M)-f(t,0,0)|dt|||_p=lim_{n\rightarrow \infty}|||(\int^T_0|f(t,Y,M)-f(t,0,0)|dt)\wedge n^{\frac{1}{p}}|||_p\leq TC_1|||Y-Y_0+M|||_p+TC_2(|Y_0|+|||M|||_p)<+\infty($P--a.s$)$, namely $F_t(Y,M)=\int^t_0f(s,Y,M)ds$ is a well--defined mapping from $S^p_{\mathcal{F}_0}\times M^p_{\mathcal{F}_0,0}$ to $S^p_{\mathcal{F}_0,0}$. Besides, it follows from $(ii)$ that $F$ is easily checked to be stable by a similar reasoning used in the proof of Lemma\ref{lemma2.11}.
\par
For given $Y,Y'\in S^p_{\mathcal{F}_0}$ and $M,M'\in M^p_{\mathcal{F}_0,0}$, set\\
$\delta:=\frac{C_2}{C_1}(|Y_0-Y'_0|+|||M-M'|||_p),$\\
$H^0_t:=H^0:=2(|||Y-Y'|||_p+|||M-M'|||_p),$\\
$H^{(k)}_t:=|||F^{(k)}(Y,M)-F^{(k)}(Y',M')|||_{p,[0,t]}, k\in N.$\\
Then $H^{(k)}_t\leq\int^t_0|||f(s,Y^{(k,M)},M)-f(s,(Y')^{(k,M')},M')|||_pds \\
\leq\int^t_0(C_1H_s^{(k-1)}+C_2(|Y_0-Y_0'|+|||M-M'|||_p))ds  \\
\leq C_1\int^t_0(H^{(k-1)}_s+\delta)ds,$\\
and by iteration,\\
\begin{equation*}
  H^{(k)}_t\leq \frac{(C_1t)^k}{k!}H^0+(C_1+\cdot\cdot\cdot+\frac{(C_1t)^k}{k!})\delta.
\end{equation*}
In particular,\\
$|||F^{(k)}(Y,M)-F^{(k)}(Y',M')|||_p \\
\leq 2\frac{(C_1T)^k}{k!}(|||Y-Y'|||_p+|||M-M'|||_p)+(e^{C_1T}-1)\frac{C_2}{C_1}(|Y_0-Y'_0|+|||M-M'|||_p)$\\
 for any $k\in N$.
\par
Since $(e^{C_1T}-1)\frac{C_2}{C_1}<c_p$ on $\Omega$ and $\{\frac{(C_1T)^k}{k!}: k\in N\}$ converges a.s. to $0$, let $B_k=(2\frac{(C_1T)^k}{k!}+(e^{C_1T}-1)\frac{C_2}{C_1}<c_p)$ for any $k\in N$, then $\cup_{k\in N}B_k=\Omega$. Further, let $A_1=B_1$ and $A_k=B_k\setminus B_{k-1}$ for any $k\geq 2$, then $\{A_k: k\in N\}$ forms a countable partition of $\Omega$ to $\mathcal{F}_0$. Now, define $L: \Omega\rightarrow N$ by $L(\omega)=\sum^{\infty}_{k=1} k\cdot I_{A_k}(\omega)$, for any $\omega\in \Omega$, then $|||F^{(L)}(Y,M)-F^{(L)}(Y',M')|||_p\leq C(|||Y-Y'|||_p+|||M-M'|||_p)$ with $C:=\sum^{\infty}_{k=1}(2\frac{(C_1T)^k}{k!}+(e^{C_1T}-1)\frac{C_2}{C_1})\tilde{I}_{A_k}< c_p$ on $\Omega$, and the proposition follows from Theorem\ref{theorem5.4}.
\end{proof}

\begin{remark}\label{remark5.6}
Although there exists some similarity between Proposition3.3 of \cite{CN17} and Proposition\ref{proposition5.5} here, the biggest differences between them lie in the following two points: $(1)$. Here, we need a random iteration $T^{(L)}$, which is required by randomness of $C_1$ and $C_2$; $(2)$. $||Y_0-Y'_0||_p$ in \cite[Proposition3.3]{CN17} is replaced by $|Y_0-Y'_0|$ here.
\end{remark}

\par
Similarly to \cite{CN17}, we may also consider generalized Lipschitz $BSDEs$ based on a Brown motion and a Poisson random measure. For this, let $W$ be an $n$--dimensional Brown motion and $N$ an independent Poisson random measure on $[0,T]\times E$ for $E=R^m\setminus\{0\}$ with an intensity measure of the form $dt\mu(dx)$ for a measure $\mu$ over the Borel $\sigma$--algebra $\mathcal{B}(E)$ of $E$ satisfying $\int_E(1\wedge |x|^2)\mu(dx)<+\infty$.

\par
Denote by $\widetilde{N}$ the compensated random measure $N(dt,dx)-dt\mu(dx)$ and assume that, for $A\in \mathcal{B}(E)$ with $\mu(A)<+\infty, \{\widetilde{N}([0,t]\times A): t\in [0,T]\}$ and $W$ are martingales with respect to the filtration $\mathbb{F} :=(\mathcal{F}_t)_{t\in [0,T]}$. First, let us recall the following space of integrands used in \cite{CN17}:
\begin{enumerate}[(i)]
\item $H^2$: the Banach space of equivalence class of $R^{d\times n}$--valued predictable processes $Z$ satisfying $||Z||_2:=(\int^T_0E[|Z_t|^2]dt)^{1/2}<+\infty$.
\item $L^2(\widetilde{N}):$ the Banach space of equivalence classes of $\mathcal{P}\otimes \mathcal{B}(E)$--measurable mappings $U: [0,T]\times \Omega\times E\rightarrow R^d$ such that $||U||_2:=(\int^T_0(\int_EE|U_t(x)|^2\mu(dx \\
    ))dt)^{\frac{1}{2}}<+\infty$.
\end{enumerate}
Where $\mathcal{P}$ is the $\sigma$--algebra of $\mathbb{F}$--predictable subsets of $[0,T]\times \Omega$.
\par
Any square--integrable $\mathbb{F}$--martingale $M\in M^p_0$ has a unique representation of the form
\begin{equation}\label{5.2}
  M_t=\int^t_0Z^M_sdW+\int^t_0\int_EU^M_s(x)\widetilde{N}(ds,dx)+K^M_t
\end{equation}
for a triple $(Z^M,U^M,K^M)\in H^2\times L^2(\widetilde{N})\times M^2_0$ such that $K^M$ is strongly orthogonal to $W$ and $\widetilde{N}$, see, e.g., \cite{IW89,J79}. This makes it possible for Cheridito and Nam \cite{CN17} to consider $BSDEs$
\begin{equation}\label{5.3}
  Y_t=\xi+\int^T_tf(s,Y,Z^M,U^M)ds+M_T-M_t
\end{equation}
for terminal condition $\xi\in L^2(\mathcal{F}_T)^d$ and drivers $f:[0,T]\times\Omega\times S^2\times H^2\times L^2(\widetilde{N})\rightarrow R^d$.

\par
Proposition3.7, Corollary3.9 and Propositioin3.10 of \cite{CN17} are based on the existence,uniqueness and isometry of the decomposition$(5.2)$ as above. To give the conditional versions of the three propositions, we need the existence, uniqueness and isometry of the following decomposition $($see Lemma\ref{lemma5.7} below$)$ for $M\in M^2_{\mathcal{F}_0,0}$. For this, let us first introduce the following $\mathcal{T}_{\varepsilon,\lambda}$--complete $RN$ modules:
\begin{enumerate}[$H^2_{\mathcal{F}_0}:$]
\item the $\mathcal{T}_{\varepsilon,\lambda}$--complete $RN$ module of equivalence classes of $R^{d\times n}$--valued predictable processes $Z$ satisfying $|||Z|||_2:=E[\int^T_0|Z_t|^2dt~|~\mathcal{F}_0]^{\frac{1}{2}}<+\infty,~P$--a.s., which is endowed with the $L^0$--norm $|||\cdot|||_2:H^2_{\mathcal{F}_0}\rightarrow L^0_+(\mathcal{F}_0)$ given by $|||Z|||_2=E[\int^T_0|Z_t|^2dt~|~\mathcal{F}_0]^{\frac{1}{2}}$.
\end{enumerate}

\begin{enumerate}[$L^2_{\mathcal{F}_0}(\widetilde{N}):$]
\item the $\mathcal{T}_{\varepsilon,\lambda}$--complete $RN$ module of equivalence classes of $\mathcal{P}\otimes \mathcal{B}(E)$--measurable mappings $U: [0,T]\times \Omega\times E\rightarrow R^d$ such that $|||U|||_2:=E[\int^T_0(\int_E|U_t(x)|^2\mu(dx))dt~|~\mathcal{F}_0]^{\frac{1}{2}}<+\infty, P$--a.s., which is endowed with the $L^0$--norm $|||\cdot|||_2: L^2_{\mathcal{F}_0}(\widetilde{N})\rightarrow L^0_+(\mathcal{F}_0)$ given by $|||U|||_2:=E[\int^T_0(\int_E|U_t(x)|^2\mu(dx))dt~|~\mathcal{F}_0]^{\frac{1}{2}}$.
\end{enumerate}

\begin{enumerate}[$\sideset{^{\oplus}}{^2_{\mathcal{F}_0,0}}{\mathop{\mathrm{M}}}:$]
\item $=L^0(\mathcal{F}_0)\cdot \sideset{^{\oplus}}{^2_0}{\mathop{\mathrm{M}}}:=\{\xi\cdot M:\xi \in L^0(\mathcal{F}_0)$ and $M\in \sideset{^{\oplus}}{^2_0}{\mathop{\mathrm{M}}}\}$, where $\sideset{^{\oplus}}{^2_0}{\mathop{\mathrm{M}}}:=\{M\in M^2_0: M$ is strongly orthogonal to $W$ and $\widetilde{N}\}($it is easy to see that $\sideset{^{\oplus}}{^2_0}{\mathop{\mathrm{M}}}$ is a closed subspace of the Banach space $M^2_0)$. The $L^0$--norm on $\sideset{^{\oplus}}{^2_{\mathcal{F}_0,0}}{\mathop{\mathrm{M}}}$ is the restriction of the $L^0$--norm $|||\cdot|||_2$ on $M^2_{\mathcal{F}_0,0}$ to $\sideset{^{\oplus}}{^2_{\mathcal{F}_0,0}}{\mathop{\mathrm{M}}}$, it is easy to check that $(\sideset{^{\oplus}}{^2_{\mathcal{F}_0,0}}{\mathop{\mathrm{M}}}, |||\cdot|||_2)$ is a $\mathcal{T}_{\varepsilon,\lambda}$--complete $RN$ module over $R$ with base $(\Omega,\mathcal{F}_0,P)$.
\end{enumerate}

\par
According to the property of a generalized conditional mathematical expectation, it is easy to see that $L^2_{\mathcal{F}_0}(\widetilde{N})=L^0(\mathcal{F}_0)\cdot L^2(\widetilde{N})$ and $H^2_{\mathcal{F}_0}=L^0(\mathcal{F}_0)\cdot H^2$.

\par
We can now give a key lemma as follows:
\begin{lemma}\label{lemma5.7}
Every $M\in M^2_{\mathcal{F}_0,0}$ can be uniquely decomposed as $M_t=\int^t_0Z^M_sdW_s \\ +\int^t_0\int_EU^M_s(x)\widetilde{N}(ds,dx)+K^M_t($for each $t\in [0,T])$ for a triple $(Z^M,U^M,K^M)\in H^2_{\mathcal{F}_0} \times L^2_{\mathcal{F}_0}(\widetilde{N}) \times \sideset{^{\oplus}}{^2_{\mathcal{F}_0,0}}{\mathop{\mathrm{M}}}$. Further, one has the isometry

\begin{align}\label{5.4}
 E[|M_t|^2~|~\mathcal{F}_0] =&  \int ^t_0E[|Z^M_s|^2~|~\mathcal{F}_0]ds+\int^t_0(\int_EE[|U^M_s(x)|^2~|~\mathcal{F}_0]\mu(dx))ds \\ \nonumber
   +&  E[|K^M_t|^2~|~\mathcal{F}_0]
\end{align}
\end{lemma}

\begin{proof}
For $(Z^M,U^M)\in H^2_{\mathcal{F}_0}\times L^2_{\mathcal{F}_0}(\widetilde{N})$, there exist $\xi_1\in L^0(\mathcal{F}_0)$ and $Z^{(1)}\in H^2$ such that $Z^M=\xi_1\cdot Z^{(1)}$, and there exists $\xi_2\in L^0(\mathcal{F}_0)$ and $U^{(2)}\in L^2(\widetilde{N})$ such that $Z^M=\xi_2\cdot U^{(2)}$. Clearly, stochastic integrals $\int^t_0Z^{(1)}_sdW_s$ and $\int^t_0\int_E U^{(2)}_s(x)\widetilde{N}(ds,dx)$ are both in $M^2_0$, even $\int^t_0Z^M_sdW_s$ and $\int^t_0\int_E U^M_s(x)\widetilde{N}(ds, \\
dx)$ are both well--defined and both locally square--integrable martingales, see \cite{IW89} for details. It is also obvious that $\int^t_0Z^M_sdW_s= \xi_1 \cdot\int^t_0Z^{(1)}_sdW_s$ and $\int^t_0\int_E U^M_s(x)\widetilde{N}(ds,dx)=\xi_2\cdot \int^t_0U^{(2)}_s(x)\widetilde{N}(ds,dx)$ are both generalized martingales.
\par
Since $M\in M^2_{\mathcal{F}_0,0}$, there exists $\xi \in L^0(\mathcal{F}_0)$ and $\hat{M}\in M^2_0$ such that $M=\xi\cdot \hat{M}$, then applying the decomposition$(5.2)$ to $\hat{M}$ yields a unique triple $(Z^{\hat{M}},U^{\hat{M}},K^{\hat{M}}) \\
\in H^2\times L^2(\widetilde{N})\times \sideset{^{\oplus}}{^2_0}{\mathop{\mathrm{M}}}$ such that $\hat{M}_t=\int^t_0 Z_s^{\hat{M}}dW_s+\int^t_0\int_E U_s^{\hat{M}}(x)\widetilde{N}(ds,dx)+K_t^{\hat{M}}$ for each $t\in [0,T]$, it is also well known that the following isometry holds, see, e.g.,\cite{J79}:
\begin{equation}\label{5.5}
  E[|\hat{M}_t|^2]=\int ^t_0E|Z_t^{\hat{M}}|^2ds+\int^t_0\int_EE|U_s^{\hat{M}}|^2\mu(dx)ds+E|K_t^{\hat{M}}|^2,
\end{equation}
for each $t\in[0,T]$.
\par
Thus $M_t=\xi\cdot\hat{M}_t=\int^t_0\xi\cdot Z_t^{\hat{M}}dW_s+\int^t_0\int_E\xi\cdot U_s^{\hat{M}}(x)\widetilde{N}(ds,dx)+\xi\cdot K_t^{\hat{M}}$, taking $Z^M=\xi\cdot Z^{\hat{M}}, U^M=\xi\cdot U^{\hat{M}}$ and $K^M=\xi\cdot K^{\hat{M}}$ proves the existence of $(Z^M,U^M,K^M)\in H^2_{\mathcal{F}_0}\times L^2_{\mathcal{F}_0}(\widetilde{N})\times \sideset{^{\oplus}}{^2_{\mathcal{F}_0,0}}{\mathop{\mathrm{M}}}$. We will prove the uniqueness of $(Z^M,U^M,K^M)$ as follows.
\par
Let $M=\xi\cdot \hat{M}$ be as above. Further, let $Z^M=\xi_1\cdot Z^{(1)}, U^M=\xi_2\cdot U^{(2)}$ and $K^M=\xi_3\cdot K^{(3)}$ for some $(\xi_1,\xi_2,\xi_3,Z^{(1)},U^{(2)},K^{(3)})\in L^0(\mathcal{F}_0)\times L^0(\mathcal{F}_0)\times L^0(\mathcal{F}_0)\times H^2\times L^2(\widetilde{N})\times \sideset{^{\oplus}}{^2_0}{\mathop{\mathrm{M}}}$. Putting $L:=|\xi|+|\xi_1|+|\xi_2|+|\xi_3|+1$, then $L\in L^0_{++}(\mathcal{F}_0)$. Further, denote $\bar{M}=M/L,\bar{Z}=Z^M/L,\bar{U}=U^M/L$ and $\bar{K}=K^M/L$, then $\bar{M}_t=\int^t_0 \bar{Z}_sdW_s+\int^t_0\int_E \bar{U}_s(x)\widetilde{N}(ds,dx)+\bar{K}_t$ for each $t\in [0,T]$. It is easy to observe that $\bar{M}\in M^2_0, \bar{Z}\in H^2, \bar{U}\in L^2(\widetilde{N})$ and $\bar{K}\in \sideset{^{\oplus}}{^2_0}{\mathop{\mathrm{M}}}$, it follows from the uniqueness of the decomposition$(5.2)$ for $\bar{M}$ that $(\bar{Z},\bar{U},\bar{K})$ is unique, so is $(Z^M,U^M,K^M)$.
\par
Now that we have proved the uniqueness of $(Z^M,U^M,K^M)$, then for $M = \xi\cdot\hat{M}$ with $\xi\in L^0(\mathcal{F}_0)$ and $\hat{M}\in M^2_0$, one has $Z^M=\xi\cdot Z^{\hat{M}}, U^M=\xi\cdot U^{\hat{M}}$ and $K^M=\xi\cdot K^{\hat{M}}$. Since, for any $A\in \mathcal{F}_0, I_A\cdot\hat{M}\in M^2_0,$ and $I_A\cdot\hat{M}$ can be uniquely written as $I_A\cdot\hat{M}_t=\int^t_0 I_A\cdot Z_s^{\hat{M}}dW_s+\int^t_0\int_EI_A\cdot U_s^{\hat{M}}(x)\widetilde{N}(ds,dx)+I_A\cdot K_t^{\hat{M}}$ for each $t\in [0,T]$. Again by the isometry of the decomposition$(5.2)$, one has $E[I_A|\hat{M}_t|^2]=\int^t_0E[I_A|Z_s^{\hat{M}}|^2]ds+\int^t_0\int_EE[I_A\cdot|U_s^{\hat{M}}(x)|^2]\mu(dx)ds+E[I_A|K_t^{\hat{M}}|^2]$, namely, $E[|\hat{M}_t|^2~|~\mathcal{F}_0]=\int^t_0E[|Z_s^{\hat{M}}|^2~|~\mathcal{F}_0]ds+\int^t_0\int_EE[|U_s^{\hat{M}}(x)|^2~|~\mathcal{F}_0] \mu(dx)ds+E[|K_t^{\hat{M}}|^2~|~\mathcal{F}_0]$, from which $(5.4)$ follows.
\end{proof}

\par
Lemma\ref{lemma5.7} makes it possible for us to consider $BSEDs$
\begin{equation}\label{5.6}
  Y_t=\xi+\int^T_tf(s,Y,Z^M,U^M)ds+M_T-M_t
\end{equation}
for terminal conditions $\xi\in L^2_{\mathcal{F}_0}(\mathcal{F}_T,R^d)$ and drivers
\begin{equation}\label{5.7}
 f:[0,T]\times \Omega\times S^2_{\mathcal{F}_0} \times H^2_{\mathcal{F}_0}\times L^2_{\mathcal{F}_0}(\widetilde{N})\rightarrow R^d.
\end{equation}

\begin{proposition}\label{proposition5.8}
The $BSDE(5.6)$ has a unique solution $(Y,M)\in S^2_{\mathcal{F}_0} \times M^2_{\mathcal{F}_0,0}$ for every terminal condition $\xi\in L^2_{\mathcal{F}_0}(\mathcal{F}_T,R^d)$ and driver $f:[0,T]\times \Omega\times S^2_{\mathcal{F}_0} \times H^2_{\mathcal{F}_0}\times L^2_{\mathcal{F}_0}(\widetilde{N})\rightarrow R^d$ satisfying the following two conditions:
\begin{enumerate}[(i)]
\item For all $(Y,Z,U)~\in~ S^2_{\mathcal{F}_0} ~\times~ H^2_{\mathcal{F}_0}~\times~ L^2_{\mathcal{F}_0}(\widetilde{N})$, $f(t,Y,Z,U)$ is progressively measurable with $|||\int^T_0|f(t,0,0,0)|dt|||_2<+\infty, P$--a.s.
\item There exists $C\in L^0_{+}$ such that
\begin{equation}\nonumber
\aligned
&\int^T_t|||f(s,Y,Z,U)-f(s,Y',Z',U')|||_2ds\\
&\leq C\int^T_t(|||Y_s-Y'_s|||_2+|||Z_s-Z'_s|||_2+|||U_s-U'_s|||_2)ds
\endaligned
\end{equation}
for all $t\in[0,T]$ and $(Y,Z,U),(Y',Z',U')\in S^2_{\mathcal{F}_0} \times H^2_{\mathcal{F}_0}\times L^2_{\mathcal{F}_0}(\widetilde{N})$, where $|||U_s-U'_s|||_2=E[\int_E|(U_s-U'_s)(x)|^2\mu(dx)~|~\mathcal{F}_0]^{1/2}$.
\end{enumerate}
\end{proposition}

\begin{proof}
Choose $\delta\in L^0_{++}(\mathcal{F}_0)$ such that $C\sqrt{3\delta(\delta+1)}<\frac{1}{5}$ on $\Omega$ and $k:=T/\tilde{\delta}$ is an $\mathcal{F}_0$--measurable positive integer--valued random variable, where $\tilde{\delta}$ is a chosen representative of $\delta$. As usual, for any $\eta\in L^0(\mathcal{F}_0), \tilde{\delta}\cdot\eta$ is interpreted as $\delta\cdot\eta$ and the integral $\int^T_{T-\tilde{\delta}}|f(s,Y,Z^M,U^M)|ds$ is interpreted as $\int^T_0I_{[T-\delta,T]}(s)\cdot|f(s,Y,Z^M,U^M)|ds$.
\par
By the isometry $(5.4)$, one has, for every $M\in M^2_{\mathcal{F}_0,0},$
\begin{equation}\label{5.8}
  \aligned
&(\int^t_0|||Z^M_s|||_2+|||U^M_s|||_2ds)^2\\
&\leq t\int^t_0(|||Z^M_s|||_2+|||U^M_s|||_2)^2ds\\
&\leq 2t \int^t_0|||Z^M_s|||^2_2+|||U^M_s|||^2_2ds\\
&\leq 2t|||M_t|||^2_2.
\endaligned
\end{equation}

Therefore, one obtains from the assumptions for all $(Y,M)\in S^2_{\mathcal{F}_0}\times M^2_{\mathcal{F}_0,0},$
\begin{equation}\nonumber
\aligned
&|||\int^T_{T-\tilde{\delta}}|f(s,Y,Z^M,U^M)|ds|||_2\\
&\leq \int^T_{T-\tilde{\delta}}|||f(s,Y,Z^M,U^M)-f(s,0,0,0)|||_2ds+|||\int^T_{T-\tilde{\delta}}|f(s,0,0,0)|dt|||_2~(by~the~\\
&same~ argument ~as~ in~ the~ proof~ of~ Proposition\ref{proposition5.5})\\
&\leq |||\int^T_{T-\tilde{\delta}}|f(s,0,0,0)|ds|||_2+C\cdot \int^T_{T-\tilde{\delta}}(|||Y_s|||_2+|||Z^M_s|||_2+|||U^M_s|||_2)ds\\
&\leq |||\int^T_0|f(s,0,0,0)|ds|||_2+C(\int^T_0|||Y|||_2ds+\int^T_0(|||Z^M_s|||_2+|||U^M_s|||_2)ds)\\
&\leq |||\int^T_0|f(s,0,0,0)|ds|||_2+CT|||Y|||_2+C\sqrt{2T}|||M|||_2(by~(5.8)).
\endaligned
\end{equation}
This shows that $F_t(Y.M):=\int^t_0 f(s,Y,Z^M,U^M)I_{[T-\tilde{\delta},T]}(s)ds$ defines a process in $S^2_{\mathcal{F}_0,0}$.
\par
Furthermore, by the isometry$(5.4)$ and the same argument as in \cite[Proposition3.7]{CN17} one only need to replace the norm $||\cdot||_2$ there with the conditional 2--norm $|||\cdot|||_2$ here to obtain that $|||F(Y,M)-F(Y',M')|||_2\leq C\sqrt{3\delta(\delta+1)}(|||Y-Y'|||_2+|||M-M'|||_2)$ for all $(Y,M),(Y',M')\in S^2_{\mathcal{F}_0}\times M^2_{\mathcal{F}_0,0}$. Since $C\sqrt{3\delta(\delta+1)} \\
<\frac{1}{5}$ on $\Omega$, one obtains from the final part of Theorem\ref{theorem5.4} that the $BSDE: Y_t=\xi+\int^T_tf(s,Y,Z^M,U^M)I_{[T-\tilde{\delta},T]}(s)ds+M_T-M_t$ has a unique solution $(Y^{(k)},M^{(k)})$ in $S^2_{\mathcal{F}_0}\times M^2_{\mathcal{F}_0,0}$. Now consider the $BSDE$
\begin{equation}\label{5.9}
  Y_t=Y^{(k)}_{T-\tilde{\delta}}+\int^{T-\tilde{\delta}}_t f^{(k-1)}(s,Y,Z^M,U^M)I_{[T-2\tilde{\delta},T-\tilde{\delta}]}(s)ds+M_{T-\tilde{\delta}}-M_t
\end{equation}
on the random time interval $[0,T-\tilde{\delta}]$, where $f^{(k-1)}$ is given by $f^{(k-1)}(s,Y,Z,U)\\
:=f(s,(Y,Z,U)I_{[0,T-\tilde{\delta}]}+(Y^{(k)}, Z^{M^{(k)}},U^{M^{(k)}})I_{[T-\tilde{\delta},T]})$. Then the conditions $(i)$--$(ii)$ still hold. So $(5.9)$ has a unique solution $(Y^{(k-1)}, M^{(k-1)})$ in $S^2_{\mathcal{F}_0}\times M^2_{\mathcal{F}_0,0}$ over the random time interval $[0,T-\tilde{\delta}]$. Repeating the same argument, one obtains solutions $(Y^{(j)},M^{(j)}), j=1,2,\cdot\cdot\cdot,\cdot\cdot\cdot,k.$ The remaining part of the proof is the same as the final part of the proof of \cite[Propositin 3.7]{CN17}, so is omitted.
\end{proof}

\begin{corollary}\label{corollary5.9}
The $BSDE$ $Y_t=\xi +\int^T_tf(s,Y_s,Z^M_s,U^M_s)ds+M_T-M_t$ has a unique solution $(Y,M)\in S^2_{\mathcal{F}_0}\times M^2_{\mathcal{F}_0,0}$ for every terminal condition $\xi\in L^2_{\mathcal{F}_0}(\mathcal{F}_T,R^d)$ and driver $f:[0,T]\times \Omega\times L^2_{\mathcal{F}_0}(\mathcal{F}_T,R^d)\times L^2_{\mathcal{F}_0}(\mathcal{F}_T,R^{d\times n})\times L^2_{\mathcal{F}_0}(\Omega\times E, \mathcal{F}_T \otimes\mathcal{B}(E), P\otimes\mu; R^d)\rightarrow R^d$ satisfying the following two conditions:
\begin{enumerate}[(i)]
\item For all $(Y,~Z,~U)\in S^2_{\mathcal{F}_0} \times H^2_{\mathcal{F}_0}\times L^2_{\mathcal{F}_0}(\widetilde{N}),~ f(\cdot~,Y_{\cdot}~,Z_{\cdot}~,U_{\cdot})$ is progressively measurable with $|||\int^T_0|f(t,0,0,0)| dt|||_2<+\infty (P$--a.s.$)$.
\item  There exists $C\in L^0_+(\mathcal{F})$ such that $|||f(t,Y_t,Z_t,U_t)-f(t,Y'_t,Z'_t,U'_t)|||_2\leq C(|||Y_t-Y'_t|||_2+|||Z_t-Z'_t|||_2+|||U_t-U'_t|||_2)$ for all $t\in[0,T]$ and $(Y,Z,U),(Y',Z',U')\in S^2_{\mathcal{F}_0} \times H^2_{\mathcal{F}_0}\times L^2_{\mathcal{F}_0}(\widetilde{N})$ where $L^2_{\mathcal{F}_0} (\Omega\times E, \mathcal{F}_T \otimes\mathcal{B}(E), P\otimes\mu; R^d)$ is the $\mathcal{T}_{\varepsilon,\lambda}$--complete $RN$ module of equivalence classes of $R^d$--valued $\mathcal{F}_T \otimes\mathcal{B}(E)$--measurable functions $U$ on $\Omega\times E$ satisfying $E[\int_E|U(\cdot,x)|^2\mu(dx)~|~\mathcal{F}_0]^{\frac{1}{2}}<+\infty$ (P--a.s.), which is endowed with the $L^0$--norm $|||\cdot|||_2$ given by $|||U|||_2=E[\int_E|U(\cdot,x)|^2\mu(dx)~|~\mathcal{F}_0]^{\frac{1}{2}}$ for any $U\in L^2_{\mathcal{F}_0}(\Omega\times E, \mathcal{F}_T \otimes\mathcal{B}(E), P\otimes\mu; R^d)$. It is also easy to check that $L^2_{\mathcal{F}_0}(\Omega\times E, \mathcal{F}_T \otimes\mathcal{B}(E), P\otimes\mu; R^d)=L^0(\mathcal{F}_0)\cdot L^2(\Omega\times E, \mathcal{F}_T \otimes\mathcal{B}(E), P\otimes\mu; R^d)$.
\end{enumerate}
\end{corollary}

\par
Allowing Lipschitz coefficients to be random makes it possible for us to consider the following time--delayed $BSDEs$ involving a random measure $v$. Let us recall that $v: \Omega\times \mathcal{B}[0,T]\rightarrow R^1$ is called an $\mathcal{F}_0$--measurable finite Borel random measure if, for $\omega\in\Omega,v(\omega,\cdot): \mathcal{B}[0,T]\rightarrow R^1$ is a finite Borel measure on $[0,T]$, and, for each Borel subset $B$ of $[0,T], v(\cdot,B):\Omega \rightarrow R^1$ is $\mathcal{F}_0$--measurable.

\begin{proposition}\label{proposition5.10}
Let $\xi\in L^2_{\mathcal{F}_0}(\mathcal{F}_T,R^d)$ and $v$ be an $\mathcal{F}_0$--measurable finite Borel random measure on $[0,T]$. Then the $BSDE$
\begin{equation}\label{5.10}
 Y_t=\xi+\int^T_t\int_{[0,s]}g(s-r,Z^M_{s-r},U^M_{s-r})v(dr)ds+M_T-M_t
\end{equation}
has a unique solution $(Y,M)\in S^2_{\mathcal{F}_0}\times M^2_{\mathcal{F}_0,0}$ for each mapping $g:[0,T]\times\Omega\times L^2_{\mathcal{F}_0}(\mathcal{F}_T,R^{d\times n})\times L^2_{\mathcal{F}_0}(\Omega\times E, \mathcal{F}_T \otimes\mathcal{B}(E), P\otimes\mu; R^d)\rightarrow R^d$ satisfying the two conditions:\\
$(i)$ For all $(Z,U)\in H^2_{\mathcal{F}_0}\times L^2_{\mathcal{F}_0}(\widetilde{N}), g(\cdot ,\cdot ,Z_{\cdot},U_{\cdot})$ is progressively measurable and $|||\int^T_0|g(t,0,0)|dt|||_2<+\infty(P$--a.s.$)$.\\
$(ii)$ There exists $C_1\in L^0_+(\mathcal{F})$ such that
 \begin{equation*}
  |||g(t,Z_t,U_t)-g(t,Z'_t,U'_t)|||_2\leq C_1(|||Z_t-Z'_t|||_2+|||U_t-U'_t|||_2)
\end{equation*}
for all $t\in [0,T]$ and $(Z,U),(Z',U')\in H^2_{\mathcal{F}_0}\times L^2_{\mathcal{F}_0}(\widetilde{N})$.
\end{proposition}

\begin{proof}
The generator corresponding to the $BSDE(5.10)$ is given by $F_t(M)=\int^t_0\int_{[0,s]}g(s-r,Z^M_{s-r},U^M_{s-r})v(dr)ds$. Since it does not depend on $Y$, it satisfies condition$(S)$. So, by Theorem\ref{theorem4.1}, it is enough to show that there exists a unique $V\in L^2_{\mathcal{F}_0}(\mathcal{F}_T,R^d)$ such that
\begin{equation}\label{5.11}
  V=G(V)=\xi+\int^T_0\int_{[0,s]}g(s-r,Z^{M^V}_{s-r},U^{M^V}_{s-r})v(dr)ds.
\end{equation}
From Fubini's theorem and a change of variable, one obtains $\int^T_0\int_{[0,s]}g(s-r,Z^{M^V}_{s-r},U^{M^V}_{s-r})v(dr)ds=\int^T_0v([0,T-s])g(s,Z_s^{M^V},U_s^{M^V})ds$. Since the driver $h(s,Z_s,V_s)=v([0,T-s])g(s,Z_s,U_s)$ satisfies the conditions of Corollary $5.9$ for the $L^0$--Lipschitz coefficient $C:=v[0,T]\cdot C_1$, the $BSDE~~Y_t=\xi+\int^T_t h(s,Z^M_s,U^M_s) \\
ds+M_T-M_t$ has a unique solution $(Y,M)\in S^2_{\mathcal{F}_0}\times M^2_{\mathcal{F}_0,0}$. The associated generator, $\widetilde{F}_t(M)=\int^T_0h(s,Z^M_s,U^M_s)ds$, does not depend on $Y$ either. So it also satisfies condition $(S)$, and one obtains from Theorem\ref{theorem4.1} that there exists a unique $V\in L^2_{\mathcal{F}_0}(\mathcal{F}_T,R^d)$ satisfies $(5.11)$.
\end{proof}

\section{Existence of solutions of $BSEs$ and $BSDEs$ of nonexpansive type}\label{section6}

\par
~~~~$BSE(\ref{1.1})$ and the corresponding $BSDEs$ have been thoroughly studied in \cite{CN17} by making full use of Banach's contraction mapping principle and Krasnoselskii's fixed theorem. In this section we first make full use of the classical Browder--Kirk's fixed point theorem to continue the study of $BSE(\ref{1.1})$, then we make full use of our generalized Browder--Kirk's fixed point theorem given in Section\ref{section3} to study $BSE(\ref{1.3})$. Results in this section are completely new.
\par
Throughout this section, we always assume $p\in (1,+\infty)$. $BSE(\ref{1.1})$ is said to be of nonexpansive type if the generator $F: S^p\times M^p_0\rightarrow S^p_0$ satisfies condition $(S)$ and for a given terminal condition $\xi\in L^p(\mathcal{F}_T)^d(:=L^p(\mathcal{F}_T,R^d))$ the corresponding operator $G$ determined by $(F,\xi)$ is nonexpansive, namely $\|G(V)-G(V')||_p\leq||V-V'||_p$ for all $V,V'\in L^p(\mathcal{F}_T)^d($please bear in mind: $G(V)=\xi+F_T(Y^V,M^V)$ for any $V\in L^p(\mathcal{F}_T)^d)$. Similarly, one can have the notion of ``$BSE(\ref{1.3})$ being of nonexpansive type".
\par
Theorem\ref{theorem6.1} below is almost clear, but since it is frequently employed in the sequel, we would like to summarize and prove it.

\begin{theorem}\label{theorem6.1}
$BSE(\ref{1.1})$ has a solution $(Y,M)\in S^p\times M^p_0$ if it is of nonexpansive type and there exists a bounded closed convex subset $H$ of $L^p(\mathcal{F}_T)^d$ such that $G$ maps $H$ into $H($namely $G(H)\subset H)$.
\end{theorem}

\begin{proof}
Since $1<p<\infty, L^p(\mathcal{F}_T)^d$ is uniformly convex, then $H$ is weakly compact with normal structure, so $G$ has a fixed point in $H$ by the Browder--Kirk's fixed point theorem, which means that $BSE(\ref{1.1})$ has a solution by Theorem2.3 of \cite{CN17}.
\end{proof}

\begin{corollary}\label{corollary6.2}
$BSE(\ref{1.1})$ has a solution $(Y,M)\in S^p\times M^p_0$ if it is of nenexpansive type and there exists a positive number $R_1$ such that $||\xi||_p~+~\sup\{||F_T(Y^V, \\
M^V )||_p:V\in L^p(\mathcal{F}_T)^d$ and $||V||_p\leq R_1\}~\leq~ R_1$. In particular, when $\{F_T(Y^V, \\
M^V): V\in L^p(\mathcal{F}_T)^d\}$ is a bounded set of $L^p(\mathcal{F}_T)^d$, such a $R_1$  always exists.
\end{corollary}

\begin{proof}
Let $H=\{V\in L^p(\mathcal{F}_T)^d: ||V||_2\leq R_1\}$, then $G(H)\subset H$, so the proof follows from Theorem\ref{theorem6.1}. When $R_2:=\sup\{||F_T(Y^V,M^V)||_2:V\in L^p(\mathcal{F}_T)^d\}<+\infty$, taking $R_1=||\xi||_2+R_2$ ends the proof of this corollary.
\end{proof}

\begin{example}\label{example6.3}
Let $f: R^d\rightarrow R^d$ be a bounded Lipschitzian function such that $f(0)=0$ and the Lipschitz constant of $f$ is 1. Define $F: S^p\times M^p_0\rightarrow S^p_0$ by $F_t(Y,M)=\frac{1}{2C_p}f[\frac{t}{T}(Y_0-M_t)]$ for all $(Y,M)\in S^p\times M^p_0$ and $t\in [0,T]$, where $C_p=\frac{p}{p-1}$. Then $BSE(\ref{1.1})$ has a solution.
\end{example}

\begin{proof}
It is obvious that $F$ satisfies condition $(S)$. Further, $||F_T(Y,M)-F_T(Y',M')||_p\leq \frac{1}{2C_p}||(Y_0-M_T)-(Y'_0-M'_T)||_p$ for all $(Y,M),(Y',M')\in  S^p\times M^p_0$, so $||G(V)-G(V')||_p\leq \frac{1}{2C_p} ||Y^V_0-Y^{V'}_0-(M^V_T-M_T^{V'})||_p= \frac{1}{2C_p}||E_0(V-V')-(V-V')||_p\leq ||V-V'||_p$ for all $V,V'\in L^p(\mathcal{F}_T)^d$, which shows that the corresponding $BSE(\ref{1.1})$ is of nonexpansive type. Finally, since $f$ is bounded, $\{F_T(Y^V,M^V): V\in L^p(\mathcal{F}_T)^d\}$ is a bounded set, and hence $BSE(\ref{1.1})$ has a solution by Corollary\ref{corollary6.2}.
\end{proof}

\par
Example\ref{example6.3} motivates the following:
\begin{proposition}\label{proposition6.4}
If the generator $F:S^p\times M^p_0\rightarrow S^p_0$ satisfies the following two conditions:
\begin{enumerate}[(i)]
\item There exists some positive integer $k\in N$ such that $||F^{(k)}(Y,M)-F^{(k)}(Y', \\
M')||_p\leq \frac{1}{2C_p}||(Y_0-M)-(Y_0'-M')||_p$ for all $(Y,M),(Y',M')\in S^p\times M^p_0$.
\item There exists a positive number $R_1$ such that $||\xi||_p+sup\{||F_T(Y^V,M^V)||_p: V\in L^p(\mathcal{F}_T)^d$ and $||V||_p\leq R_1\}\leq R_1$.
\end{enumerate}
 Then $BSE(\ref{1.1})$ has a solution $(Y,M)\in S^p\times M^p_0$.
\end{proposition}

\begin{proof}
In $(i)$, taking $M=M'$ yields that $||F^{(k)}(Y,M)-F(Y',M)||_p\leq \frac{1}{2C_p}||Y_0-Y_0'||_p\leq \frac{1}{2C_p}||Y-Y'||_p ($please bear in mind that $||Y-Y'||_p=||(\sup_{0\leq t\leq T}|Y_t-Y_t'|)||_p$ for all $Y,Y'\in S^p)$. Since $C_p=\frac{p}{p-1}, \frac{1}{2C_p}<1$, one has that $F$ satisfies condition $(S)$ by lemma2.5 of \cite{CN17}. Furthermore, $F(Y^V,M^V)=F^{(k)}(Y^V,M^V)$ for all $V\in L^p(\mathcal{F}_T)^d$, and $||F(Y^V,M^V)-F(Y^{V'},M^{V'})||_p\leq \frac{1}{2C_p}||E_0(V-V')-M^{V-V'}||_p\leq ||V-V'||_p$ by Doob's inequality, for all $V,V'\in L^p(\mathcal{F}_T)^d$, and hence also $||G(V)-G(V')||_p\leq ||F(Y^V,M^V)-F(Y^{V'}, M^{V'})||_p\leq ||V-V'||_p$, namely the corresponding $BSE(\ref{1.1})$ is of nonexpansive type. This and $(ii)$ imply that the corresponding $BSE(\ref{1.1})$ has a solution by Corollary\ref{corollary6.2}.
\end{proof}

\begin{corollary}\label{corollary6.5}
The $BSDE$
\begin{equation}\label{6.1}
 Y_t=\xi+\int^T_t f(s,Y,M)ds+M_T-M_t
\end{equation}
has a solution for a terminal condition $\xi \in L^p(\mathcal{F}_T)^d$ if the driver $f: [0,T]\times \Omega\times S^p\times M^p_0\rightarrow R^d$ satisfies the following three conditions:
\begin{enumerate}[(i)]
\item For all $(Y,M)\in S^p\times M^p_0, f(\cdot,Y,M)$ is progressively measurable with $||\int^T_0f(t,0,0)dt||_p<+\infty$.
\item $||f(t,Y,M)-f(t,Y',M')||_p\leq \frac{1}{2TC_p}||(Y_0-M)-(Y'_0-M')||_p$ for all $t\in [0,T]$ and all $(Y,M), (Y',M')\in S^p\times M^p_0$.
\item There exists $R_1>0$ such that $||\xi||_p+\sup\{||\int^T_0f(t,Y^V,M^V)dt||_p: V\in L^p(\mathcal{F}_T)^d$ and $||V||_p\leq R_1\}\leq R_1$.
\end{enumerate}
\end{corollary}

\begin{proof}
First, $F: S^p\times M^p_0\rightarrow S^p_0$, given by $F_t(Y,M)=\int^t_0f(s,Y,M)ds$, is well--defined since $||F(Y,M)||_p\leq ||\int^T_0|f(t,Y,M)|dt||_p\leq ||\int^T_0(|f(t,Y,M)-f(t,0,0)|+|f(t,0,0)|)dt||_p\leq ||\int^T_0|f(t,Y,M)-f(t,0,0)|dt||_p+||\int^T_0|f(t,0,0)|dt \\
||_p \leq \frac{1}{2C_p}||Y_0-M||_p+ ||\int^T_0|f(t,0,0)|dt||_p<+\infty$ for all  $(Y,M)\in S^p\times M^p_0($by $(i)$ and $(ii))$. Second, $||F(Y,M)-F(Y',M')||_p\leq ||\int^T_0|f(t,Y,M)-f(t,Y',M')|dt||_p \\
\leq \int^T_0||f(t,Y,M)-f(t,Y',M')||_pdt\leq \frac{1}{2C_p}||(Y_0-M)-(Y'_0-M')||_p$.
\par
Finally, since $\|F_T(Y^V,M^V)||_p~=~\|\int^T_0f(t,Y^V,M^V)dt\|_p$, then $\|\xi\|_p~+~\sup\{ \\
\|F_T(Y^V,M^V)\|_p: V\in L^p(\mathcal{F}_T)^d$ and $\|V\|_p\leq R_1\}\leq \|\xi\|_p~+~\sup\{\|\int^T_0f(t,Y^V, \\ M^V)dt\|_p: V\in L^p(\mathcal{F}_T)^d$ and $\|V\|_p\leq R_1\}\leq R_1$. Thus the generator $F_t(Y,M)= \int^t_0f(s,Y,M)ds$ satisfies all the conditions of Proposition\ref{proposition6.4}.
\end{proof}

\par
Let $W$ and $N$ be the same Brown motion and Poisson random measure as in Section \ref{section5} respectively, we consider the associated $BSDE$ as follows:
\begin{proposition}\label{proposition6.6}
The $BSDE$
\begin{equation}\label{6.2}
Y_t=\xi+ \int^T_tf(s,Z^M,U^M)ds+M_T-M_t
\end{equation}
has a solution $(Y,M)\in S^2\times M^2_0$ for a terminal condition $\xi \in L^2(\mathcal{F}_T)^d$ if the driver $f: [0,T]\times \Omega\times H^2\times L^2(\widetilde{N})\rightarrow R^d$ satisfies the following three conditions:
\begin{enumerate}[(i)]
\item For all $(Z,U)\in H^2\times L^2(\widetilde{N}), f(\cdot,Z,U)$ is progressively measurable with $\|\int^T_0|f(t,0,0)|dt\|_2<+\infty$.
\item $\|\int^T_0|f(t,Z,U)-f(t,Z',U')|dt\|_2\leq \sqrt{\|Z-Z'\|^2_2~+~\|U-U'\|^2_2}$ for all \\ $(Z,U),(Z',U')\in H^2\times L^2(\widetilde{N})$.
\item There exists some positive number $R_1$ such that $\|\xi\|_2+\sup\{\|\int^T_0 f(t,Z^{M^V}, \\
    U^{M^V})dt\|_2: V\in L^2(\mathcal{F}_T)^d$ and $\|V\|_2\leq R_1\}\leq R_1$.
\end{enumerate}
\end{proposition}

\begin{proof}
For any $M\in M^2_0$, the decomposition $M_t=\int^t_0 Z^M_s\cdot dW_s+ \int^t_0\int_E U^M_s(x) \\
\widetilde{N}(ds,dx)~+~K^M_t$ satisfies the isometry: $E(|M_t|^2)+\int^t_0E|Z^M_s|^2ds+ \int^t_0\int_E E[|U^M_s \\
(x)|^2]\mu(dx)ds+ E[|K^M_t|^2]$, so\\
 $\|\int^T_0|f(t,Z^M,U^M)|dt\|_2 \\
 \leq \|\int^T_0|f(t,0,0)|dt\|_2+ \|\int^T_0|f(t,Z^M,U^M)-f(t,0,0)|dt\|_2 \\
 \leq \|\int^T_0|f(t,0,0)|dt\|_2+ \sqrt{\int^T_0 E[|Z^M_s|^2]ds+ \int^T_0\int_E E[|U^M_s(x)|^2]\mu(dx)ds}\\
 \leq \|\int^T_0|f(t,0,0)|dt\|_2+ \sqrt{E[|M_T|^2]}<+\infty$, \\
 which shows that the generator $F:M^2_0\rightarrow S^2_0$, defined by $F_t(M)=\int^t_0f(s,Z^M, \\
 U^M)dt$, is well defined. Since $F$ does not depend on $Y, F$ satisfies condition $(S)$.
\par
Further, $\|G(V)-G(V')\|_2\\
=\|\int^T_0[f(t,Z^{M^V},U^{M^V})-f(t,Z^{M^{V'}},U^{M^V})]dt\|_2\\
 \leq \sqrt{\|Z^{M^V}-Z^{M^{V'}}\|^2_2+\|U^{M^V}-U^{M^{V'}}\|_2^2}\\
 = \sqrt{\|M^V_T-M^{V'}_T\|^2_2}\\
 = \sqrt{\|E_0(V-V')-(V-V')\|^2_2}\\ \leq \|V-V'\|_2$
 for all $V,V'\in L^2(\mathcal{F}_T)^d$, \\
 so $BSDE(6.2)$ is of nonexpansive type, which means that $BSDE(6.2)$ has a solution $(Y,M)\in S^2\times M^2_0$ by Corollary\ref{corollary6.2}.
\end{proof}

\par
In the following, we will give the conditional versions of Theorem\ref{theorem6.1}, Corollary\ref{corollary6.2}, Proposition\ref{proposition6.4}, Corollary\ref{corollary6.5} and Proposition\ref{proposition6.6}, respectively. \vspace{2mm} \\
\textbf{Theorem6.1$'$}\label{theorem6.1'}
\emph{$BSE(\ref{1.3})$ has a solution $(Y,M)\in S^p_{\mathcal{F}_0}\times M^p_{\mathcal{F}_0,0}$ if it is of nonexpansive type and there exists an a.s. bounded $\mathcal{T}_{\varepsilon,\lambda}$--closed $L^0$--convex subset $H$ of $L^p_{\mathcal{F}_0}(\mathcal{F}_T,R^d)$ such that the associated operator $G: L^p_{\mathcal{F}_0}(\mathcal{F}_T,R^d)\rightarrow L^p_{\mathcal{F}_0}(\mathcal{F}_T,R^d)$, given by $G(V)=\xi +F_T(Y^V,M^V)$ for any $V\in L^p_{\mathcal{F}_0}(\mathcal{F}_T,R^d)$, maps $H$ into $H$.}

\begin{proof}
Since $1<p<+\infty, L^p_{\mathcal{F}_0}(\mathcal{F}_T,R^d)$ is a random uniformly convex $\mathcal{T}_{\varepsilon,\lambda}$--complete $RN$ module. It follows from Corollary\ref{corollary3.9} that $G$ has a fixed point in $H$, which further implies that the $BSE(\ref{1.3})$ has a solution $(Y,M)\in S^p_{\mathcal{F}_0}\times M^p_{\mathcal{F}_0,0}$ by Theorem\ref{theorem4.1}.
\end{proof}
\par
The following four results are merely stated without the proofs since the idea of the proofs of them is very similar to that of the proofs of Corollary\ref{corollary6.2}, Proposition\ref{proposition6.4}, Corollary\ref{corollary6.5} and Proposition\ref{proposition6.6}, respectively, which one can easily see only by replacing the $p$--norms with the conditional $p$--norms.  \vspace{2mm} \\
\textbf{Corollary6.2$'$}\label{corollary6.2'}
\emph{$BSE(\ref{1.3})$ has a solution $(Y,M)\in S^p_{\mathcal{F}_0}\times M^p_{\mathcal{F}_0,0}$ if it is of nonexpansive type and there exists $R_1\in L^0_+(\mathcal{F}_0)$ such that $|||\xi|||_p+ \bigvee \{|||F_T(Y^V,M^V)|||_p: V\in L^p_{\mathcal{F}_0}(\mathcal{F}_T,R^d)$ and $|||V|||_p\leq R_1\}\leq R_1$.} \vspace{1mm} \\
\textbf{Proposition6.3$'$}\label{proposition6.3'}
\emph{If the generator $F: S^p_{\mathcal{F}_0}\times M^p_{\mathcal{F}_0,0}\rightarrow S^p_{\mathcal{F}_0,0}$ satisfies the following two conditions:
\begin{enumerate}[(i)]
\item There exists some positive integer $k\in N$ such that $|||F^{(k)}(Y,M)-F^{(k)}(Y', \\
M')|||_p\leq \frac{1}{2C_p}|||(Y_0-M)-(Y_0'-M')|||_p$ for all $(Y,M), (Y',M')\in S^p_{\mathcal{F}_0}\times M^p_{\mathcal{F}_0,0}$; or $F$ is stable and there exists some $\mathcal{F}_0$--measurable positive integer--valued random variable $L:\Omega \rightarrow N$ such that $|||F^{(L)}(Y,M)-F^L(Y',M')|||_p \\
\leq \frac{1}{2C_p}|||(Y_0-M)-(Y_0'-M')|||_p$ for all $(Y,M), (Y',M')\in S^p_{\mathcal{F}_0}\times M^p_{\mathcal{F}_0,0}$.
\item There exists $R_1\in L^0_+(\mathcal{F}_0)$ such that $|||\xi|||_p+\bigvee\{|||F_T(Y^V,M^V)|||_p: V\in L^p_{\mathcal{F}_0}(\mathcal{F}_T,R^d)$ and $|||V|||_p\leq R_1\}\leq R_1$.
\end{enumerate}
Then $BSE(\ref{1.3})$ has a solution $(Y,M)\in S^p_{\mathcal{F}_0}\times M^p_{\mathcal{F}_0,0}$.}\vspace{2mm}\\
\textbf{Corollary $6.4'$} \label{corollary6.4'}
\emph{The $BSDE$
\begin{equation}\label{6.1'}
 Y_t=\xi +\int^T_tf(s,Y,M)ds+M_T-M_t \tag{6.1$^{\prime}$}
\end{equation}
has a solution $(Y,M)\in S^p_{\mathcal{F}_0}\times M^p_{\mathcal{F}_0,0}$ for a terminal condition $\xi \in L^p_{\mathcal{F}_0}(\mathcal{F}_T,R^d)$ if the driver $f: [0,T]\times \Omega\times S^p_{\mathcal{F}_0}\times M^p_{\mathcal{F}_0,0}\rightarrow R^d$ satisfies the following three conditions:
\begin{enumerate}[(i)]
\item For all $(Y,M)\in S^p_{\mathcal{F}_0}\times M^p_{\mathcal{F}_0,0}, f(\cdot,Y,M)$ is progressively measurable with $|||\int^T_0|f(t,0,0)|dt|||_p<+\infty(P$--a.s.$)$.
\item $|||f(t,Y,M)-f(t,Y',M')|||_p\leq \frac{1}{2TC_p}|||(Y_0-M)-(Y'_0-M')|||_p$ for all $t\in [0,T]$ and all $(Y,M), (Y',M') \in S^p_{\mathcal{F}_0}\times M^p_{\mathcal{F}_0,0}$.
\item There exists $R_1\in L^0_+(\mathcal{F}_0)$ such that $|||\xi|||_p+\bigvee \{|||\int^T_0f(s,Y^V,M^V)ds|||_p: V\in L^p_{\mathcal{F}_0}(\mathcal{F}_T,R^d)$ and $|||V|||_p\leq R_1\}\leq R_1$.
\end{enumerate}}
 \noindent \textbf{Proposition6.5$'$}\label{proposition6.6'}
\emph{The $BSDE$
\begin{equation}\label{6.2'}
  Y_t=\xi +\int^T_tf(s,Z^M,U^M)ds+M_T-M_t \tag{6.2$^{\prime}$}
\end{equation}
has a solution $(Y,M)\in S^2_{\mathcal{F}_0}\times M^2_{\mathcal{F}_0,0}$ for a terminal condition $\xi \in L^2_{\mathcal{F}_0}(\mathcal{F}_T,R^d)$ if the driver $f: [0,T]\times \Omega\times H^2_{\mathcal{F}_0}\times L^2_{\mathcal{F}_0}(\widetilde{N})\rightarrow R^d$ satisfies the following three conditions:
\begin{enumerate}[(i)]
\item For all $(Z,U)\in H^2_{\mathcal{F}_0}\times L^2_{\mathcal{F}_0}(\widetilde{N}), f(\cdot,Z,U)$ is progressively measurable with $|||\int^T_0|f(t,0,0)|dt|||_2<+\infty(P$--a.s.$)$.
\item $|||\int^T_0|f(t,Z,U)-f(t,Z',U')|dt|||_2\leq \sqrt{|||Z-Z'|||^2_2+|||U-U'|||^2_2}$ for all $(Z,U), (Z',U')\in H^2_{\mathcal{F}_0}\times L^2_{\mathcal{F}_0}(\widetilde{N})$.
\item There exists $R_1\in L^0_+(\mathcal{F}_0)$ such that $|||\xi|||_2+\bigvee \{|||\int^T_0f(t,Z^{M^V},U^{M^V})dt|||_2 \\
    : V\in L^2_{\mathcal{F}_0}(\mathcal{F}_T,R^d)$ and $|||V|||_2\leq R_1\}\leq R_1$.
\end{enumerate}}
\section{The relation between the solutions of $BSEs$ of the forms (\ref{1.1}) and (\ref{1.3})}\label{section7}
Let $(E,\|\cdot\|)$ be an $RN$ module with base $(\Omega,\mathcal{F},P)$ such that $E$ is stable. For any subset $G$ of $E$, $H_{cc}(G) := \{ \sum^\infty_{n=1} \tilde{I}_{A_n} \cdot g_n : \{ A_n, n \in N \}$ is a countable partition of $\Omega$ to $\mathcal{F}$ and $\{ g_n, n \in N \}$ is a sequence in $G \}$, is called the countable concatenation hull of $G$.
\par
In \cite{GZZ14}, it is already proved that $L^p_{\mathcal{F}_0}(\mathcal{F}_T, R^d) = H_{cc}(L^p(\mathcal{F}_T, R^d))$. In fact, $L^p_{\mathcal{F}_0}(\mathcal{F}_T, R^d)$ is a complete $RN$ module with base $(\Omega,\mathcal{F}_0,P)$, for any element $g \in L^p_{\mathcal{F}_0}(\mathcal{F}_T, R^d)$ there exist some $\xi \in L^0(\mathcal{F}_0)$ and $x \in L^p(\mathcal{F}_T, R^d)$ such that $g = \xi \cdot x$. Let $A_n = (n-1 \leq |\xi| < n)$ for each $n \in N$, then each $A_n \in \mathcal{F}_0$ and $\sum^\infty_{n=1} A_n = \Omega$, since $\xi = \sum^\infty_{n=1} \tilde{I}_{A_n} \cdot \xi$, again let $x_n = \tilde{I}_{A_n} \cdot \xi \cdot x$, then $g = \sum^\infty_{n=1} \tilde{I}_{A_n} \cdot x_n \in H_{cc}(L^p(\mathcal{F}_T, R^d))$ since $x_n \in L^p(\mathcal{F}_T, R^d)$ for each $n \in N$. Namely, $L^p_{\mathcal{F}_0}(\mathcal{F}_T, R^d) \subset H_{cc}(L^p(\mathcal{F}_T, R^d))$, the reverse inclusion is obvious since $L^p(\mathcal{F}_T, R^d) \subset L^p_{\mathcal{F}_0}(\mathcal{F}_T, R^d)$ and $L^p_{\mathcal{F}_0}(\mathcal{F}_T, R^d)$ is stable. Similarly, one has $S^p_{\mathcal{F}_0} = H_{cc}(S^p)$, $M^p_{\mathcal{F}_0} = H_{cc}(M^p)$ and $H^2_{\mathcal{F}_0} = H_{cc}(H^2)$.
\par
For the sake of convenience, let us first introduce the following:
\begin{definition} \label{definition7.1}
$BSE$(\ref{1.3}) is said to be stable if its generator $F : S^p_{\mathcal{F}_0} \times M^p_{\mathcal{F}_0,0} \to S^p_{\mathcal{F}_0,0}$ is stable. $BSE$(\ref{1.3}) is said to be regular if its generator $F$ maps $S^p \times M^p_0$ into $S^p_0$, namely $F(S^p \times M^p_0) \subset S^p_0$, in which case we denote by $\hat{F}$ the limitation of $F$ to $S^p \times M^p_0$.
\end{definition}

\begin{proposition}\label{proposition7.2}
Suppose that $BSE$(\ref{1.3}) is both stable and regular and that a sequence $\{ \xi_n, n \in N \}$ of $L^p(\mathcal{F}_T, R^d)$ is such that $BSE$(\ref{1.1}) has a solution $(Y^{(n)},M^{(n)}) \in S^p \times M^p_0$ for the generator $\hat{F}$ and the terminal condition $\xi_n$. Then, for any countable partition $\{ A_n, n \in N \}$ of $\Omega$ to $\mathcal{F}_0$, $(Y,M) := (\sum^\infty_{n=1} \tilde{I}_{A_n} \cdot Y^{(n)}, \sum^\infty_{n=1} \tilde{I}_{A_n} \cdot M^{(n)})$ is a solution of $BSE$(\ref{1.3}) for the terminal condition $\xi := \sum^\infty_{n=1} \tilde{I}_{A_n} \cdot \xi_n$.
\end{proposition}
\begin{proof}
Since $Y^{(n)}_t + \hat{F}_t(Y^{(n)},M^{(n)}) + M^{(n)}_t = \xi_n + \hat{F}_T(Y^{(n)},M^{(n)}) + M^{(n)}_T$ for each $t \in [0, T]$ and each $n \in N$, then $Y^{(n)}_t + F_t(Y^{(n)},M^{(n)}) + M^{(n)}_t = \xi_n + F_T(Y^{(n)},M^{(n)}) + M^{(n)}_T$ since $\hat{F}$ is the limitation of $F$ to $S^p \times M^p_0$. Further, $(\sum^\infty_{n=1} \tilde{I}_{A_n} \cdot Y^{(n)})_t + \sum^\infty_{n=1} \tilde{I}_{A_n} \cdot F_t(Y^{(n)},M^{(n)}) + (\sum^\infty_{n=1} \tilde{I}_{A_n} \cdot M^{(n)})_t =  \sum^\infty_{n=1} \tilde{I}_{A_n} \cdot \xi_n + \sum^\infty_{n=1} \tilde{I}_{A_n} \cdot F_T(Y^{(n)},M^{(n)}) + (\sum^\infty_{n=1} \tilde{I}_{A_n} \cdot M^{(n)})_T$, then $Y_t + F_t(Y,M) + M_t = \xi + F_T(Y,M) + M_T$ by stability of $F$.
\end{proof}
\begin{proposition}\label{proposition7.3}
Suppose that $BSE$(\ref{1.3}) is both stable and regular and further suppose that $BSE$(\ref{1.3}) has a unique solution in $S^p_{\mathcal{F}_0} \times M^p_{\mathcal{F}_0,0}$ for each terminal condition $\xi \in L^p_{\mathcal{F}_0}(\mathcal{F}_T, R^d)$ and $BSE$(\ref{1.1}) also has a unique solution in $S^p \times M^p_0$ for the generator $\hat{F}$ and each terminal condition $\xi \in L^p(\mathcal{F}_T, R^d)$. Then, for each terminal condition $\xi$ in $L^p_{\mathcal{F}_0}(\mathcal{F}_T, R^d)$ (for example, let us write $\xi = \sum^\infty_{n=1} \tilde{I}_{A_n} \cdot \xi_n$ for some countable partition $\{ A_n, n \in N \}$ of $\Omega$ to $\mathcal{F}_0$ and some sequence $\{ \xi_n, n \in N \}$ in $L^p(\mathcal{F}_T, R^d)$, and further denote by $(Y^{(n)},M^{(n)}) \in S^p \times M^p_0$ the unique solution of $BSE$(\ref{1.1}) for $\hat{F}$ and $\xi_n$),$(Y,M) := (\sum^\infty_{n=1} \tilde{I}_{A_n} \cdot Y^{(n)}, \sum^\infty_{n=1} \tilde{I}_{A_n} \cdot M^{(n)})$ is the unique solution of $BSE$(\ref{1.3}) for the terminal condition $\xi$.
\end{proposition}
\begin{proof}
One only need to notice $M^p_{\mathcal{F}_0,0} = H_{cc}(M^p_0)$, the remaining part of Proof is similar to the proof of Proposition \ref{proposition7.2}, so is omitted.
\end{proof}
\par
Although Proposition \ref{proposition7.2} and \ref{proposition7.3} are simple, their meaning is obvious: namely the problem to find a solution for a stable and regular $BSE$(\ref{1.3}) can reduce, to some extent, to one to find a solution for $BSE$(\ref{1.1}). Before we illustrate that the classical $BSDEs$ as studied in \cite{PP90,Peng04} are just the case which Proposition \ref{proposition7.3} can be applied to , let us first give the following two remarks.
\begin{remark}\label{remark7.4}
Proposition 3.3 of \cite{CN17} requires the driver $f$ to satisfy $\int^T_0 \| f(t,0,0) \\
 \|_pdt < +\infty$, which is used to guarantee that the generator $F_t(Y,M) = \int^t_0 f(s,Y, \\
 M)ds$ is well defined. In fact, as in the proof of Corollary \ref{corollary6.5}, the weaker condition that $\| \int^T_0 |f(t,0,0)|dt\|_p < +\infty$ is enough to guarantee that $F$ is well defined. Similarly, the corresponding conditions of Proposition 3.7, Corollary 3.9 and Proposition 3.10 of \cite{CN17} can also be relaxed as above.
\end{remark}
\begin{remark}\label{remark7.5}
In the studies of $BSE$(\ref{1.1}) and $BSE$(\ref{1.3}), the initial time can be also chosen as an intermediate time $t_0 \in (0,T)$ rather than only the time 0. Correspondingly, we consider the filtration $\mathbb{F} = (\mathcal{F}_t)_{t \in [t_0,T]}$ and $S$ to be the set of equivalence classes of $RCLL$ $(\mathcal{F}_t)_{t \in [t_0,T]}$--adapted processes with time parameter set $[t_0,T]$, the spaces employed in the studies of $BSEs$ are changed to the following: \\
$S^p := \{ Y \in S : \|Y\|_p = \|\sup_{t_0 \leq t \leq T} |Y_t| \|_p < +\infty \}$\\
$S^p_{t_0} := \{ Y \in S^p : Y_{t_0} =0 \}$ \\
$M^p := \{ M \in S^p : M$ is a martingale \} \\
$M^p_{t_0} := \{ M \in M^p : M_{t_0} =0 \}$ \\
$H^2 := \{ Z : Z$ is an equivalence class of an $(\mathcal{F}_t)_{t \in [t_0,T]}$--predictable $R^{d\times n}$--valued process such that $\|Z\|_2 = E [\int^T_{t_0} |Z_t|^2 dt]^{1/2} < +\infty \}$.
\par
The conditional analogues to the spaces as above are the following: \\
$S^p_{\mathcal{F}_{t_0}} := \{ Y \in S : |||Y|||_p = E [(\sup_{t_0 \leq t \leq T} |Y_t|)^p~|~\mathcal{F}_{t_0}]^{1/p} < +\infty$(a.s.)\} for $1 \leq p < +\infty$.\\
$S^\infty_{\mathcal{F}_{t_0}} := \{ Y \in S : |||Y|||_\infty = \bigwedge\{ \xi \in \bar{L}^0_{++}(\mathcal{F}_{t_0}) : \sup_{t_0 \leq t \leq T} |Y_t| \leq \xi \} \in L^0_+(\mathcal{F}_{t_0}) \}$. \\
$S^p_{\mathcal{F}_{t_0},0} := \{ Y \in S^p_{\mathcal{F}_{t_0}} : Y_{t_0} = 0 \}$ for $1 \leq p \leq +\infty$. \\
$M^p_{\mathcal{F}_{t_0}} := \{ M \in S^p_{\mathcal{F}_{t_0}} : M$ is a generalized $(\mathcal{F}_t)_{t \in [t_0,T]}$--martingale\} for $1 \leq p\leq +\infty$. \\
$M^p_{\mathcal{F}_{t_0},0} := \{ M \in M^p_{\mathcal{F}_{t_0}} : M_{t_0} =0 \}$ for $1 \leq p \leq +\infty$. \\
$H^2_{\mathcal{F}_{t_0}} := \{ Z : Z$ is an equivalence class of an $(\mathcal{F}_t)_{t \in [t_0,T]}$--predictable $R^{d\times n}$--valued process such that $|||Z|||_2 = E [\int^T_{t_0} |Z_t|^2 dt~|~ \mathcal{F}_{t_0}]^{1/2} < +\infty $(a.s.)\}.
\end{remark}
\par
Now we can return to the case of the classical $BSDEs$, whose formulation is slightly more general than the $BSDEs$ as studied in \cite{PP90,Peng04}.
\begin{example}\label{example7.6}
We consider the following $BSDE$:
\begin{equation}\label{7.1}
  Y_t = \xi + \int^T_t g(s,Y_s,Z_s)ds -\int^T_t Z_sdW_s, \forall t \in [t_0, T]
\end{equation}
where $W$ is a $n$--dimensional Brown motion on a probability space $(\Omega,\mathcal{F},P)$ with the time parameter set $[0,T]$, for each $t_0 \leq t \leq T$, $\mathcal{F}_t = \sigma \{ \sigma (B_s : t_0 \leq s \leq t) \cup \mathcal{N} \}$, where $\mathcal{N}$ stands for the family of the sets $A$ of $\mathcal{F}$ with $P(A) =0$, and the function $g : \Omega \times [t_0,T] \times R^d \times R^{d\times n} \to R^d$ satisfies the following conditions :
\begin{enumerate}[(i)]
\item For each $(y,z) \in R^d \times R^{d\times n}, g(\cdot, y, z)$ is $(\mathcal{F}_t)_{t_0\leq t \leq T}$--progressively measurable and $\|\int^T_{t_0} |g(s,0,0)| ds \|_2 < +\infty$.
\item There exists a nonnegative constant $C$ such that $|g(\omega,t,y,z)- g(\omega, t,y',z')| \\
\leq C (|y-y'| + |z-z'|)$ for all $(\omega,t) \in \Omega \times [t_0,T]$ and all $(y,z), (y',z') \in R^d \times R^{d\times n}$.
\end{enumerate}
Then the following hold:
\begin{enumerate}[(1)]
\item For each $\xi \in L^2_{\mathcal{F}_{t_0}}(\mathcal{F}_T, R^d)$, (\ref{7.1})~ has a unique solution $(Y,Z) ~:=~ (Y_t,Z_t \\
    )_{t_0 \leq t  \leq T} \in S^2_{\mathcal{F}_{t_0}} \times H^2_{\mathcal{F}_{t_0}}$.
\item For each $\xi \in L^2(\mathcal{F}_T, R^d)$, (\ref{7.1}) has a unique solution $(Y,Z) := (Y_t,Z_t)_{t_0 \leq t \leq T} \\
    \in S^2 \times H^2$.
\item Let $\xi$ and $(Y,Z)$ be the same as in (1), denote $Y_{t_0}$ by $\mathcal{E}^g[\xi~|~\mathcal{F}_{t_0}]$. Then, when $E[\int^T_{t_0} | g(\omega,t,0,0)|^2 dt] <+\infty$, we have the following estimate:
    \begin{equation}\label{7.2}
      |\mathcal{E}^g[\xi_1~|~\mathcal{F}_{t_0}]- \mathcal{E}^g[\xi_2~|~\mathcal{F}_{t_0}]| \leq C_1 E[|\xi_1 - \xi_2|^2~|~\mathcal{F}_{t_0}]^{1/2},~ \forall~ \xi_1,\xi_2 \in L^2_{\mathcal{F}_{t_0}}(\mathcal{F}_T, R^d),
    \end{equation}
where $C_1 := e^{8(1+C^2)(T-t_0)}$.
\item Let $\mathcal{E}^g[\cdot~|~\mathcal{F}_{t_0}] : L^2_{\mathcal{F}_{t_0}}(\mathcal{F}_T, R^d) \to L^0(\mathcal{F}_{t_0})$ be the same as in (3), then $\mathcal{E}^g[\cdot~|~\mathcal{F}_{t_0}]$ is stable, namely for any countable partition $\{ A_n, n \in N \}$ of $\Omega$ to $\mathcal{F}_{t_0}$ and sequence $\{ \xi_n, n \in N \}$ in $L^2_{\mathcal{F}_{t_0}}(\mathcal{F}_T, R^d)$, one has
    \begin{equation}\label{7.3}
      \mathcal{E}^g[\sum^\infty_{n=1} \tilde{I}_{A_n} \cdot \xi_n~|~\mathcal{F}_{t_0}]=\sum^\infty_{n=1} \tilde{I}_{A_n} \cdot \mathcal{E}^g[\xi_n~|~\mathcal{F}_{t_0}].
    \end{equation}
\end{enumerate}
\end{example}
\begin{proof}
(1). By (i), for each $(y,z) \in R^d \times R^{d \times n}$, $g(\cdot, \cdot, y,z) : \Omega \times [t_0,T] \to R^d$ is $(\mathcal{F}_t)_{t_0\leq t \leq T}$--progressively measurable, and by (ii), for each $(\omega,t) \in \Omega \times [t_0,T]$, $g(\omega,t,\cdot,\cdot) : R^d \times R^{d \times n} \to R^d$ is continuous. Since for each $(Y,Z) \in S^2_{\mathcal{F}_{t_0}} \times H^2_{\mathcal{F}_{t_0}}$, both $Y(\cdot,\cdot) : \Omega \times [t_0,T] \to R^d$ and $Z(\cdot,\cdot) : \Omega \times [t_0,T] \to R^{d\times n}$ are $(\mathcal{F}_t)_{t_0\leq t \leq T}$--progressively measurable, then the process $g(\cdot,\cdot,Y(\cdot,\cdot),Z(\cdot,\cdot))$ is, of course,$(\mathcal{F}_t)_{t_0\leq t \leq T}$--progressively measurable. Further, since  $\|\int^T_{t_0} |g(t,0,0)| dt \|_2 \\
< +\infty$, it is, of course, that  $|||\int^T_{t_0} |g(t,0,0)| dt |||_2 := E[\int^T_{t_0} |g(t,0,0)| dt~|~\mathcal{F}_{t_0}]^{1/2} <+\infty$(a.s.). Again let $M_t = \int^t_{t_0} (-Z_s) dW_s$ for any $t \in [t_0,T]$ and $Z \in  H^2_{\mathcal{F}_{t_0}}$, then $M \in M^2_{\mathcal{F}_{t_0},0}$ and it and $Z$ uniquely determine each other. If we denote $Z$ by $Z^M$, then $BSDE$(\ref{7.1}) becomes
\begin{equation}\label{7.4}
  Y_t = \xi + \int^T_t g(s,Y_s,Z^M_s)ds + M_T -M_t ~for ~any~ t \in [t_0,T]
\end{equation}
\par
If we consider the driver $f : \Omega \times [t_0,T] \times  L^2_{\mathcal{F}_{t_0}}(\mathcal{F}_T, R^d) \times  L^2_{\mathcal{F}_{t_0}}(\mathcal{F}_T, R^{d \times n}) \to R^d$, defined by $f (\omega,t,r,\eta) = g(\omega,t,r(\omega),\eta(\omega))$ for all $(\omega,t,r,\eta) \in \Omega \times [t_0,T] \times  L^2_{\mathcal{F}_{t_0}}(\mathcal{F}_T, R^d) \times  L^2_{\mathcal{F}_{t_0}}(\mathcal{F}_T, R^{d \times n})$, then $BSDE$(\ref{7.4}) becomes a special case of Corollary \ref{corollary5.9} where $f$ is independent of $U^M$. Thus $BSDE$(\ref{7.4}) has a unique solution  $(Y,M) \in S^2_{\mathcal{F}_{t_0}} \times M^2_{\mathcal{F}_{t_0},0}$, namely $BSDE$(\ref{7.1}) has a unique solution  $(Y,Z) \in S^2_{\mathcal{F}_{t_0}} \times H^2_{\mathcal{F}_{t_0}}$, at which time $Y$ is obviously continuous.\\
(2). Similar to the proof of (1) above, one can see that when $\xi \in L^2(\mathcal{F}_T, R^d)$, $BSDE$(\ref{7.1}) can be regarded as a special case of Corollary 3.9 of \cite{CN17}.\\
(3). It is similar to the proof of Theorem 3.2 of \cite{Peng04}.\\
(4). Since the generator corresponding to $BSDE$(\ref{7.1}) is $F : S^2_{\mathcal{F}_{t_0}} \times M^2_{\mathcal{F}_{t_0},0} \to S^2_{\mathcal{F}_{t_0},0}$ given by $F_t(Y,M) = \int^t_{t_0} g(s,Y_s,Z^M_s)ds$ when we consider the terminal condition $\xi \in L^2_{\mathcal{F}_{t_0}}(\mathcal{F}_T, R^{d})$, where $Z^M$ is the same as in the proof of (1), then it is obvious that the $BSE: Y_t + F_t(Y,M) + M_t = \xi + F_T(Y,M) + M_T$, $\forall t \in [t_0,T]$, which corresponds to $BSDE$(\ref{7.1}), is stable and regular. By (1), $BSDE$(\ref{7.1}) has a unique solution $(Y^{(n)},Z^{(n)}) \in S^2_{\mathcal{F}_{t_0}} \times H^2_{\mathcal{F}_{t_0}}$ for each $\xi_n \in L^2_{\mathcal{F}_{t_0}}(\mathcal{F}_T, R^{d})$, and $BSDE$(\ref{7.1}) has a unique solution $(Y,Z)$ for $\xi := \sum^\infty_{n=1} \tilde{I}_{A_n} \cdot \xi_n$. Further, it is very easy to see that $(\sum^\infty_{n=1} \tilde{I}_{A_n} \cdot Y^{(n)}, \sum^\infty_{n=1} \tilde{I}_{A_n} \cdot Z^{(n)})$ is also the unique solution to $BSDE$(\ref{7.1}) for the terminal condition $\xi$, so $Y_t = \sum^\infty_{n=1} \tilde{I}_{A_n} \cdot Y_t^{(n)}$ for each $t \in [t_0,T]$, in particular, $Y_{t_0} = \sum^\infty_{n=1} \tilde{I}_{A_n} \cdot Y_{t_0}^{(n)}$, namely $\mathcal{E}^g[\xi~|~\mathcal{F}_{t_0}] =\sum^\infty_{n=1} \tilde{I}_{A_n} \cdot \mathcal{E}^g[\xi_n~|~\mathcal{F}_{t_0}].$
\end{proof}
\par
To link Example \ref{example7.6} with the work of \cite{GZZ14}, we give the following final remark:
\begin{remark}\label{remark7.7}
In Example \ref{example7.6}, let $d=1$ and $g$ be independent of $y$ such that $g(\omega,t,\cdot) : R^n \to R$ is a convex function and $g(\omega,t,0) = 0$ for each $(\omega,t) \in \Omega \times [t_0,T]$. Define $\rho^g_{t_0}(\cdot) : L^2(\mathcal{F}_T,R^1) \to L^2(\mathcal{F}_{t_0},R^1)$ by $\rho^g_{t_0}(\xi) = \mathcal{E}^g[-\xi~|~\mathcal{F}_{t_0}]$ for each $\xi \in L^2(\mathcal{F}_T,R^1)$, where $\mathcal{E}^g$ is defined in (3) of Example \ref{example7.6}, then $\rho_{t_0}$ is a conditional convex risk measure and satisfies the conditional Lipschitzian property : $|\rho^g_{t_0}(\xi_1) - \rho^g_{t_0}(\xi_2)| \leq C_1 \cdot E[|\xi_1 -\xi_2|^2~|~ \mathcal{F}_{t_0}]^{1/2}$ for all $\xi_1$ and $\xi_2 \in  L^2(\mathcal{F}_T,R^1)$. Since $L^2(\mathcal{F}_T,R^1)$, as a subset of $L^2_{\mathcal{F}_{t_0}}(\mathcal{F}_T,R^1)$, is dense in the $RN$ module $(L^2_{\mathcal{F}_{t_0}}(\mathcal{F}_T,R^1),|||\cdot|||_2)$ with the $(\varepsilon,\lambda)$--topology, it is just the conditional Lipschitzian property that is used in \cite{GZZ14} to obtain the existence of the unique extension $\bar{\rho}^g_{t_0}$ of $\rho^g_{t_0}$ onto $L^2_{\mathcal{F}_{t_0}}(\mathcal{F}_T,R^1)$. In fact, $\bar{\rho}^g_{t_0}(\xi) = \mathcal{E}^g[-\xi~|~\mathcal{F}_{t_0}]$ for each $\xi \in L^2_{\mathcal{F}_{t_0}}(\mathcal{F}_T,R^1)$ since both $\bar{\rho}^g_{t_0}$ and $\mathcal{E}^g$ are stable and $\rho^g_{t_0}(\xi) = \mathcal{E}^g[-\xi~|~\mathcal{F}_{t_0}]$ for each $\xi \in L^2(\mathcal{F}_T,R^1)$. Thus the study of $BSE$(\ref{1.3}) can make us directly obtain $\bar{\rho}^g_{t_0}$ !
\end{remark}



\begin{thebibliography}{99}

\bibitem{BK00}
C.Bender,Kohlmann, BSDEs with stochastic Lipschitz condition, Universit\"{a}t Konstanz, Fakult\"{a}t f\"{u}r Mathematik und Informatik, 2000.

\bibitem{B72}
A.T.Bharucha-Reid, Random Integral Equations, Academic Press, New York and London, 1972.

\bibitem{B76}
A.T.Bharucha-Reid, Fixed point theorems in probabilistic analysis, Bull Amer.Math.Soc. 82(1976) 641--657.

\bibitem{B65}
F.E.Browder, Nonexpansive nonlinear operators in a Banach space, Proc.Nat.Acad.Sci.USA 54(1965) 1041--1044.

\bibitem{BLP09}
R.Buckdahn,J.Li,S.Peng, Mean-field backward stochastic differential equations and related partial differential equations, Stoch.Proc.Appl. 119(2009) 3133--3154.



\bibitem{Chang94}
X.Q.Chang, Some Sawyer type inequalities for martingales, Studia Math. 3(2)(1994) 187--194.

\bibitem{CE02}
Z.Chen,L.Epstein, Ambiguity, risk and asset return in continuous time, Econometrica 70(2002) 1403--1443.

\bibitem{CN17}
P.Cheridito,K.Nam, BSE's, BSDE's and Fixed--Point Problems, Ann.Prob. 45(6A)(2017) 3795--3828.


\bibitem{DI10a}
L.Delong,P.Imkeller, Backward stochastic differential equations with time--delayed generators―-results and counterexamples, Ann.Appl.Probab. 20(2010) 1512--1536.

\bibitem{DI10b}
L.Delong,P.Imkeller, On Malliavin's differentiability of BSDEs with time--delayed
generators driven by Brownian motions and Poisson random measures, Stoch.Proc.Appl. 120(2010) 1748--1775.


\bibitem{DKKS13}
S.Drapeau,M.Karliczek,M.Kupper,M.StreckfuB, Brouwer fixed point theorem in
$(L^0 )^ d$, Fixed Point Theory Appl. (2013) 301.


\bibitem{DS58}
N.Dunford, J.T.Schwartz, Linear operators, Interscience, London, 1958.

\bibitem{EH97}
N.El Karoui,S.J.Huang, A general result of existence and
uniqueness of backward stochastic differential equations, In: El Karoui N, Mazliak L. (eds.), Backward Stochastic Differential Equations, Addison Wesley Longman, 1997.



\bibitem{FKV09}
D.Filipovi\'{c},M.Kupper,N.Volgelpoth, Separation and duality in locally $L^0$--convex
modules, J.Funct.Anal. 256(2009) 3996--4029.

\bibitem{GR84}
K.Goebel,S.Reich, Uniform convexity, hyperbolic geometry, and nonexpansive mappings, Marcel Dekker, New York, and Basel, 1984.

\bibitem{G65}
D.G\"{o}hde, Zum Prinzip der kontraktiven Abbildung, Math.Nachr. 30(1965) 251--258.

\bibitem{Guo89}
T.X.Guo, The theory of probabilistic metric spaces with applications to random functional analysis, Master's thesis, Xi'an Jiaotong University, Xi'an, China, 1989.

\bibitem{Guo92}
T.X.Guo, Random metric theory and its applications, Ph.D thesis, Xi'an Jiaotong University, Xi'an, China, 1992.

\bibitem{Guo93}
T.X.Guo, A new approach to random functional analysis, in proceedings of the first China postdoctral academic conference, The China National Defense and Industry Press, Beijing, 1993, pp.1150--1154.

\bibitem{Guo95}
T.X.Guo, Extension theorems of continuous random linear operators on random domains, J.Math.Anal.Appl. 193(1995) 15--27.

\bibitem{Guo96a}
T.X.Guo, The Radon--Nikod\'{y}m property of conjugate spaces and the $w^*$--equivalence theorem for $w^*$--measurable functions, Sci.China Math.Ser.A 39(1996) 1034--1041.

\bibitem{Guo96b}
T.X.Guo, Module homomorphisms on random normed modules, Chinese Northeast Math.J. 12(1996) 102--114.

\bibitem{Guo97}
T.X.Guo, A characterization for a complete random normed module to random reflexive, J.Xiamen Univ.Natur.Sci. 36(1997) 499--502.

\bibitem{Guo99}
T.X.Guo, Some basic theories of random normed liear spaces and random inner product spaces, Acta Anal.Funct. Appl. 1(2)(1999) 160--184.

\bibitem{Guo07}
T.X.Guo, Several applications of the theory of random conjugate spaces to measurability problems, Sci.China Math.Ser.A 50(2007) 737--747.

\bibitem{Guo08}
T.X.Guo, The relation of Banach--Alaoglu theorem and Banach--Bourbaki--Kakutani--\v{S}mulian theorem in complete random normed modules to stratification structure, Sci.China Math.Ser.A 51(9)(2008) 1651--1663.

\bibitem{Guo10}
T.X.Guo, Relations between some basic results derived from two kinds of topologies for a random locally convex module, J.Funct.Anal. 258(2010) 3024--3047.

\bibitem{Guo13}
T.X.Guo, On some basic theorems of continuous module homomorphisms between random normed modules, J.Funct.Space Appl. 2013, Article ID 989102, 13 pages.

\bibitem{GL05}
T.X.Guo,S.B.Lin, The James theorem in complete random normed modules, J.Math.Anal.Appl. 308(2005) 257--265.

\bibitem{GY96}
T.X.Guo,Z.Y.You, The Riesz representation theorem on complete random inner product modules and its applications, Chinese Ann.Math.Ser.A 17(1996) 361--364.

\bibitem{GZ10}
T.X.Guo,X.L.Zeng, Random strict convexity and random uniform convexity in random normed modules, Nonlinear Anal. 73(2010) 1239--1263.

\bibitem{GZ12}
T.X.Guo,X.L.Zeng, An $L^0(\mathcal{F},R)$--valued function's intermediate value theorem and its applications to random uniform convexity, Acta Math.Sin. 28(5)(2012) 909--924.

\bibitem{GZWW17}
Guo T X, Zhang E X, Wang Y C, Wu M Z. $L^0$--convex compactness and its applications, 2017, arXiv: 1709.07137V3.

\bibitem{GZWYYZ17}
T.X.Guo,E.X.Zhang,M.Z.Wu,B.X.Yang,G.Yuan,X.L.Zeng, On random convex analysis, J.Nonlinear Conv.Anal. 18(11)(2017) 1967--1996.

\bibitem{GZZ14}
T.X.Guo,S.E.Zhao,X.L.Zeng, The relations among the three kinds of conditional risk meaures, Sci.China Math. 57(8)(2014) 1753--1764.

\bibitem{GZZ15a}
T.X.Guo,S.E.Zhao,X.L.Zeng, Random convex analysis (I): separation and Fenchel--Moreau duality in random locally convex modules (in Chinese), Sci.Sin.Math. 45(12)(2015) 1960--1980 (see also arXiv: 1503.08695V3).

\bibitem{GZZ15b}
T.X.Guo,S.E.Zhao,X.L.Zeng, Random convex analysis(II): continuity and subdifferentiability in $L^0$--pre--barreled random locally convex modules (in Chinese), Sci.Sin.Math. 45(5)(2015) 647--662 (see also arXiv: 1503.08637V2).


\bibitem{HR87}
L.P.Hansen,S.F.Richard, The role of conditioning information in deducing testable
restrictions implied by dynamic asset pricing models, Econometrica 55(3)(1987) 587--613.

\bibitem{HLR91}
R.Haydon,M.Levy,Y.Raynaud, Randomly Normed Spaces, Hermann, Paris, 1991.

\bibitem{HWY92}
S.W.He,J.G.Wang,J.A.Yan, Semimartingale Theory and Stochastic Calculus, CRC Press, 1992.

\bibitem{IW89}
N.Ikeda,S.Watanabe,  Stochastic Differential Equations and Diffusion Processes,
2nd ed, North-Holland, Amsterdam, 1989.



\bibitem{J79}
J.Jacod,  Calcul Stochastique et Problèmes de Martingales, Lectures Notes in Mathematics
714, Springer, Heidelberg, 1979.

\bibitem{J64}
R.C.James, Weakly compact sets, Trans.Amer.Math.Soc. 113(1964) 129--140.



\bibitem{K65}
W.A.Kirk, A fixed point theorem for mappings which do not increase distances, Amer.Math.Monthly 72(1965) 1004--1006.

\bibitem{K83}
W.A.Kirk, Fixed point theory for nonexpansive mappings I, II, Lect Notes Math, vol. 886, Springer--Verlag, Berlin and New York, 1981, pp. 484--505; Cont.Math. 18(1983): 121--140.

\bibitem{Kob00}
M.Kobylanski, Backward stochastic differential equations and partial differential equations with quadratic growth, Ann.Probab. 28(2000) 558--602.

\bibitem{Kra64}
M.A.Krasnoselskii, Topological Methods in the Theory of Nonlinear Integral Equations, Macmillan, New York, 1964.


\bibitem{LLQ11}
G.Liang,T.Lyons,Z.Qian, Backward stochastic dynamics on a filtered probability
space, Ann.Probab. 39(2011) 1422--1448.


\bibitem{MT03}
M.Mania,R.Tevzadze, A semimartingale backward equation and the variance optimal martingale measure under general information flow, SIAM J.Control Optim. 42(2002) 1703--1726.

\bibitem{Mau81}
R.D.Mauldin, The Scottish book : Mathematics from the Scottish caf\'{e}, Boston (Birkhauser), 1981.

\bibitem{PP90}
Pardoux \'{E},S.G.Peng, Adapted solution of a backward stochastic differential equation, Systems Control Lett.  14(1990) 55--61.

\bibitem{P99}
S.Peng,  Open problems on backward stochastic differential equations, In Control of Distributed Parameter and Stochastic Systems, IFIP Advances in Information and Communication Technology 13(1999): 265--273, Kluwer Academic, Boston, MA.

\bibitem{Peng04}
S.Peng, Nonlinear expectation, nonlinear evaluations and risk measures, In : Back etal, Lecture Notes in Math 1856, Frittelli M. and Runggaldier (Eds.), pp.165--253, 2004, Springer--Verlag Berlin Heidelberg 2004.

\bibitem{PY09}
S.Peng,Z.Yang, Anticipated backward stochastic differential equations, Ann.Probab. 37(2009) 877--902.


\bibitem{SS83}
B.Schweizer,ASklar, Probabilistic Metric Spaces. Elsevier/North Holland, New York,1983; Dover Publications, New York, 2005.


\bibitem{SG18}
Z.Y.Sun,J.Y.Guo, Optimal mean--variance investment and reinsurance problem for an insuer with stochastic volatility, Math.Meth.Oper.Res. 88(2018) 59--79.


\bibitem{TL94}
S.Tang,X.Li, Necessary conditions for optimal control of stochastic systems with
random jumps, SIAM J.Control Optim. 32(1994) 1447--1475.


\bibitem{YZG91}
Z.Y.You,L.H.Zhu,T.X.Guo, Random conjugate spaces for a class of quasinormed linear spaces, J.Xi'an Jiaotong Univ. 3(1991) 133--134 (Abstract in Chinese).

\bibitem{Z13}
X.L.Zeng, Various expressions for modulus of random convexity, Acta Math.Sin. 29(2)(2013) 263--280.

\bibitem{Z10}
G.\v{Z}itkovi\'{c}, Convex compactness and its applications, Math.Financ.Econ. 3(1)(2010) 1--12.
\end{thebibliography}

\end{document}